\pdfoutput=1
\documentclass[11pt, a4paper, twoside,leqno]{amsart}
\usepackage[centering, totalwidth = 380pt, totalheight = 600pt]{geometry}
\usepackage{amssymb, amsmath, amsthm, listings, xspace}
\usepackage{enumerate, microtype, stmaryrd, url} 
\usepackage[latin1]{inputenc}
\usepackage[arrow, matrix, tips, curve, graph, rotate]{xy}
 \usepackage{color}
\definecolor{darkgreen}{rgb}{0,0.45,0}
\usepackage[colorlinks,citecolor=darkgreen,linkcolor=darkgreen]{hyperref}
\SelectTips{cm}{10}

 \DeclareMathOperator{\ob}{ob}

\DeclareMathOperator{\co}{c} 

\newcommand{\cat}[1]{\mathbf{#1}}
\newcommand{\op}{\mathrm{op}}
\renewcommand{\co}{\mathrm{co}}
\newcommand{\id}{\mathrm{id}}
\newcommand{\thg}{{\mathord{\text{--}}}}

\renewcommand{\c}{,\,}

\DeclareMathAlphabet      {\mathbf}{OT1}{cmr}{b}{n}

\newcommand{\Rl}[1][\J]{{#1}^\pitchfork}
\newcommand{\Ll}[1][\K]{{}^\pitchfork{#1}}
\newcommand{\Drl}[1][\J]{{#1}^{\pitchfork\mskip-9.2mu\pitchfork}}
\newcommand{\Dll}[1][\K]{{}^{\pitchfork\mskip-9.2mu\pitchfork}{#1}}

\newcommand{\spn}[1]{{\langle{#1}\rangle}}

\newcommand{\dcat}[1]{\cat{\mathbb #1}}
\newcommand{\cd}[2][]{\vcenter{\hbox{\xymatrix#1{#2}}}}

\makeatletter
\def\matrixobject@{%
   \edef \next@{={\DirectionfromtheDirection@ }}%
   \expandafter \toks@ \next@ \plainxy@
   \let\xy@@ix@=\xyq@@toksix@
   \xyFN@ \OBJECT@}
\let\xy@entry@@norm=\entry@@norm
\def\entry@@norm@patched{%
   \let\object@=\matrixobject@
   \xy@entry@@norm }
\AtBeginDocument{\let\entry@@norm\entry@@norm@patched}
\makeatother

\renewcommand{\phi}{\varphi}
\newcommand{\A}{{\mathcal A}}
\newcommand{\B}{{\mathcal B}}
\newcommand{\C}{{\mathcal C}}
\newcommand{\D}{{\mathcal D}}
\newcommand{\E}{{\mathcal E}}

\newcommand{\I}{{\mathcal I}}
\newcommand{\J}{{\mathcal J}}
\newcommand{\K}{{\mathcal K}}
\renewcommand{\L}{{\mathcal L}}
\newcommand{\M}{{\mathcal M}}

\renewcommand{\P}{{\mathcal P}}

\newcommand{\R}{{\mathcal R}}
\newcommand{\Ss}{{\mathcal S}}

\newcommand{\V}{{\mathcal V}}

\newcommand{\xtor}[1]{\cdl[@1]{{} \ar[r]|-{\object@{|}}^{#1} & {}}}

\makeatletter

\newcommand{\setmanuallabel}[1]{\stepcounter{equation}{\edef\@currentlabel{\theequation}\label{#1}}}
\newcommand{\printmanuallabel}[1]{\stepcounter{equation}\text{(\theequation)}}

\def\hookleftarrowfill@{\arrowfill@\leftarrow\relbar{\relbar\joinrel\rhook}}
\def\twoheadleftarrowfill@{\arrowfill@\twoheadleftarrow\relbar\relbar}
\def\leftbararrowfill@{\arrowdoublefill@{\leftarrow\mkern-5mu}\relbar\mapstochar\relbar\relbar}
\def\Leftbararrowfill@{\arrowdoublefill@{\Leftarrow\mkern-2mu}\Relbar\Mapstochar\Relbar\Relbar}
\def\leftringarrowfill@{\arrowdoublefill@{\leftarrow\mkern-3mu}\relbar{\mkern-3mu\circ\mkern-2mu}\relbar\relbar}
\def\lefttriarrowfill@{\arrowfill@{\mathrel\triangleleft\mkern0.5mu\joinrel\relbar}\relbar\relbar}
\def\Lefttriarrowfill@{\arrowfill@{\mathrel\triangleleft\mkern1mu\joinrel\Relbar}\Relbar\Relbar}

\def\hookrightarrowfill@{\arrowfill@{\lhook\joinrel\relbar}\relbar\rightarrow}
\def\twoheadrightarrowfill@{\arrowfill@\relbar\relbar\twoheadrightarrow}
\def\rightbararrowfill@{\arrowdoublefill@{\relbar\mkern-0.5mu}\relbar\mapstochar\relbar\rightarrow}
\def\Rightbararrowfill@{\arrowdoublefill@{\Relbar\mkern-2mu}\Relbar\Mapstochar\Relbar\Rightarrow}
\def\rightringarrowfill@{\arrowdoublefill@\relbar\relbar{\mkern-2mu\circ\mkern-3mu}\relbar{\mkern-3mu\rightarrow}}
\def\righttriarrowfill@{\arrowfill@\relbar\relbar{\relbar\joinrel\mkern0.5mu\mathrel\triangleright}}
\def\Righttriarrowfill@{\arrowfill@\Relbar\Relbar{\Relbar\joinrel\mkern1mu\mathrel\triangleright}}

\def\leftrightarrowfill@{\arrowfill@\leftarrow\relbar\rightarrow}
\def\mapstofill@{\arrowfill@{\mapstochar\relbar}\relbar\rightarrow}

\renewcommand*\xleftarrow[2][]{\ext@arrow 20{20}0\leftarrowfill@{#1}{#2}}
\providecommand*\xLeftarrow[2][]{\ext@arrow 60{22}0{\Leftarrowfill@}{#1}{#2}}
\providecommand*\xhookleftarrow[2][]{\ext@arrow 10{20}0\hookleftarrowfill@{#1}{#2}}
\providecommand*\xtwoheadleftarrow[2][]{\ext@arrow 60{20}0\twoheadleftarrowfill@{#1}{#2}}
\providecommand*\xleftbararrow[2][]{\ext@arrow 10{22}0\leftbararrowfill@{#1}{#2}}
\providecommand*\xLeftbararrow[2][]{\ext@arrow 50{24}0\Leftbararrowfill@{#1}{#2}}
\providecommand*\xleftringarrow[2][]{\ext@arrow 10{26}0\leftringarrowfill@{#1}{#2}}
\providecommand*\xlefttriarrow[2][]{\ext@arrow 80{24}0\lefttriarrowfill@{#1}{#2}}
\providecommand*\xLefttriarrow[2][]{\ext@arrow 80{24}0\Lefttriarrowfill@{#1}{#2}}

\renewcommand*\xrightarrow[2][]{\ext@arrow 01{20}0\rightarrowfill@{#1}{#2}}
\providecommand*\xRightarrow[2][]{\ext@arrow 04{22}0{\Rightarrowfill@}{#1}{#2}}
\providecommand*\xhookrightarrow[2][]{\ext@arrow 00{20}0\hookrightarrowfill@{#1}{#2}}
\providecommand*\xtwoheadrightarrow[2][]{\ext@arrow 03{20}0\twoheadrightarrowfill@{#1}{#2}}
\providecommand*\xrightbararrow[2][]{\ext@arrow 01{22}0\rightbararrowfill@{#1}{#2}}
\providecommand*\xRightbararrow[2][]{\ext@arrow 04{24}0\Rightbararrowfill@{#1}{#2}}
\providecommand*\xrightringarrow[2][]{\ext@arrow 01{26}0\rightringarrowfill@{#1}{#2}}
\providecommand*\xrighttriarrow[2][]{\ext@arrow 07{24}0\righttriarrowfill@{#1}{#2}}
\providecommand*\xRighttriarrow[2][]{\ext@arrow 07{24}0\Righttriarrowfill@{#1}{#2}}

\providecommand*\xmapsto[2][]{\ext@arrow 01{20}0\mapstofill@{#1}{#2}}
\providecommand*\xleftrightarrow[2][]{\ext@arrow 10{22}0\leftrightarrowfill@{#1}{#2}}
\providecommand*\xLeftrightarrow[2][]{\ext@arrow 10{27}0{\Leftrightarrowfill@}{#1}{#2}}

\makeatother

\newcommand{\twocong}[2][0.5]{\ar@{}[#2] \save ?(#1)*{\cong}\restore}
\newcommand{\twoeq}[2][0.5]{\ar@{}[#2] \save ?(#1)*{=}\restore}
\newcommand{\rtwocell}[3][0.5]{\ar@{}[#2] \ar@{=>}?(#1)+/l 0.2cm/;?(#1)+/r 0.2cm/^{#3}}
\newcommand{\ltwocell}[3][0.5]{\ar@{}[#2] \ar@{=>}?(#1)+/r 0.2cm/;?(#1)+/l 0.2cm/^{#3}}
\newcommand{\ltwocello}[3][0.5]{\ar@{}[#2] \ar@{=>}?(#1)+/r 0.2cm/;?(#1)+/l 0.2cm/_{#3}}
\newcommand{\dtwocell}[3][0.5]{\ar@{}[#2] \ar@{=>}?(#1)+/u  0.2cm/;?(#1)+/d 0.2cm/^{#3}}
\newcommand{\dltwocell}[3][0.5]{\ar@{}[#2] \ar@{=>}?(#1)+/ur  0.2cm/;?(#1)+/dl 0.2cm/^{#3}}
\newcommand{\drtwocell}[3][0.5]{\ar@{}[#2] \ar@{=>}?(#1)+/ul  0.2cm/;?(#1)+/dr 0.2cm/^{#3}}
\newcommand{\dthreecell}[3][0.5]{\ar@{}[#2] \ar@3{->}?(#1)+/u  0.2cm/;?(#1)+/d 0.2cm/^{#3}}
\newcommand{\utwocell}[3][0.5]{\ar@{}[#2] \ar@{=>}?(#1)+/d 0.2cm/;?(#1)+/u 0.2cm/_{#3}}
\newcommand{\dtwocelltarg}[3][0.5]{\ar@{}#2 \ar@{=>}?(#1)+/u  0.2cm/;?(#1)+/d 0.2cm/^{#3}}
\newcommand{\utwocelltarg}[3][0.5]{\ar@{}#2 \ar@{=>}?(#1)+/d  0.2cm/;?(#1)+/u 0.2cm/_{#3}}

\newcommand{\pullbackcorner}[1][dr]{\save*!/#1+1.2pc/#1:(1,-1)@^{|-}\restore}
\newcommand{\pushoutcorner}[1][dr]{\save*!/#1-1.2pc/#1:(-1,1)@^{|-}\restore}

\newdir{(}{{}*!<0em,-.14em>-\cir<.14em>{l^r}}
\newdir{ (}{{}*!/-5pt/\dir{(}}
\newdir{ >}{{}*!/-5pt/\dir{>}}
\newcommand{\sh}[2]{**{!/#1 -#2/}}

\theoremstyle{definition}

\theoremstyle{plain}
\newtheorem{Thm}{Theorem}
\newtheorem{Prop}[Thm]{Proposition}

\newtheorem{Cor}[Thm]{Corollary}
\newtheorem{Lemma}[Thm]{Lemma}
\numberwithin{equation}{section}

\theoremstyle{definition}

\newtheorem{Ex}[Thm]{Example}
\newtheorem{Exs}[Thm]{Examples}
\newtheorem{Rk}[Thm]{Remark}

\newcommand{\Lan}{\mathrm{Lan}}

\renewcommand{\l}[1]{L{#1}}
\newcommand{\fl}[1]{\alg{L}{#1}}
\renewcommand{\r}[1]{R{#1}}
\newcommand{\fr}[1]{\alg{R}{#1}}
\newcommand{\m}[1]{\mu_{#1}}
\renewcommand{\c}[1]{\Delta_{#1}}

\newcommand{\lp}[1]{L'{#1}}
\newcommand{\flp}[1]{\alg{L}'{#1}}
\newcommand{\rp}[1]{R'{#1}}
\newcommand{\frp}[1]{\alg{R}'{#1}}

\newcommand{\cp}[1]{\Delta'_{#1}}

\newcommand{\Coalg}[1]{\mathsf{#1}\text-\cat{Coalg}}
\newcommand{\Alg}[1]{\mathsf{#1}\text-\cat{Alg}}
\newcommand{\Kl}[1]{\cat{Kl}(\mathsf{#1})}
\newcommand{\DCoalg}[1]{\mathsf{#1}\text-\mathbb C\cat{oalg}}
\newcommand{\DAlg}[1]{\mathsf{#1}\text-\mathbb A\cat{lg}}
\newcommand{\Sq}[1]{\mathbb S\cat{q}(#1)}
\newcommand{\aone}{{\mathbf 1}}
\newcommand{\atwo}{{\mathbf 2}}
\newcommand{\athree}{{\mathbf 3}}
\newcommand{\alg}[1]{\boldsymbol{#1}}
\newcommand{\awfs}{\textsc{awfs}\xspace}

\newcommand{\Lax}{{\cat{AWFS}_\mathrm{lax}}}
\newcommand{\Oplax}{{\cat{AWFS}_\mathrm{oplax}}}
\newcommand{\Ladj}{{\cat{AWFS}_\mathrm{ladj}}}
\newcommand{\Radj}{{\cat{AWFS}_\mathrm{radj}}}

\newcommand{\Cat}{{\cat{Cat}}}
\newcommand{\CAT}{{\cat{CAT}}}
\newcommand{\Set}{{\cat{Set}}}
\newcommand{\SET}{{\cat{SET}}}
\newcommand{\VCat}[1][\V]{{#1\text-\cat{Cat}}}
\newcommand{\VCAT}[1][\V]{{#1\text-\cat{CAT}}}
\newcommand{\DBL}{{\cat{DBL}}}

\begin{document}
\leftmargini=2em \title[Algebraic weak factorisation systems
I]{Algebraic weak factorisation\\systems I: accessible awfs}
\author{John Bourke}
\address{Department of Mathematics and Statistics, Masaryk University, Kotl\'a\v rsk\'a 2, Brno 60000, Czech Republic}
\email{bourkej@math.muni.cz} 
\author{Richard Garner} \address{Department of Mathematics, Macquarie
  University, NSW 2109, Australia} \email{richard.garner@mq.edu.au}

\date{\today}

\thanks{The first author acknowledges the support of the Grant agency
  of the Czech Republic, grant number P201/12/G028. The second author
  acknowledges the support of an Australian Research Council Discovery
  Project, grant number DP110102360.}

\begin{abstract}
  Algebraic weak factorisation systems (\awfs) refine weak
  factorisation systems by requiring that the assignations sending a
  map to its first and second factors should underlie an interacting
  comonad--monad pair on the arrow category. We provide a
  comprehensive treatment of the basic theory of \awfs---drawing on
  work of previous authors---and complete the theory with two main new
  results. The first provides a characterisation of \awfs and their
  morphisms in terms of their double categories of left or right maps.
  The second concerns a notion of \emph{cofibrant generation} of an
  \awfs by a small double category; it states that, over a locally
  presentable base, any small double category cofibrantly generates an
  \awfs, and that the \awfs so arising are precisely those with
  accessible monad and comonad. Besides the general theory, numerous
  applications of \awfs are developed, emphasising particularly those
  aspects which go beyond the non-algebraic situation.
\end{abstract}
\maketitle

\section{Introduction}
\label{sec:introduction}
 
A \emph{weak factorisation system} on a category $\C$ comprises two
classes of maps $\L$ and $\R$, each closed under retracts in the arrow
category, and obeying two axioms: firstly, that each map $f \in \C$
admit a factorisation $f = \r f \cdot \l f$ with $\l f \in \L$ and $\r
f \in \R$; and secondly, that each $r \in \R$ have the right lifting
property with respect to each $\ell \in \L$---meaning that, for every
square as in the solid part of:
\begin{equation}
\cd{
  A \ar[d]_{\ell} \ar[r] & C \ar[d]^{r} \\
  B \ar[r] \ar@{.>}[ur] & D }\label{eq:1}
\end{equation}
there should exist a commuting diagonal filler as indicated. Weak
factorisation systems play a key role in \emph{Quillen model
  structures}~\cite{Quillen1967Homotopical}, which comprise two
intertwined weak factorisation systems on a category; but they also
arise elsewhere, for example in the 
categorical semantics of intensional type
theory~\cite{Awodey2009Homotopy,Gambino2008The-identity}.

\emph{Algebraic} weak factorisation systems were introduced
in~\cite{Grandis2006Natural}; they refine the basic notion by
requiring that the factorisation process $f \mapsto (\l f, \r f)$
yield a compatible comonad~$\mathsf L$ and monad~$\mathsf
R$ on the arrow category of $\C$. Given $(\mathsf L, \mathsf R)$, we
re-find $\L$ and $\R$ as the retract-closures of the classes of maps
admitting $\mathsf L$-coalgebra or $\mathsf R$-algebra structure, so
that one may define an algebraic weak factorisation system (henceforth
\awfs) purely in terms of a comonad--monad pair $(\mathsf L, \mathsf
R)$ satisfying suitable axioms; we recall these in
Section~\ref{sec:revision} below.

As shown in~\cite{Garner2009Understanding}, any cofibrantly generated
weak factorisation system on a well-behaved category may be realised
as an \awfs, so that the algebraic notions are entirely appropriate
for doing homotopy theory; this point of view has been pushed by
Riehl, who in~\cite{Riehl2011Algebraic,Riehl2013Monoidal} gives
definitions of algebraic model category and algebraic monoidal model
category, and in subsequent collaboration has used these notions to
obtain non-trivial homotopical
results~\cite{Barthel2013On-the-construction,Blumberg2014Homotopical,Ching2014Coalgebraic}.

Yet \awfs can do more than just serve as well-behaved realisations of
their underlying weak factorisation systems; by making serious use of
the monad $\mathsf R$ and comonad $\mathsf L$, we may capture
phenomena which are invisible in the non-algebraic setting. For
example, each \awfs on $\C$ induces a \emph{cofibrant replacement
  comonad} on $\C$ by factorising the unique maps out of $0$; and if we
choose our \awfs carefully, then the Kleisli category of this comonad
$\mathsf Q$---whose maps $A \rightsquigarrow B$ are maps $QA \to B$ in
the original category---will equip $\C$ with a usable notion of
\emph{weak map}. For instance, there is an \awfs on the category of
tricategories~\cite{Gordon1995Coherence} and strict morphisms
(preserving all structure on the nose) for which $\Kl{Q}$ comprises
the tricategories and their trihomomorphisms (preserving all structure
up to coherent equivalence); this example and others were described
in~\cite{Garner2010Homomorphisms}, and will be revisited in
the companion paper~\cite{Bourke2014AWFS2}.

Probably the most important expressive advantage of \awfs is that
their left and right classes can delineate kinds of map which mere
weak factorisation systems cannot---the reason being that we interpret
the classes of an \awfs as being composed of the
$\mathsf L$-coalgebras and $\mathsf R$-algebras, rather than the
underlying $\L$-maps and $\R$-maps. For example, here are some classes
of map in $\Cat$ which are not the $\R$-maps of any weak factorisation
system, but which---as shown in Examples~\ref{ex:5} below---may be
described in terms of the possession of $\mathsf R$-algebra structure
for a suitable \awfs:
\begin{itemize}
\item The Grothendieck fibrations; 
\item The Grothendieck fibrations whose fibres are groupoids;
\item The Grothendieck fibrations whose fibres have
  finite limits, and whose reindexing functors preserve them;
\item The left adjoint left inverse functors.
\end{itemize}
At a crude level, the reason that these kinds of map cannot be
expressed as classes of $\R$-maps is that they are not retract-closed.
The deeper explanation is that being an $\R$-map is a mere
\emph{property}, while being an $\mathsf R$-algebra is a
\emph{structure} involving choices of basic lifting operations---and
this choice allows for necessary equational axioms to be imposed
between derived operations. As a further demonstration of the power
this affords, we mention the result of~\cite{Hebert2011Weak} that for
any monad $\mathsf T$ on a category $\C$ with finite products, there
is an \awfs on $\C$ whose ``algebraically fibrant objects''---$\mathsf
R$-algebras $X \to 1$---are precisely the $\mathsf T$-algebras.

The existence of \awfs that reach beyond the scope of the
non-algebraic theory opens up intriguing vistas. In a projected sequel
to this paper, we will consider the theory of enrichment over monoidal
\awfs~\cite{Riehl2013Monoidal} and use it to develop an abstract
``homotopy coherent enriched category theory''. In fact, we touch on
this already in the current paper;
Section~\ref{sec:enrich-small-object} describes the \emph{enriched
  small object argument} and sketches some of its applications to
two-dimensional category theory, and to notions from the theory of
quasicategories~\cite{Joyal2002Quasi-categories,Joyal2008The-theory,Lurie2009Higher}
such as \emph{limits}, \emph{Grothendieck fibrations} and \emph{Kan
  extensions}.

The role of these examples, and others like them, will be to
illuminate and justify the main contribution of this paper---that of
giving a comprehensive account of the theory of unenriched \awfs.
Parts of this theory can be found developed across the
papers~\cite{Barthel2013On-the-construction,Garner2009Understanding,Grandis2006Natural,Riehl2011Algebraic,Riehl2013Monoidal};
our objective is to draw the most important of these results together,
and to complete them with two new theorems that clarify and simplify
both the theory and the practice of \awfs.

In order to explain our two main theorems, we must first recall some
\emph{double-categorical} aspects of \awfs. A double category
$\mathbb A$ is an internal category in $\CAT$, as on the left below;
we refer to objects and arrows of $\A_0$ as \emph{objects} and
\emph{horizontal arrows}, and to objects and arrow of $\A_1$ as
\emph{vertical arrows} and \emph{squares}. Internal functors and
internal natural transformations between internal categories in $\CAT$
will be called \emph{double functors} and \emph{horizontal natural
  transformations}; they comprise a $2$-category $\DBL$.
\[
\cd[@C+0.2em]{
  \A_1\!\times_{\A_0}\! \A_1 \ar[r]^-{\circ}
  &
  \A_1 \ar@<-5pt>[r]_-{\mathrm{cod}}\ar@<5pt>[r]^-{\mathrm{dom}} &
  \ar[l]|-{\mathrm{id}}
  \A_0 } \qquad \quad
\cd[@C+0.2em]{
  \Alg{R}\!\times_{\C}\! \Alg{R} \ar[r]^-{\circ} 
  &
  \Alg{R} \ar@<-5pt>[r]_-{\mathrm{cod}}\ar@<5pt>[r]^-{\mathrm{dom}} &
  \ar[l]|-{\mathrm{id}} \C\rlap{ .} } 
\]

To each \awfs $(\mathsf L, \mathsf R)$ on a category $\C$ we may
associate a double category $\DAlg{R}$, as on the right above, whose
objects and horizontal arrows are the objects and arrows of $\C$,
whose vertical arrows are the $\mathsf R$-algebras, and whose squares
are maps of $\mathsf R$-algebras. The functor $\circ \colon \Alg{R}
\times_\C \Alg{R} \to \Alg{R}$ encodes a canonical \emph{composition
  law} on $\mathsf R$-algebras---recalled in
Section~\ref{sec:double-categ-algebr} below---which is an algebraic
analogue of the fact that $\R$-maps in a weak factorisation system are
closed under composition.

There is a forgetful double functor $\DAlg{R} \to \Sq{\C}$ into the
double category of \emph{squares} in $\C$---wherein objects are those
of $\C$, vertical and horizontal arrows are arrows of $\C$, and
squares are commuting squares in $\C$.
In~\cite[Lemma~6.9]{Riehl2011Algebraic}, it was shown that the
assignation sending an \awfs $(\mathsf L, \mathsf R)$ on $\C$ to the
double functor $\DAlg{R} \to \Sq{\C}$ constitutes the action on
objects of a $2$-fully faithful $2$-functor
\begin{equation}
  \DAlg{(\thg)} \colon \Lax \to \DBL^\atwo
\label{eq:36}
\end{equation}
from the $2$-category of \awfs, lax \awfs morphisms and \awfs
$2$-cells---whose definition we recall in
Section~\ref{sec:change-base} below---to the arrow $2$-category
$\DBL^\atwo$.

Our first main result, Theorem~\ref{thm:recognition} below, gives an
elementary characterisation of the essential image of~\eqref{eq:36}
and of its dual, the fully faithful coalgebra $2$-functor
$\DCoalg{(\thg)} \colon \Oplax \to \DBL^\atwo$. In the theory of
monads, a corresponding characterisation of the \emph{strictly
  monadic} functors---those in the essential image of the $2$-functor
sending a monad $\mathsf T$ to the forgetful functor $U^\mathsf T
\colon \Alg{T} \to \C$---is given by Beck's monadicity theorem; and so
we term our result a \emph{Beck theorem} for \awfs.
Various aspects of this result are already in the
literature---see~\cite[Appendix]{Garner2009Understanding},
\cite[Proposition 2.8]{Garner2010Homomorphisms} or \cite[Theorem
2.24]{Riehl2011Algebraic}---and during the preparation of this paper,
we became aware that Athorne had independently arrived at a similar
result as~\cite[Theorem~2.5.3]{Athorne2013Coalgebraic}. We nonetheless
provide a complete treatment here, as this theorem is crucial to a
smooth handling of \awfs, in particular allowing them to be
constructed simply by giving a double category of the correct form to
be one's double category of $\mathsf R$-algebras or $\mathsf
L$-coalgebras.

The second main result of this paper deals with the appropriate notion
of cofibrant generation for \awfs. Recall that a weak factorisation
system $(\L, \R)$ is \emph{cofibrantly generated} if there is a mere
\emph{set} of maps $J$ such that $\R$ is the precisely the class of
maps with the right lifting property against each $j \in J$.
Cofibrantly generated weak factorisation systems are commonplace due
to Quillen's \emph{small object argument}~\cite[\S
II.3.2]{Quillen1967Homotopical}, which ensures that for any set of
maps $J$ in a locally presentable category~\cite{Gabriel1971Lokal},
the \awfs $(\L, \R)$ cofibrantly generated by $J$ exists.

In~\cite[Definition~3.9]{Garner2009Understanding} is described a
notion of cofibrant generation for \awfs: given a small category
$U \colon \J \to \C^\atwo$ over $\C^\atwo$, the \awfs cofibrantly
generated by $\J$, if it exists, has as $\mathsf R$-algebras, maps $g$
equipped with choices of right lifting against each map $Uj$,
naturally with respect to maps of $\J$. This notion is already more
permissive than the usual one---as witnessed
by~\cite[Example~13.4.5]{Riehl2014Categorical}, for example---but we
argue that it is still insufficient, since it excludes important
examples on $\awfs$, such as the ones on $\Cat$ listed above, whose
$\mathsf R$-algebras are not retract-closed.

To rectify this, we introduce in Section~\ref{sec:cofibr-fibr-gener-1}
the notion of an \awfs being cofibrantly generated by a small
\emph{double category} $U \colon \dcat J \to \Sq{\C}$. This condition
specifies the $\mathsf R$-algebras as being maps equipped with
liftings against $Uj$ for each vertical map $j \in \dcat{J}$,
naturally with respect to squares of $\dcat J$, but now with the extra
requirement that, for each pair of composable vertical maps
$j \colon x \to y$ and $k \colon y \to z$ of $\dcat J$, the specified
lifting against $kj$ should be obtained by taking specified lifts
first against $k$ and then against $j$.

This broader notion of cofibrant generation is permissive enough to
capture all our leading examples, including the ones on $\Cat$ listed
above. Our second main theorem justifies it at a theoretical level, by
characterising the cofibrantly generated \awfs on locally presentable
categories as being exactly the \emph{accessible} \awfs---those
$(\mathsf L, \mathsf R)$ whose comonad $\mathsf L$ and monad $\mathsf
R$ preserve $\kappa$-filtered colimits for some regular
cardinal~$\kappa$. More precisely, we show for a locally presentable
$\C$ that every small $\dcat J \to \Sq{\C}$ cofibrantly generates
an accessible \awfs on $\C$; and that every accessible \awfs on $\C$
is cofibrantly generated by a small $\dcat J \to \Sq{\C}$.

We now give a summary of the contents of the paper. We begin in
Section~\ref{sec:revision} with a revision of the basic theory of
\awfs: the definition, the relation with weak factorisation systems,
double categories of algebras and coalgebras, morphisms of \awfs, and
the fully faithful $2$-functors $\DAlg{(\thg)}$ and $\DCoalg{(\thg)}$.
In Section~\ref{sec:beck}, we give our first main result, the Beck
theorem for~\awfs described above, together with two useful variants.
Section~\ref{sec:double-categories-at} then uses the Beck theorem to
give constructions of a wide range of \awfs; in particular, we discuss
the \awfs for \emph{split epis} in any category with binary
coproducts; the \awfs for \emph{lalis} (left adjoint left inverse
functors) in any $2$-category with oplax limits of arrows; and the
construction of \emph{injective} and \emph{projective} liftings of
\awfs.

Section~\ref{sec:cofibrant-generation-1} revisits the notion of
cofibrant generation of \awfs by small categories, as introduced
in~\cite{Garner2009Understanding}, and uses the Beck theorem to give a
simplified proof that such \awfs always exist in a locally presentable
$\C$. This prepares the way for our second main result; in
Section~\ref{sec:cofibr-gener-double} we introduce cofibrant
generation by small double categories, and in
Section~\ref{sec:char-access-awfs}, show that over a locally
presentable base $\C$, such \awfs always exist and are precisely the
accessible \awfs on $\C$.

Finally in Section~\ref{sec:enrich-small-object} we say a few words
about \emph{enriched cofibrant generation}. As mentioned above, a
sequel to this paper will deal with this in greater detail; here, we
content ourselves with giving the basic construction and a range of
applications. In particular, we will see how to express notions such
as \emph{Grothendieck fibrations}, \emph{categories with limits}, and
\emph{Kan extensions} in terms of algebras for suitable \awfs, and
explain how to extend these constructions to the quasicategorical
context.

\section{Revision of algebraic weak factorisation systems}
\label{sec:revision}
Algebraic weak factorisation systems were introduced
in~\cite{Grandis2006Natural}---there called \emph{natural} weak
factorisation systems---and their theory developed further
in~\cite{Athorne2012The-coalgebraic,Barthel2013On-the-construction,Garner2009Understanding,Riehl2011Algebraic,Riehl2013Monoidal}.
They are highly structured objects, and an undisciplined approach runs
the risk of foundering in a morass of calculations. The above papers,
taken together, show that a smoother presentation is possible; in this
introductory section, we draw together the parts of this presentation
so as to give a concise account of the basic aspects of the theory.

Before beginning, let us state our foundational assumptions. $\kappa$
will be a Grothendieck universe, and sets in $\kappa$ will be called
\emph{small}, while general sets will be called \emph{large}.
$\cat{Set}$ and $\cat{SET}$ are the categories of small and large
sets; $\cat{Cat}$ and $\cat{CAT}$ are the $2$-categories of small
categories (ones internal to $\cat{Set}$) and of locally small
categories (ones enriched in $\cat{Set}$). Throughout the paper, all
categories will be assumed to be locally small and all $2$-categories
will be assumed to be locally small (=$\Cat$-enriched) except for ones
whose names, like $\SET$ or $\CAT$, are in capital letters.

\subsection{Functorial factorisations}
\label{sec:functorial-fact}
By a \emph{functorial factorisation} on a category $\C$, we mean a
functor $F \colon \C^\atwo \to \C^\mathbf 3$ from the category of
arrows to that of composable pairs which is a section of the
composition functor $\C^\mathbf 3 \to \C^\atwo$. We write $F = (\l{},
E, \r{})$, to indicate that the value of $F$ at an object $f$ or
morphism $(h,k) \colon f \to g$ of $\C^\atwo$ is given as on the left
and the right of:
\[
X \xrightarrow{\l f} Ef \xrightarrow{\r f} Y
\qquad \quad \qquad
\cd{
 X
  \ar[r]^{\l f} \ar[d]_{h} & Ef \ar[d]|{E(h, k)} \ar[r]^{\r f} & Y
  \ar[d]^{k} \\
 W \ar[r]^{\l g} & Eg \ar[r]^{\r g} & Z\rlap{ .}}
\]

Associated to $(\l{}, E, \r{})$ are the functors $L, R \colon \C^\atwo
\to \C^\atwo$ with actions on objects $f \mapsto \l f$ and $f \mapsto
\r f$, and the natural transformations $\epsilon \colon L \Rightarrow
1$ and $\eta \colon 1 \Rightarrow R$ with components at $f$ given by
the commuting squares:
\begin{equation}
\cd{
A \ar[r]^1 \ar[d]_{\l f} & A \ar[d]^{f} \\
Ef \ar[r]^{\r f} & B
} \qquad \text{and} \qquad
\cd{
A \ar[r]^{\l f} \ar[d]_{f} & Ef \ar[d]^{\r f} \\
B \ar[r]^{1} & B\rlap{ .}
}\label{eq:3}
\end{equation}
Note that $(\l{}, E, \r{})$ is completely determined by either $(L,
\epsilon)$ or $(R, \eta)$.

\subsection{Algebraic weak factorisation systems}
\label{sec:algebr-weak-fact}
An \emph{algebraic weak factorisation system} (\awfs) on $\C$ comprises:
\begin{enumerate}[(i)]
\item A functorial factorisation $(\l{}, E, \r{})$ on $\C$;
\item An extension of $(L, \epsilon)$ to  a comonad $\mathsf L =
  (L,\epsilon,\Delta)$;
\item An extension of
$(R,\eta)$ to a monad $\mathsf R = (R,\eta,\mu)$;
\item A distributive law~\cite{Beck1969Distributive} of the comonad
  $\mathsf L$ over the monad $\mathsf R$, whose underlying 
  transformation $\delta \colon LR \Rightarrow RL$ satisfies
  $\mathrm{dom} \cdot \delta = \mathrm{cod} \cdot \Delta$ and
  $\mathrm{cod} \cdot \delta = \mathrm{dom} \cdot \mu$.
\end{enumerate}

This definition involves less data than may be immediately apparent.
The counit and unit axioms $\epsilon L \cdot \Delta = 1$ and $\mu
\cdot \eta R = 1$ force the components of $\Delta \colon L \to LL$ and
$\mu \colon RR \to R$ to be identities in their respective
domain and codomain components, so given by commuting squares as
on the left and right of:
\begin{equation}
  \label{eq:2}
 \cd{
A \ar[d]_{\l f} \ar[r]^1 & A \ar[d]^{\l{\l f}} \\
Ef \ar[r]^{\c f} & E\l f}
\qquad 
\cd{
Ef \ar[d]_{\l{\r f}} \ar[r]^{\c f} & E\l f
\ar[d]^{\r{\l f}} \\
E\r f \ar[r]^{\m f} & Ef
}
\qquad 
\cd{
E \r f \ar[r]^{\m f} \ar[d]_{\r{\r f}} & Ef \ar[d]^{\r f} \\
B \ar[r]^1 & B\rlap{ ;}}
\end{equation}
while the conditions on $\delta$ in (iv) force its component at $f$
to be given by the central square. Thus the only additional data
beyond the underlying functorial factorisation (i) are maps $\c f
\colon Ef \to E\l f$ and $\m f \colon E\r f \to Ef$, satisfying
suitable axioms. Note that the comonad $\mathsf L$ and monad $\mathsf
R$ between them contain all of the data, so that we may unambiguously
denote an \awfs simply by $(\mathsf L, \mathsf R)$.

By examining~\eqref{eq:3} and~\eqref{eq:2}, we see that the comonad
$\mathsf L$ of an \awfs is a comonad \emph{over}
the domain functor $\mathrm{dom} \colon \C^{\atwo} \to \C$. By this
we mean that $\mathrm{dom} \cdot L = \mathrm{dom}$ and that
$\mathrm{dom} \cdot \epsilon$ and $\mathrm{dom} \cdot \Delta$ are identity
natural transformations. Dually, the monad $\mathsf R$ is a monad
over the codomain functor.

\subsection{Coalgebras and algebras}
To an algebraic weak factorisation system $(\mathsf L, \mathsf R)$, we
may associate the categories $\Coalg{L}$ and $\Alg{R}$ of coalgebras
and algebras for $\mathsf L$ and $\mathsf R$; these are to be thought
of as constituting the respective left and right classes of the \awfs.
Because $\mathsf L$ is a comonad over the domain functor, a coalgebra
structure $f \to Lf$ on $f \colon A \to B$ necessarily has its domain
component an identity, and so is determined by a single map
$s \colon B \to Ef$; we write such a coalgebra as
$\alg f = (f, s) \colon A \to B$. Dually, an $\mathsf R$-algebra
structure on $g \colon C \to D$ is determined by a single map
$p \colon Ef \to C$, and will be denoted
$\alg g = (g,p) \colon C \to D$. If the base category $\C$ has an
initial object $0$, then we may speak of \emph{algebraically cofibrant
  objects}, meaning $\mathsf L$-coalgebras $0 \to X$; dually, if $\C$
has a terminal object~$1$, then an \emph{algebraically fibrant object}
is an $\mathsf R$-algebra $X \to 1$.


\subsection{Lifting of coalgebras against algebras}\label{sec:lift}
Given an $\mathsf L$-coalgebra $\alg f = (f,s)$, an $\mathsf
R$-algebra $\alg g = (g,p)$ and a commuting square as in the centre
of:
\begin{equation}\label{eq:filler}
\cd{
A' \ar[d]_{f'} \ar[r]^a & A \ar[d]^{f} \\
B' \ar[r]_{b} & B} \qquad \qquad
\cd{
A \ar[d]_{f} \ar[r]^{u}  & C \ar[d]^{g} \\
B \ar@{.>}[ur] \ar[r]_{v}  & D} \qquad \qquad
\cd{
C \ar[d]_{g} \ar[r]^c & C' \ar[d]^{g'} \\
D \ar[r]_{d} & D'}
\end{equation}
there is a canonical diagonal filler $\Phi_{\alg f, \alg g}(u,v) \colon B
\to C$ given by the composite $p \cdot E(u,v) \cdot s \colon B \to Ef
\to Eg \to C$. 
These fillers are natural with respect to morphisms of
$\mathsf L$-coalgebras and $\mathsf R$-algebras; which is to say that
if we have commuting squares as on the left and right above
which underlie, respectively, an $\mathsf L$-coalgebra
morphism $\alg f' \to \alg f$ and an $\mathsf R$-algebra morphism
$\alg g \to \alg g'$, then composing~\eqref{eq:filler} with these two
squares preserves the canonical filler:
\[
c \cdot \Phi_{\alg f, \alg g}(u,v) \cdot b = 
\Phi_{\alg f', \alg g'}(cua,dvb)\rlap{ .}
\]

Writing $U \colon \Coalg{L} \to \C^{\atwo}$ and
$V \colon \Alg{R} \to \C^{\atwo}$ for the forgetful
functors, this naturality may be expressed by saying that the
canonical liftings constitute the components of a natural
transformation
\begin{equation}\label{eq:4}
  \Phi \colon \C^{\atwo}(U\thg,V?) \Rightarrow
  \C(\mathrm{cod}\ U\thg, \mathrm{dom}\ 
  V?) \colon \Coalg{L}^\op \times \Alg{R} \to \Set \ \text.
\end{equation}

\subsection{Factorisations with universal properties}
\label{sec:fact-with-univ}
The two parts of the factorisation $f = \r f \cdot
\l f$ of a map underlie the cofree $\mathsf L$-coalgebra
$\fl f = (\l f,\c f) \colon A \to Ef$ and
the free $\mathsf R$-algebra $\fr f = (\r f,\m f)
\colon Ef \to B$. The freeness of the latter says that,
for any $R$-algebra $\alg g = (g,p)$ and morphism $(h,k) \colon f \to
g$ in $\C^\atwo$, there is a unique arrow $\ell$ such that the
left-hand diagram in
\begin{equation}\label{eq:5}
\cd[@-0.5em]{A\ar[rr]^{h} \ar[dd]_{f} \ar[dr]^{\l f} & & C \ar[dd]^{g}\\
& Ef \ar[d]_{\r f}  \ar@{.>}[ur]|{\exists ! \ell} \\
B \ar[r]^{1} & B \ar[r]^{k} & D\\
} \qquad \qquad
\cd[@-0.5em]{A \ar[r]^h \ar[dd]_f & C \ar[r]^1 \ar[d]^{\l g} & C \ar[dd]^{g}\\
& Ef \ar[dr]^{\r g} \\
B \ar@{.>}[ur]|{\exists ! m} \ar[rr]^{k} & & D
}
\end{equation}
commutes and such that $(\ell,k)$ is an algebra morphism $\fr f \to
\alg g$.
The dual universal property, as on the right above, describes the
co-freeness of $\fl g$ with respect to maps out of an $\mathsf
L$-coalgebra $\alg f = (f,s)$.

\setmanuallabel{eq:9z}  \setmanuallabel{eq:9az}
\setmanuallabel{eq:10z}
\addtocounter{equation}{-3}
Observe also the following canonical liftings involving (co)free
(co)algebras:
\begin{equation*}
  \begin{array}{cccc}
  \quad\cd{A \ar[d]_{\fl g} \ar[r]^-{1}  & A \ar[d]^{(g,p)\!\!\!}
    \\ Eg \ar[r]_-{\r g} \ar@{.>}[ur]|p & B} \quad &
  \quad\cd{A \ar[d]_{(f,s)} \ar[r]^-{\l f}  & A \ar[d]^{\fr f}
    \\ B \ar[r]_-{1} \ar@{.>}[ur]|s & B} \quad &
  \quad\cd{E f \ar[d]_{\fl {\r f}} \ar[r]^-{\c f}  & E\l f \ar[d]^{\fr {\l f}}
    \\ E\r f \ar[r]_{\m f} \ar@{.>}[ur]|{\c f \cdot \m f} & Ef\rlap{ .}}\quad \\
\printmanuallabel{eq:9z}&   \printmanuallabel{eq:9az} &   \printmanuallabel{eq:10z}
\end{array}
\end{equation*}
Here, \eqref{eq:9z} implies that an $\mathsf R$-algebra is uniquely
determined by its liftings against $\mathsf L$-coalgebras, and dually
in~\eqref{eq:9az}; while~\eqref{eq:10z} expresses precisely the
non-trivial axiom of the distributive law $\delta \colon \mathsf{LR}
\Rightarrow \mathsf{RL}$.

\subsection{Underlying weak factorisation system}
\label{sec:underly-weak-fact-1}
By the preceding two sections, the classes of maps in $\C$ that admit
$\mathsf L$-coalgebra or $\mathsf R$-algebra structure satisfy all of
the axioms needed to be the two classes of a weak factorisation
system, except maybe for closure under retracts. On taking the
retract-closures of these two classes, we thus obtain a weak
factorisation system $(\L, \R)$, the \emph{underlying weak
  factorisation system} of $(\mathsf L, \mathsf R)$.

\subsection{Algebras and coalgebras from liftings}
\label{sec:algebr-coalg-from}
A \emph{coalgebra lifting operation} on a map $g \colon C \to D$ is a
function $\phi_{\thg g}$ assigning to each $\mathsf L$-coalgebra $\alg
f$ and each commuting square $(u,v) \colon f \to g$ as
in~\eqref{eq:filler} a diagonal filler $\phi_{\alg f,g}(u,v) \colon B
\to C$, naturally in maps of $\mathsf L$-coalgebras. Maps equipped
with coalgebra lifting operations are the objects of a category
$\Rl[\Coalg{L}]$ over $\C^\atwo$---the nomenclature will be explained
in Section~\ref{sec:lifting-operations} below---whose maps $(g,
\phi_{\thg g}) \to (g', \phi_{\thg g'})$ are maps $(c,d) \colon g \to
g'$ of $\C^\atwo$ for which $c \cdot \phi_{\alg f,g}(u,v) = \phi_{\alg
  f,g'}(cu,dv)$. Any $\mathsf R$-algebra $\alg g$ induces the lifting
operation $\Phi_{\thg \alg g}$ on its underlying map, and
by~\eqref{eq:4} any algebra map respects these liftings; so we have a
functor $\bar \Phi \colon \Alg{R} \to \Rl[\Coalg{L}]$ over
$\C^\atwo$.

\begin{Lemma}
  \label{lem:2}
  $\bar \Phi \colon \Alg{R} \to \Rl[\Coalg{L}]$ is
  injective on objects and fully faithful, and has in its image just
  those $(g, \phi_{\thg g})$ such that $\phi_{\fl f, g}(u,v) \cdot
  \mu_f = \phi_{\fl{\r f}, g}(\phi_{\fl f, g}(u,v),v \cdot \mu_f)$ for
  all maps $(u,v) \colon \l f \to g$:
  \begin{equation}
\cd[@-0.6em@C+1em]{
  A \ar[d]_{\fl f} \ar[r]^1 & A \ar[dd]^(0.55){\fl f} \ar[r]^u & C \ar[dd]^g \\
  Ef \ar[d]_{\fl {\r f}} \\
  E\r f \ar[r]_{\m f} \ar@{.>}[uurr]
\ar@<1.1ex>@{}[uurr]|(0.4)*=0[@]{\scriptstyle{\phi(u,v) \cdot \mu_f}}
& Ef \ar[r]_v \ar@{.>}[uur]
\ar@<-1.1ex>@{}[uur]|(0.5)*=0[@]{\scriptstyle{\phi(u,v)}}
& D }
\qquad = \qquad \cd[@-0.6em@C+1.2em]{
  A \ar[d]_{\fl f} \ar[r]^u & C \ar[dd]^g \\
  Ef \ar[d]_{\fl {\r f}} \ar@{.>}[ur]  \ar@<1.1ex>@{}[ur]|(0.45)*=0[@]{\scriptstyle{\phi(u,v)}}\\
  E\r f \ar[r]_-{v \cdot \m f} \ar@{.>}[uur]
  \ar@<-1.1ex>@{}[uur]|-*=0[@]{\scriptstyle{\phi(\phi(u,v),v \mu_f)}} &
  D\rlap{ .} }\label{eq:37}
\end{equation}
\end{Lemma}
\begin{proof}
  We observed in Section~\ref{sec:fact-with-univ} that an $\mathsf
  R$-algebra is determined by its liftings against coalgebras, whence
  $\bar \Phi$ is injective on objects. For full fidelity, let $\alg g
  = (g,p)$ and $\alg h = (h,q)$ be $\mathsf R$-algebras and $(c,d)
  \colon (g, \Phi_{\thg\alg g}) \to (h, \Phi_{\thg \alg h})$ a map of
  underlying lifting operations; we must show $(c,d)$ is an $\mathsf
  R$-algebra map. Consider the diagram
\[
\cd{C \ar[d]_{\fl g} \ar[r]^-{1}  & C \ar[d]^{\alg g} \ar[r]^c & D
  \ar[d]^{\alg h} 
    \\ Eg \ar[r]_-{\r g} \ar@{.>}[ur]|p & C \ar[r]_d & C'} 
  \quad = \quad 
\cd{C \ar[d]_{\fl g} \ar[r]^-{c}  & C' \ar[d]_{\fl h} \ar[r]^-1 & C'
  \ar[d]^{\alg h}
    \\ Eg \ar[r]_-{E(c,d)}  & Eh \ar[r]_{\r{h}} \ar@{.>}[ur]|{q} &
    D\rlap{ .}}
\]
By~\eqref{eq:9z} $\Phi_{\thg\alg g}$ assigns the filler $p$ to the far
left square; so as $(c,d)$ respects liftings, $\Phi_{\thg\alg h}$
assigns the filler $c \cdot p$ to the composite left rectangle.
Likewise $\Phi_{\thg\alg h}$ assigns the filler $q$ to the far right
square; so as the square to its left is a (cofree) map of $\mathsf
L$-coalgebras, $\Phi_{\thg\alg h}$ assigns the filler $q \cdot E(c,d)$
to the composite right rectangle. Thus $c \cdot p = q \cdot E(c,d)$
and $(c,d)$ is an $\mathsf R$-algebra map as required.

We next show that each $(g, \Phi_{\thg \alg g})$ in the image of $\bar
\Phi$ satisfies~\eqref{eq:37}. By universality~\eqref{eq:5} and
naturality~\eqref{eq:4}, it suffices to take $\alg g = \fr {\l f}$ and
$(u,v) = (\l{\l f}, 1)$, and now
\begin{align*}
  \Phi_{\fl f, \fr {\l f}}(\l{\l f},1) \cdot \m f &=
  \c f \cdot \m f = \Phi_{\fl {\r f}, \fr {\l f}}(\c f, \m f) \\
&=\Phi_{\fl {\r f}, \fr {\l f}}(\Phi_{\fl f, \fr {\l f}}(\l{\l f},1), \m f)
\end{align*}
by~\eqref{eq:9z} and \eqref{eq:10z}, as required. Finally, we show that
$\bar \Phi$ is surjective onto those pairs $(g, \phi_{\thg g})$
satisfying~\eqref{eq:37}. Given such a pair, we define $p = \phi_{\fl
  g, g}(1, \r g)$; now by universality~\eqref{eq:5} and
naturality~\eqref{eq:4}, the pair $(g, \phi_{\thg g})$ will be the
image under $\bar \Phi$ of $\alg g =
(g,p)$ so long as $(g,p)$ is in fact an $\mathsf R$-algebra. The unit
axiom is already inherent in $p$'s being a lifting; as for the
multiplication axiom, we have
\begin{align*}
  p \cdot \mu_g & = \phi_{\fl g, g}(1, \r g) \cdot \mu_g
  = \phi_{\fl{\r g}, g}(\phi_{\fl g, g}(1,\r g),\r g \cdot \mu_g)\\
  & = \phi_{\fl{\r g}, g}(p,\r{\r g}) = \phi_{\fl{\r g}, g}(p,\r g
  \cdot E(p,1))\\
  & = \phi_{\fl{g}, g}(1,\r g ) \cdot E(p,1) = p \cdot E(p,1)
\end{align*}
by definition of $p$, \eqref{eq:37}, and naturality of $\phi_{\thg
  g}$ with respect to
the $\mathsf L$-coalgebra morphism $(p,E(p,1)) \colon \fl {\r g} \to
\fl g$.
\end{proof}

It is not hard to see that the category $\Rl[\Coalg{L}]$ is
in fact isomorphic to the category $(R,\eta)\text-\cat{Alg}$ of
algebras for the mere pointed endofunctor $(R, \eta)$ underlying the
monad $\mathsf R$, with the functor $\bar \Phi$ corresponding under
this identification to the natural inclusion of $\Alg{R}$ into $(R,
\eta)\text-\cat{Alg}$. Of course, all these results have a dual form
characterising coalgebras in terms of their liftings against algebras.

\subsection{Double categories of algebras and coalgebras}
\label{sec:double-categ-algebr}In a weak factorisation system the left
and right classes of maps are closed under composition and contain the
identities. We now describe the analogue of this for \awfs. For binary
composition, given $\mathsf R$-algebras $\alg g \colon A \to
B$ and $\alg h \colon B \to C$ we may define a coalgebra lifting
operation $\Phi_{\thg, \alg h \cdot \alg g}$ on the composite
underlying map $h \cdot g$, whose liftings are obtained by first
lifting against $\alg h$ and then against $\alg g$, as on the left in:
\begin{equation}
  \Phi_{\alg f, \alg h \cdot \alg g}(u, v) = \Phi_{\alg f, \alg g}(u,
  \Phi_{\alg f, \alg h}(gu, v)) \qquad \qquad \Phi_{\alg f, \alg 1_A}(u,v) =
  v
  \rlap{ .}\label{eq:7}
\end{equation}
It is easy to verify that this is a coalgebra lifting operation
satisfying~\eqref{eq:37}, so that by Lemma~\ref{lem:2}, it is the
canonical lifting operation associated to a unique $\mathsf R$-algebra
$\alg h \cdot \alg g \colon A \to C$, the composite of $\alg g$ and
$\alg h$. As for nullary composition, each identity map $1_A$ bears a
\emph{unique} coalgebra lifting operation as on the right above, which
is easily seen to satisfy~\eqref{eq:37}; so each identity map bears an
$\mathsf R$-algebra structure $\alg 1_A$ which is in fact
\emph{unique}.

This composition law for algebras is associative and unital: to see
associativity, we check that $\alg k \cdot (\alg h \cdot \alg g)$ and
$(\alg k \cdot \alg h) \cdot \alg g$ have the same lifting operations,
and apply Lemma~\ref{lem:2}; unitality is similar. We may also use
Lemma~\ref{lem:2} to verify that each $(a,a) \colon \alg 1_A \to \alg
1_{A'}$ is a map of $\mathsf R$-algebras, and that if $(a,b) \colon
\alg g \rightarrow \alg g'$ and $(b,c) \colon \alg h \to \alg h'$ are
maps of $\mathsf R$-algebras, then so too is $(a,c) \colon \alg h
\cdot \alg g \to \alg h' \cdot \alg g'$.
Consequently, there is a double category $\DAlg{R}$ whose objects and
horizontal arrows are those of $\C$, and whose vertical arrows and
squares are the $\mathsf R$-algebras and the maps thereof.
There is a forgetful double functor $U^\mathsf R \colon \DAlg{R} \to \Sq{\C}$
into the double category of commutative squares in $\C$, which,
displayed as an internal functor between internal categories in
$\CAT$, is as on the left in:
\[
\cd[@C+2.2em@R-0.3em]{
  \Alg{R} \times_{\C} \Alg{R} \ar[r]^-{U^\mathsf R \times_\C U^\mathsf
    R}\ar[d]|{m} \ar@<-5pt>[d]_p \ar@<5pt>[d]^q  & \C^{\mathbf 3} \ar@<-5pt>[d]_p \ar@<5pt>[d]^q \ar[d]|m \\
  \Alg{R} \ar@<-5pt>[d]_d\ar@<5pt>[d]^c \ar[r]^{U^\mathsf R} &
  \C^\atwo
  \ar@<-5pt>[d]_d\ar@<5pt>[d]^c \\
  \ar[u]|i \C \ar[r]_1 & \C \ar[u]|i } \qquad \quad \qquad
\cd[@C+1.5em@R-0.3em]{
  \A_1\!\times_{\A_0}\! \A_1 \ar[r]^-{V_1
    \times_{V_0} V_1}\ar[d]|{m} \ar@<-5pt>[d]_p \ar@<5pt>[d]^q & \C^\mathbf 3 \ar[d]|m \ar@<-5pt>[d]_p \ar@<5pt>[d]^q\\
  \A_1 \ar@<-5pt>[d]_d\ar@<5pt>[d]^c \ar[r]^{V_1} &
  \C^\atwo
  \ar@<-5pt>[d]_d\ar@<5pt>[d]^c \\
  \ar[u]|i \A_0 \ar[r]_{V_0 = 1} & \C\rlap{ .} \ar[u]|i }\] 

Note that this double functor has object component \emph{an identity}
and arrow component a \emph{faithful functor}. By a \emph{concrete
  double category over $\C$}, we mean a double functor $V \colon
\mathbb A \to \Sq{\C}$ (as on the right above) whose $V$ has these two
properties. 
For example, the $\mathsf L$-coalgebras also constitute a concrete
double category $\DCoalg{L}$ over $\C$.

We note before continuing that the equality~\eqref{eq:37} may now be
re-expressed as saying that each square as on the
left below is one of $\DCoalg{L}$, and each square on the right is one
of $\DAlg{R}$. These squares will be important in what follows.
\begin{equation}
\cd[@-0.5em]{
A \ar[d]_{\fl f} \ar[r]^1 & A \ar[dd]^{\fl f} \\
 Ef \ar[d]_{\fl {\r f}} \\
E\r f \ar[r]_{\m f} & Ef
} \qquad \qquad \qquad
\cd[@-0.5em]{
Ef \ar[dd]_{\fr {f}} \ar[r]^{\c f} & E\l f
\ar[d]^{\fr {\l f}} \\
& Ef \ar[d]^{\fr f} \\
B \ar[r]_{1} & B\rlap{ .}
}\label{eq:13}
\end{equation}

\subsection{Morphisms of AWFS}\label{sec:change-base}
Given \awfs $(\mathsf L, \mathsf R)$ and $(\mathsf L', \mathsf R')$ on
categories $\C$ and $\D$, a \emph{lax morphism} of \awfs $(F, \alpha)
\colon (\C, \mathsf L, \mathsf R) \to (\D, \mathsf L', \mathsf R')$
comprises a functor $F \colon \C \to \D$ and a natural family of maps
$\alpha_f$ rendering commutative each square as on the left in:
\begin{equation}\label{eq:laxmorphism}
  \cd[@-1em@C-0.5em]{
    & FA \ar[dl]_{\lp {Ff}} \ar[dr]^{F\l f} \\
    E'Ff \ar[rr]^{\alpha_f} \ar[dr]_{\rp {Ff}} & & 
    FEf\rlap{ } \ar[dl]^{F\r f} \\ & FB} \qquad \qquad
  \cd[@!@-0.8em]{
    E'Ff \ar[r]^{\alpha_f} \ar[d]_{E'(\gamma_A,  \gamma_B)} & 
    FEf \ar[d]^{\gamma_{Ef}} \\
    E'Gf \ar[r]_{\beta_f} & GEf\rlap{ ,}
  }
\end{equation}
and such that the induced $(\alpha, 1) \colon R'F^\atwo \Rightarrow
F^\atwo R$ and $(1, \alpha) \colon L'F^\atwo \Rightarrow F^\atwo L$
are respectively a lax monad morphism $\mathsf R \to \mathsf R'$ and a
lax comonad morphism $\mathsf L \to \mathsf L'$ over $F^\atwo \colon
\C^\atwo \to \D^\atwo$ (i.e., a \emph{monad functor} and a
\emph{comonad opfunctor} in the terminology
of~\cite{Street1972The-formal}). A transformation $(F, \alpha)
\Rightarrow (G, \beta)$ between lax morphisms is a natural
transformation $\gamma \colon F \Rightarrow G$ rendering commutative
the square above right for each $f \colon A \to B$ in $\C$. Algebraic
weak factorisation systems, lax morphisms and transformations
constitute a $2$-category $\Lax $. Dually, an \emph{oplax morphism} of \awfs involves maps
$\alpha_f$ with the opposing orientation to~\eqref{eq:laxmorphism},
and with the induced $( \alpha, 1)$ and $(1, \alpha)$ now being
\emph{oplax} monad and comonad morphisms over $F^\atwo$. With a
similar adaptation on $2$-cells, this yields a $2$-category
$\Oplax$.

\subsection{Adjunctions of AWFS}
\label{sec:adjunctions}
Given \awfs $(\mathsf L, \mathsf R)$ and $(\mathsf L', \mathsf R')$ on
categories $\C$ and $\D$, and an adjunction $F \dashv G \colon \C \to
\D$, there is a bijection between $2$-cells $\alpha$ exhibiting $G$ as
a lax \awfs morphism and $2$-cells $\beta$ exhibiting $F$ as oplax;
this is the \emph{doctrinal adjunction} of~\cite{Kelly1974Doctrinal}.
From $\alpha$ we
determine the components of 
$\beta$ by
\[
\cd{
\beta_f = FE'f \xrightarrow{FE'(\eta_A,\eta_B)} FE'GFf
\xrightarrow{F\alpha_{Ff}}  FGEFf \xrightarrow{\epsilon_{EFf}} EFf
\rlap{ ,}
}
\]
so that $\beta$ is the \emph{mate}~\cite{Kelly1974Review} of $\alpha$
under the adjunctions $F \dashv G$ and $F^\atwo \dashv G^\atwo$. The
functoriality of this correspondence is expressed through an
identity-on-objects isomorphism of $2$-categories
$\Radj^{\co\op} \cong \Ladj$,
where $\Radj$ is defined identically to $\Lax$ except that its
$1$-cells come equipped with chosen left adjoints, and where $\Ladj$
is defined from $\Oplax$ dually.
By an \emph{adjunction of \awfs}, we mean a morphism of one of these
isomorphic $2$-categories; thus, a pair of a lax \awfs
morphism $(G, \alpha)$ and an oplax \awfs morphism $(F, \beta)$ whose
underlying functors are adjoint, and whose $2$-cell data determine
each other by mateship. In this situation, an easy calculation shows
that
\looseness=-1
\begin{equation}
\Phi_{\alg f, G \alg g}(u,v) = \Phi_{F \alg f, \alg g}(\bar u,\bar
v)\label{eq:46}
\end{equation}
for any $\mathsf L'$-coalgebra $\alg f$, any $\mathsf R$-algebra
$\alg g$, and any $(u,v) \colon f \to Gg$ in $\D^\atwo$ with adjoint
transpose $(\bar u, \bar v) \colon Ff \to g$ in $\C^\atwo$.

\subsection{Semantics $2$-functors}
\label{sec:semantics-2-functors}
A lax morphism of \awfs
$(F, \alpha) \colon (\C, \mathsf L, \mathsf R) \to (\C', \mathsf L',
\mathsf R')$
has an underlying lax monad morphism, yielding as
in~\cite[Lemma~1]{Johnstone1975Adjoint} a lifted functor as to the
left in
\begin{equation}
\cd{
\Alg{R} \ar[d]_{U^\mathsf R} \ar@{.>}[r]^{\bar F} & \Alg{R'}
\ar[d]^{U^\mathsf{R'}} \\
\C^\atwo \ar[r]_{F^\atwo} & \D^\atwo
} \qquad \qquad \qquad
\cd{
\DAlg{R} \ar[d]_{U^\mathsf R} \ar@{.>}[r]^{\bar F} & \DAlg{R'}
\ar[d]^{U^\mathsf{R'}} \\
\Sq{\C} \ar[r]_{\Sq{F}} & \Sq{\D}\rlap{ ,}
}\label{eq:17}
\end{equation}
whose action on algebras we will abusively denote by $\alg g \mapsto
F\alg g$. A short calculation from the fact that $(1, \alpha)$ is a
lax comonad morphism shows that $ F\Phi_{\fl f, \alg g}(u,v) \cdot
\alpha_f = \Phi_{\flp{Ff}, F\alg g}(Fu, Fv \cdot \alpha_f)$ for each
$\mathsf R$-algebra $\alg g$. From this and~\eqref{eq:7} it follows
that $\Phi_{\flp{Ff}, F\alg h \cdot F \alg g}(Fu, Fv \cdot
\alpha_{hg}) = \Phi_{\flp{Ff}, F(\alg h \cdot \alg g)}(Fu, Fv \cdot
\alpha_{hg})$ for all composable $\mathsf R$-algebras $\alg g$ and
$\alg h$; now taking $f = hg$ and $(u,v) = (1, \r{hg})$ and
applying~\eqref{eq:9z}, we conclude that $F\alg h \cdot F \alg g =
F(\alg h \cdot \alg g)$. Thus the lifted functor on the left
of~\eqref{eq:17} preserves algebra composition; it must also preserve
the (unique) algebra structures on identities, and so underlies a
double functor as on the right.
Moreover, each $2$-cell $\gamma \colon (F, \alpha) \to (G, \beta)$ in
$\Lax $ has an underlying lax monad transformation, so that $\gamma
\colon F \to G$ may be lifted to a transformation on categories of
algebras, which---by concreteness---lifts further to a horizontal
transformation on double categories. 

The preceding constructions have evident duals involving $\Oplax$ and
coalgebras; and in this way, we obtain the left and right
\emph{semantics $2$-functors}:
\begin{equation}
  \label{eq:18}
\DCoalg{(\thg)} \colon \Oplax 
\to \DBL^\atwo  \qquad   \qquad
\DAlg{(\thg)} \colon \Lax
\to \DBL^\atwo\rlap{ ,}
\end{equation}
where $\DBL$ denotes the $2$-category of double categories, double
functors and horizontal transformations. 
The following result is now~\cite[Lemma~6.9]{Riehl2011Algebraic}:

\begin{Prop}\label{prop:ff2semantics}
  The $2$-functors $\DAlg{(\thg)}$ and $\DCoalg{(\thg)}$ are $2$-fully faithful.
\end{Prop}
\begin{proof}
  By duality, we need only deal with the case of $\DAlg{(\thg)}$. Note
  that each double functor $\Sq{\C} \to \Sq{\D}$ must be of the form
  $\Sq{F}$ for some $F \colon \C \to \D$; thus given $(\C, \mathsf L,
  \mathsf R)$ and $(\D, \mathsf L', \mathsf R')$ in $\Lax$, a morphism
  between their images in $\DBL^\atwo$ amounts to a square as on the
  right of~\eqref{eq:17}. Each such has its underlying action on
  vertical arrows and squares as on the left in~\eqref{eq:17}, and so
  must be induced by a unique lax monad morphism $\gamma \colon
  \mathsf R \to \mathsf R'$ over $F^\atwo$. As $\mathsf R$ and
  $\mathsf R'$ are monads over the codomain functor, we must have
  $\gamma = (\alpha, 1)$ for natural maps $\alpha_f$ as
  in~\eqref{eq:laxmorphism}. It remains to show that $(1, \alpha)$ is
  a lax comonad morphism $\mathsf L \to \mathsf L'$ over $F^\atwo$.
  The compatibilities in~\eqref{eq:laxmorphism} already show that
  $(1,\alpha)$ commutes with counits; we need to establish the same
  for the comultiplications. Consider the diagrams:
\begin{gather*}
  \cd[@C-0.7em]{ E'Ff \ar[dd]^{\frp {Ff}} \ar[r]^-{\cp {Ff}} & \sh{r}{0.2em}E'\lp
    Ff \ar[d]|{\frp {\lp {Ff}}} \ar[rr]^{E'(1, \alpha_f)}& & E'F\l f
    \ar[d]|{\frp {F \l f}} \ar[r]^{\alpha_{\l f}} & FE\l f \ar[d]_{F
      \fr {\l f}}
    \\
    & E'Ff \ar[rr]_{\alpha_f} \ar[d]|{\frp {Ff}} & & FEf \ar[d]|{F\fr f}
    \ar[r]_1 & FEf
    \ar[d]_{F \fr f}\\
    FB \ar[r]_{1} &FB \ar[rr]_{1} & & FB \ar[r]_{1} & FB } \ \quad
  \cd[@C-0.25em]{ E'Ff \ar[dd]^{\frp {Ff}} \ar[r]^{\alpha_f} & FEf
    \ar[dd]|{F \fr f} \ar[r]^{F\c f} &
    FE\l f \ar[d]_{F \fr{ \l f}} \\
    & & FEf
    \ar[d]_{F \fr f} \\
    FB \ar[r]_1 & FB \ar[r]_{1} & FB\rlap{ .} }
\end{gather*}

Every region here is a square of $\DAlg{R'}$: the leftmost as it is of
the form~\eqref{eq:13}, the rightmost as it is the image under the
lifted double functor of another such square in $\DAlg{R}$, and all
the rest because $(\alpha, 1)$ is a lax monad morphism. So the
exterior squares are also squares in $\DAlg{R'}$; but as both have the
same composite $F \l{\l f}$ with the unit morphism $(\lp F f, 1)
\colon F f \to \rp F f$, they must, by freeness of $\frp{Ff}$,
coincide; now the equality of their domain-components expresses
precisely the required compatibility with comultiplications. 

This completes the proof of full fidelity on $1$-cells; on $2$-cells,
it is easy to see that in the commuting diagram of $2$-functors
\[
\cd[@C+1.5em]{
\Lax  \ar[r]^-{\DAlg{(\thg)}} \ar[d]_{(\C, \mathsf L,
  \mathsf R) \mapsto (\C^\atwo, \mathsf R)} & \DBL^\atwo \ar[d]^{(\thg)_1} \\
\cat{MND}_\mathrm{lax} \ar[r]^-{\Alg{(\thg)}} & \CAT^\atwo
}
\] 
the left and bottom edges are locally fully faithful, and that the
right edge is locally fully faithful on the concrete double categories
in the image of $\DAlg{(\thg)}$; whence the top edge is also locally
fully faithful.
\end{proof}
\subsection{Orthogonal factorisation systems} Recall that an
\emph{orthogonal factorisation system}~\cite{Freyd1972Categories} is a
weak factorisation systems $(\L, \R)$ wherein liftings~\eqref{eq:1} of
$\L$-maps against $\R$-maps are unique. The following result
characterises those \awfs arising from orthogonal factorisation
systems: it improves on Theorem~3.2 of~\cite{Grandis2006Natural} by
requiring idempotency of only \emph{one} of $\mathsf L$ or $\mathsf
R$. This resolves the open question posed in Remark~3.3(a) of
\emph{ibid}.

\begin{Prop}
  \label{prop:8}
  Let $(\mathsf L, \mathsf R)$ be an \awfs on $\C$. The following are
  equivalent:
  \begin{enumerate}[(i)]
  \item $\Coalg{L} \to \C^\atwo$ is fully faithful;
  \item $\mathsf L$ is an idempotent comonad;
  \item For each $\mathsf L$-coalgebra
  $\alg f \colon A \to B$, there is a coalgebra map
  $(f, 1) \colon \alg f \to \alg 1_B$;
  \item $\Alg{R} \to \C^\atwo$ is fully faithful;
  \item $\mathsf R$ is an idempotent monad;
  \item For each $\mathsf R$-algebra
  $\alg g \colon C \to D$, there is an algebra map
  $(1, g) \colon \alg 1_C \to \alg g$;
  \item Liftings of $\mathsf L$-coalgebras against $\mathsf
    R$-algebras are unique;
  \item The underlying weak factorisation system $(\L, \R)$ is 
    orthogonal.
  \end{enumerate}
  Under these circumstances, moreover, the \awfs $(\mathsf L, \mathsf
  R)$ is determined up to isomorphism by the underlying $(\L, \R)$.
\end{Prop}
\begin{proof}
  (i) $\Rightarrow$ (ii) $\Rightarrow$ (i) by standard properties of
  idempotent comonads, and (i) $\Rightarrow$ (iii) is trivial. We next
  prove (iii) $\Rightarrow$ (vii). So given $\alg f \in \Coalg{L}$ and
  $\alg g \in \Alg{R}$, we must show that any diagonal filler $j$ for
  a square $(u,v) \colon f \to g$, as on the left below, is equal to
  the canonical filler $\Phi_{\alg f, \alg g}(u,v)$.
\[
\cd{A \ar[d]_f \ar[r]^u & C \ar[d]^g \\
B \ar@{.>}[ur]|j \ar[r]_{v} & D} \qquad = \qquad
\cd{
A \ar[d]_f \ar[r]^f & B \ar[d]_{1} \ar[r]^j & C \ar[d]^g \\
B \ar[r]_1 & B \ar[r]_{v} & D\rlap{ .}
}
\]
So factorise $(u,v)$ as on the right above; by (iii), the left-hand
square is a coalgebra map $\alg f \to \alg 1_B$, whence by
naturality~\eqref{eq:4} we have $j = \Phi_{\alg{1}_B, \alg g}(j,v) =
\Phi_{\alg f, \alg g}(u,v)$ as required. We next prove (vii)
$\Rightarrow$ (iv); so for algebras $\alg g$ and $\alg h$, we must
show that each $(c,d) \colon g \to h$ underlies an algebra map $\alg g
\to \alg h$. By Lemma~\ref{lem:2}, it suffices to show that $(c,d)$
commutes with the coalgebra lifting functions; in other words, that
for each coalgebra $\alg f$ and each $(u,v) \colon f \to g$, we have
$c \cdot \Phi_{\alg f,\alg g}(u,v) = \Phi_{\alg f,\alg h}(cu,dv)$.
This is so by (vii) since both these maps fill the square $(cu,dv)
\colon f \to h$.

Dual arguments now prove that (iv) $\Rightarrow$ (v) $\Rightarrow$
(iv) $\Rightarrow$ (vi) $\Rightarrow$ (vii) $\Rightarrow$ (i); it
remains to show that (i)--(vii) are equivalent to (viii). Clearly
(viii) $\Rightarrow$ (vii); on the other hand, if (i) and (iv) hold,
then any retract of an $\mathsf L$-coalgebra or $\mathsf R$-algebra is
again a coalgebra or algebra---being given by the splitting of an
idempotent in $\Alg{R}$ or $\Coalg{L}$---so that any map in $\L$ or
$\R$ is the underlying map of an $\mathsf L$-coalgebra or $\mathsf
R$-algebra; whence by (vii) liftings of $\L$-maps against $\R$-maps
are unique.

Finally, if (i)--(viii) hold, then we can reconstruct $\mathsf L$ and
$\mathsf R$ up to isomorphism from $\Alg{R} \to \C^\atwo$ and
$\Coalg{L} \to \C^\atwo$, and can reconstruct these in turn from $\L$
and $\R$ as the full subcategories of $\C^\atwo$ on the $\L$-maps and
the $\R$-maps.
\end{proof}

\section{A Beck theorem for awfs}
\label{sec:beck}
In this section, we give our first main
result---Theorem~\ref{thm:recognition} below---which provides an
elementary characterisation of the concrete double categories in the
essential image of the semantics $2$-functors~\eqref{eq:18}. This will
allow us to construct an \awfs simply by exhibiting a double category
of an appropriate form to be one's double category of algebras or
coalgebras. As explained in the introduction, the essential images of
the corresponding semantics $2$-functors for monads and comonads are
characterised by Beck's (co)monadicity theorem, and so we term our
result a ``Beck theorem'' for algebraic weak factorisation systems.

\subsection{Reconstruction}
\label{sec:reconstruction}
There are two main aspects to the Beck theorem. The first is the
following reconstruction result; in the special case where $\mathsf R$
is \emph{idempotent}, this
is~\cite[Proposition~5.9]{Im1986On-classes}, while the general case is
given as~\cite[Theorem~4.15]{Barthel2013On-the-construction}; for the
sake of a self-contained presentation, we include a proof here, which
improves on that of~\cite{Barthel2013On-the-construction} only in
trivial ways.

\begin{Prop}\label{prop:reconstruct}
  Let $\mathsf R \colon \C^\atwo \to \C^\atwo$ be a monad over the
  codomain functor. The right semantics $2$-functor~\eqref{eq:18}
  induces a bijection between extensions of $\mathsf R$ to an
  algebraic weak factorisation system $(\mathsf L, \mathsf R)$ and
  extensions of the diagram
  \begin{equation}
    \cd[@C+1em]{
      \Alg{R} \ar@<-3pt>[d]_{dU^\mathsf R}\ar@<3pt>[d]^{cU^\mathsf R} \ar[r]^{U^\mathsf R} &
      \C^\atwo
      \ar@<-3pt>[d]_d\ar@<3pt>[d]^c \\
      \C \ar[r]^1 & \C  }\label{eq:14}
  \end{equation}
  to a concrete double category over $\C$.
\end{Prop}

\begin{proof}
  To give an extension of~\eqref{eq:14} to a concrete double category
  over ${\C}$ is to give a composition law $\alg h, \alg g \mapsto
  \alg h \star \alg g$ on $\mathsf R$-algebras which is associative,
  unital, and compatible with $\mathsf R$-algebra maps. For each such,
  we must exhibit a unique $(\mathsf L, \mathsf R)$ whose induced
  $\mathsf R$-algebra composition is $\star$. As in
  Section~\ref{sec:functorial-fact}, the monad $\mathsf R$ determines
  the $(L, \epsilon)$ underlying $\mathsf L$, and so we need only give
  the maps $\c f$ satisfying appropriate axioms. Consider the diagram
  on the left in:
  \begin{equation*}
    \cd[@-0.5em]{
      A\ar[rr]^{\l {\l f}} \ar[dd]_{f} \ar[dr]^{\l f} & & E\l f
      \ar[d]^{\r {\l f}}\\
      & Ef \ar[d]_{\r f}  \ar@{.>}[ur]|{\c f} & Ef \ar[d]^{\r f}\\
      B \ar[r]^{1} & B \ar[r]^{1} & B
    } \qquad \qquad
    \cd[@-0.35em]{
      A\ar[rr]^{f} \ar[dd]_{f} \ar[dr]^{\l f} & & B
      \ar[dd]^{1_B}\\
      & Ef \ar[d]_{\r f}  \ar@{.>}[ur]|{Rf} \\
      B \ar[r]^{1} & B \ar[r]^{1} & B\rlap{ .}
    }
\end{equation*}

The outer square commutes, and the right edge bears the $\mathsf
R$-algebra structure $\fr f \star \fr {\l f}$; now
applying~\eqref{eq:5} yields the map $\c f$, with the induced square
forming an $\mathsf R$-algebra morphism $\fr f \to \fr f \star \fr {\l
  f}$. If $\star$ did arise from an \awfs, this induced square would
be exactly the right square of~\eqref{eq:13}; whence these $\c f$'s
are the unique possible choice for $\mathsf L$'s comultiplication. By
pasting together appropriate squares as in the proof of
Proposition~\ref{prop:ff2semantics}---involving algebra squares as
on the left above, and also ones
$(Rf, 1) \colon \fr f \to \alg 1_B$ induced by universality as on the
right above---we may now check that these $\c f$'s satisfy the comonad
and distributivity axioms;
we thus have the desired $(\mathsf L, \mathsf R)$, and it remains only
to show that the composition law it induces coincides with
$\star$. So let $\alg g$ and $\alg h$ be composable $\mathsf
R$-algebras, write
$f = hg$ for the composite of the underlying maps, and consider the
left-hand diagram in:
\[
\cd[@-0.7em]{A\ar[rr]_{1} \ar[dd]_{\l f} \ar[dr]_{\l {\l f}} & & A \ar[d]^{g}\\
& E\l f \ar[d]_{\r {\l f}}  \ar@{.>}[ur]|{\exists ! \ell} & B \ar[dd]^h \\
Ef \ar[d]_{\r f} \ar[r]_1 & Ef \ar[d]_{\r f} \ar@{.>}[ur]|{\exists ! m}\\
C \ar[r]^{1} & C \ar[r]^{1} & C
} \qquad \quad \quad
\cd[@-0.5em]{
Ef \ar[dd]_{\r f} \ar[r]^{\c
    f} & E \l f \ar[d]_{\r {\l f}} \ar[r]^{\ell} & A \ar[d]^g \\
 & E f \ar[r]^m \ar[d]_{\r f} & B \ar[d]^h \\
C \ar[r]^1 & C \ar[r]^1 & C\rlap{ .}}
\]
The dotted maps $m$ and $\ell$ are obtained by successively
applying~\eqref{eq:5} to the maps $(g,1) \colon f \to h$ and $(1,m)
\colon \l f \to g$; thus the two squares on the far right above are
maps of $\mathsf R$-algebras $\fr {\l f} \to \alg g$ and $\fr f \to
\alg h$. The rectangle to the left of these squares is a map of
$\mathsf R$-algebras $\fr f \to \fr f \star \fr {\l f}$ by definition,
and also one $\fr f \to \fr f \cdot \fr {\l f}$, by~\eqref{eq:13}.
Thus the composite square can be seen both as a map $\fr f \to \fr {\l
  f} \cdot \fr f \to \alg h \cdot \alg g$ and as one $\fr f \to \fr
{\l f} \star \fr f \to \alg h \star \alg g$. Since precomposing
further with the unit $f \to \r f$ yields the identity square on $f =
hg$, we conclude by universality of $\fr f$ that $(\ell \c f, 1)$ is
both the $\mathsf R$-algebra structure of $\alg h \cdot \alg g$ and
that of $\alg h \star \alg g$, which thus coincide.
\end{proof}

\subsection{Monads over the codomain functor}
\label{sec:monads-over-codomain}
What is missing from the last result is a characterisation of
when a monad on an arrow category is over the codomain functor. Our
next result provides this; it generalises the characterisation
in~\cite[Proposition~5.1]{Im1986On-classes} of \emph{idempotent}
monads over the codomain functor.

\begin{Prop}\label{prop:overcodomain}
  A monad $\mathsf R$ on $\C^\atwo$ is isomorphic to one over the
  codomain functor if and only if:
  \begin{enumerate}[(a)]
  \item Each identity map has an $\mathsf
  R$-algebra structure $\alg 1_A$;
\item For each $f \colon A \to B$, there is an algebra map $(f,f)
  \colon \alg 1_A \to \alg 1_B$;
  \item For each $\mathsf R$-algebra
  $\alg g \colon A \to B$, there is an algebra map
  $(g, 1) \colon \alg g \to \alg 1_B$. 
  \end{enumerate}
Moreover, in this situation, the algebra structures on identity
arrows are unique.
\end{Prop}
\begin{proof}
  As the properties (a)--(c) are clearly invariant under monad
  isomorphism, we may assume in the ``only if'' direction that
  $\mathsf R$ is a monad over the codomain functor; now a short
  calculation shows that $\mathbf 1_A = (1_A, \r{1_A}) \colon A \to A$
  is the unique $\mathsf R$-algebra structure on $1_A$, and that
  (b)--(c) are then satisfied. Conversely, let $\mathsf R$ satisfy
  (a)--(c), and consider the unit map $(r,s) \colon f \to \r f$, as on
  the left in:
\begin{equation*}
\cd{A\ar[d]_{f} \ar[r]^{r} & C \ar[d]^{\r f} \\
B \ar[r]^{s}  & D}
\quad \quad \quad
\cd{A\ar[d]_{f} \ar[r]^{f} & C \ar[d]^{1}\\
B \ar[r]^{1}  & B}
\ 
=
\ 
\cd{A \ar[d]_{f} \ar[r]^{r} & C \ar[d]^{\r f} \ar[r]^{t} & B
  \ar[d]^{1}\\
B \ar[r]^{s}  & D \ar[r]^{u} & B\rlap{ .} }
\end{equation*}
We claim that $s$ is invertible. Since $\r f$ underlies the
free $\mathsf R$-algebra on $f$, and $1_C$ has by (a) the algebra
structure $\alg 1_C$, the commuting square in the middle
above induces a unique algebra map $(t,u) \colon \fr f \to \alg 1_C$
making the rightmost diagram commute. In particular $us=1$, and it
remains to show that $su=1$. Using (b), we have the algebra map
$(s,s)\cdot(t,u) \colon \fr f \to \alg 1_B \to \alg 1_D$, and using
(c) we have $(\r f,1) \colon \fr f \to \alg 1_D$. These maps coincide
on precomposition with the unit $(r,s) \colon f \to \r f$, as on the
left of:
\begin{equation*}
\cd{A \ar[d]_{f} \ar[r]^{r} & C \ar[d]^{\r f} \ar[r]^{t} & B 
  \ar[d]^{1} \ar[r]^{s} & D \ar[d]^{1}\\
B \ar[r]^{s}  & D \ar[r]^{u} & B \ar[r]^{s} & D}
\ 
=
\ 
\cd{A \ar[d]_{f} \ar[r]^{r} & C \ar[d]^{\r f} \ar[r]^{\r f} & D
  \ar[d]^{\alg 1}\\
B \ar[r]^{s}  & D \ar[r]^{1} & D} \qquad \quad
\cd{C\ar[d]_{\r f} \ar[r]^{1} & C \ar[d]^{\rp f} \\
D \ar[r]^{u} & B}
\end{equation*}
So they must themselves agree; in particular, $su=1$ and $s$ and $u$
are isomorphisms. So defining $\rp f = u \r f$, we have isomorphisms
$\r f \cong \rp f$ for each $f \in \C^\atwo$ as on the right above;
now successively transporting the functor $R$ and the monad
structure thereon along these isomorphisms yields a monad $\mathsf
R' \cong \mathsf R$ over the codomain functor, as required.
\end{proof}

\subsection{The Beck theorem}
\label{sec:beck-theorem}
Combining the preceding two results, we obtain our first main theorem,
characterising the concrete double categories in the essential image
of the semantics $2$-functor $\DAlg{(\thg)}$. Of course, we have also
the dual result, which we do not trouble to state, characterising the
essential image of $\DCoalg{(\thg)}$.
\begin{Thm}\label{thm:recognition}
  The $2$-functor $\DAlg{(\thg)} \colon \Lax \to \DBL^\atwo$ has in
  its essential image exactly those concrete double categories $V
  \colon \mathbb A \to \Sq{\C}$ such that:
\begin{enumerate}[(i)]
\item The functor $V_1 \colon \A_{1} \to \C^{\atwo}$ on vertical
  arrows and squares is strictly monadic;
\item For each vertical arrow $\alg f \colon A \to B$ of $\mathbb
  A$, the following is a square of $\mathbb A$:
  \begin{equation}
\cd{ A\ar[r]^{f} \ar[d]_{\alg f} & B \ar[d]^{\alg 1} \\
B \ar[r]^{1} & B\rlap{ .}}\label{eq:16}
\end{equation}
\end{enumerate}
\end{Thm}
\begin{proof}
  Any $U^\mathsf R \colon \DAlg{R} \to \Sq{\C}$ clearly has property
  (i) while property (ii) follows from
  Proposition~\ref{prop:overcodomain}. Suppose conversely that the
  concrete $V \colon \mathbb A \to \Sq{\C}$ satisfies (i) and (ii). As
  $V_1$ is strictly monadic, $\A_1$ is isomorphic over $\C^\atwo$ to
  the category of algebras for the monad $\mathsf R$ induced by $V_1$
  and its left adjoint. This $\mathsf R$ satisfies (a)--(c) of
  Proposition~\ref{prop:overcodomain}: (c) by virtue of (ii) above,
  and (a) and (b) using the vertical identities of the double category
  $\mathbb A$. Thus we have a monad $\mathsf R' \cong \mathsf R$ over
  the codomain functor; now transporting the double category structure
  of $V \colon \mathbb A \to \Sq{\C}$ along the isomorphisms $\A_1
  \cong \Alg{R} \cong \Alg{R'}$ and $\C^\atwo$ yields $V' \colon
  \mathbb A' \to \Sq{\C}$ which, by
  Proposition~\ref{prop:reconstruct}, is in the image of
  $\DAlg{(\thg)}$.
\end{proof}
We will call a concrete double category \emph{right-connected} if it
satisfies (ii) above, and \emph{monadic right-connected} if it
satisfies (i) and (ii). Now combining Theorem~\ref{thm:recognition}
with Proposition~\ref{prop:reconstruct} and the remarks preceding it
yields the following result; again, there is a dual form
which we do not state characterising $\Oplax$.
\begin{Cor}\label{cor:characterise}
  The $2$-category $\Lax $ is equivalent to the full sub-$2$-category
  of $\DBL^\atwo$ on the monadic right-connected concrete double
  categories.
\end{Cor}
We conclude this section by describing two variations on our main
result which will come in useful from time to time.
\subsection{Discrete pullback-fibrations}
\label{sec:discr-pullb-fibr}
Under mild conditions on the base category $\C$, we may rephrase
Proposition~\ref{prop:overcodomain} so as to obtain a more intuitive
characterisation of the concrete double categories over $\C$ in the
essential image of the semantics functor. Let us define a functor
$p \colon \A \to \C^\atwo$ to be a \emph{discrete pullback-fibration}
if, for every $\alg g \in \A$ over $g \in \C^\atwo$ and every pullback
square $(h, k) \colon f \to g$, there is a unique arrow
$\phi \colon \alg f \to \alg g$ in $\A$ over $(h,k)$, and this arrow
is cartesian:
\begin{equation}
\cd[@R-1.8em]{
& & & & A \ar[dd]_{f}
\ar[r]^h \pushoutcorner & C \ar[dd]^{g} \\
\exists! \alg f \ar@{.>}[r]^{\exists! \phi} & \alg g & {}\ar@{|->}[r] & {} \\ & & & & B \ar[r]^k &
D\rlap{ .}
}\label{eq:15}
\end{equation}

If $\C$ has all pullbacks, then the codomain functor $\C^\atwo \to \C$
is a fibration, with the pullback squares in $\C^\atwo$ as its
cartesian arrows; now the fact that the lifts~\eqref{eq:15} are
cartesian implies that the composite functor
$\mathrm{cod} \cdot p \colon \A \to \C^\atwo \to \C$ is also a
Grothendieck fibration, with cartesian liftings preserved and
reflected by $p$.

\begin{Prop}\label{prop:overcodomain2}
  If $\C$ has pullbacks, then a monad $\mathsf R$ on $\C^\atwo$ is
  isomorphic to one over the codomain functor if and only if
  $U^\mathsf R \colon \Alg{R} \to \C^\atwo$ is a discrete
  pullback-fibration.
\end{Prop}
Under the hypotheses of this proposition, the composite
$\mathrm{cod} \cdot U^\mathsf R \colon \Alg{R} \to \C$ is a
Grothendieck fibration; which, loosely speaking, is the statement that
``$\mathsf R$-algebra structure is stable under pullback''.
\begin{proof}
  First let $\mathsf R$ be over the codomain functor. Given an
  $\mathsf R$-algebra $\alg g = (g,p) \colon C \to D$ and pullback
  square as in~\eqref{eq:15}, form the unique map $q \colon Ef \to A$
  with $fq = \r f$ and $hq = p \cdot E(h,k)$. This easily yields an
  $\mathsf R$-algebra $(f,q) \colon A \to B$ for which
  $(h,k) \colon (f,q) \to (g,p)$ is a cartesian algebra map; moreover,
  any $q'$ for which $(h,k)$ were a map $(f,q') \to (g,p)$ would have
  to satisfy the defining conditions of $q$; whence $q$ is unique as
  required. Conversely, suppose that $U^\mathsf R$ is a discrete
  pullback-fibration. Arguing as in
  Proposition~\ref{prop:overcodomain}, it suffices to show that each
  unit-component $(r,s) \colon f \to \r f$ has $s$ invertible. Form
  the pullback of $\r f$ along $s$ and induced map $t$ as on the left
  in:
  \begin{equation*}
    \cd{A \ar[d]_{f} \ar@/^1.5em/[rr]^r \ar@{.>}[r]^{t} & E \pushoutcorner \ar[d]^{g} \ar[r]^{u} & C
  \ar[d]^{\r f}\\
B \ar[r]^{1}  & B \ar[r]^{s} & D} \qquad \qquad \cd{
  A \ar[r]^t \ar[d]_f & E \ar[d]^g \\ B \ar[r]^1 & B} \ = \ 
\cd{A \ar[r]^r \ar[d]_f & C \ar[r]^v \ar[d]_{\r f} & E \ar[d]^g \\ B \ar[r]^s & D
  \ar[r]^w & B\rlap{ .}}
\end{equation*}
As $U^\mathsf R$ is a discrete pullback-fibration, there is a unique
$\mathsf R$-algebra structure $\alg g$ on $g$ making $(u,s) \colon
\alg g \to \fr f$ an algebra map; now applying freeness of $\fr f$ to
the central square above yields a unique algebra map $(v,w) \colon \fr
f \to \alg g$ making the right-hand diagram commute. In particular,
$ws = 1$; on the other hand, the algebra map $(uv,sw) \colon \fr f \to
\alg g \to \fr f$ precomposes with the unit $(r,s) \colon f \to \r f$
to yield $(r,s)$; whence $(uv,sw) = (1,1)$ and $s$
is invertible as required.
\end{proof}

The evident adaptation of the proof of Theorem~\ref{thm:recognition}
now yields:
\begin{Thm}\label{thm:recognition2}
  If $\C$ admits pullbacks, then a concrete double category $V \colon \mathbb A \to
  \Sq{\C}$ is in the essential image of $\DAlg{(\thg)} \colon
  \Lax  \to \DBL^\atwo$
  just when:
\begin{enumerate}[(i)]
\item $V_1 \colon \A_{1} \to \C^{\atwo}$ 
 is strictly monadic;
\item $V_1 \colon \A_{1} \to \C^{\atwo}$ is a discrete
  pullback-fibration.
\end{enumerate}
\end{Thm}
\begin{Rk}
  Note that for the ``if'' direction of this result, $\C$ does not
  need to have \emph{all} pullbacks; a closer examination of the proof
  of Proposition~\ref{prop:overcodomain2} shows that only pullbacks
  along (the underlying maps of) $\mathsf R$-algebras are needed. This
  is useful in practice; categories with pullbacks of this restricted
  kind arise in the study of stacks~\cite{Stanculescu2014Stacks}, or
  in the categorical foundations of Martin--L\"of type
  theory~\cite{Gambino2008The-identity}.
\end{Rk}

\subsection{Intrinsic concreteness}
\label{sec:intr-char}
Above, we have chosen to view the semantics $2$-functors~\eqref{eq:18}
as landing in the $2$-category $\DBL^\atwo$. We have done this so
as to stress the analogy with the situation for monads and comonads,
and also because this is the most practically useful form of our
results. Yet this presentation has some redundancy, as we now explain.

In the presence of the remaining hypotheses, condition (ii) of
Theorem~\ref{thm:recognition} is easily seen to be equivalent to the
requirement that the codomain functor $c \colon \A_1 \to \A_0$ of the
double category $\mathbb A$ be a left adjoint left inverse for the
identities functor $i \colon \A_0 \to \A_1$ (such \emph{lalis} will
appear again in Section~\ref{sec:lalis} below). This is a
property of $\mathbb A$, rather than extra structure; namely the
property that
\begin{equation}
  \text{each } c_{\alg f, iB} \colon \A_1(\alg f, iB) \to
  \A_0(c\alg f, B) \text{ is invertible.}\label{eq:24}
\end{equation}
But given only this, we may reconstruct the double functor $V \colon
\mathbb A \to \Sq{\A_0}$. First we reconstruct the
components~\eqref{eq:16} of the unit $\eta \colon 1 \Rightarrow ic$
using~\eqref{eq:24}; now we reconstruct $V_1 \colon \A_1 \to
\A_0^\atwo$ as the functor classifying the natural transformation
\begin{equation}
\cd{
\A_1 \ar@/^0.7em/[r]^{1} \ar@/_0.7em/[r]_{ic}
\dtwocell{r}{\eta} & \A_1 \ar[r]^d & \A_0\rlap{ ;}
}\label{eq:19}
\end{equation}
finally, by pasting together unit squares and using~\eqref{eq:24}, we
verify that $V = (\id,V_1)$ preserves vertical composition, and so is
a double functor. If $W \colon \mathbb B \to \Sq{\B_0}$ is another
concrete double category whose $W$ is determined in this manner, then
\emph{any} double functor $F \colon \mathbb A \to \mathbb B$ will
by~\eqref{eq:24} satisfy
$F_1 \cdot \eta_\mathbb A = \eta_\mathbb B \cdot F_1$, and so
$\Sq{F_0} \cdot V = W \cdot F$; that is, any double functor
$\mathbb A \to \mathbb B$ is concrete.

Let us, therefore, define a double category $\mathbb A$ to be
\emph{right-connected} if it satisfies~\eqref{eq:24}, and
\emph{monadic right-connected} if in addition the induced functor
$\A_1 \to (\A_0)^\atwo$ is strictly monadic. The preceding discussion
now shows that:
\begin{Prop}\label{prop:monadicrc}
  The assignation $(\C, \mathsf L, \mathsf R) \mapsto \DAlg{R}$ gives
  an equivalence of $2$-categories between $\Lax$ and the full
  sub-$2$-category of $\DBL$ on the monadic right-connected double
  categories.
\end{Prop}

\section{Double categories at work}\label{sec:double-categories-at}
As a first illustration of the usefulness of the double categorical
approach and our Theorem~\ref{thm:recognition}, we use it to construct
a variety of algebraic weak factorisation systems, some well-known,
and some new.

\subsection{Split epimorphisms}
\label{sec:split-epimorphisms}
Given a category $\C$, we write $\cat{SplEpi}(\C)$ for the category of
split epimorphisms therein: objects are pairs of a map $g \colon A \to
B$ of $\C$ together with a section $p$ of $g$, while morphisms $(g,p)
\to (h,q)$ are serially commuting diagrams as on the left in:
\begin{equation}
\cd{A \ar@<-3pt>[d]_{g} \ar[r]^{u} & C \ar@<-3pt>[d]_{h} \\
B \ar@<-3pt>[u]_{p} \ar[r]_{s} & D \ar@<-3pt>[u]_{q}}
\quad\quad \quad \quad \quad
\cd{
1 \ar[r]^{m} \ar[dr]_{1} & 0 \ar[d]^{e} \ar[dr]^{me} \\
& 1 \ar[r]_{m} & 0\rlap{ .}
}\label{eq:22}
\end{equation}
Split epimorphisms compose---by composing the sections---so that we
have a double category $\dcat{SplEpi}(\C)$ which is concrete over $\C$
and easily seen to be right-connected. It will thus be the double
category of algebras of an \awfs on $\C$ whenever $U \colon
\cat{SplEpi}(\C) \to \C^\atwo$ is strictly monadic.

Now, we may identify $U$ with $\C^j \colon \C^\Ss \to \C^\atwo$, where
$\Ss$ is the \emph{free split epimorphism}, drawn above right, and $j
\colon \atwo \to \Ss$ is the evident inclusion. Thus $U$ strictly
creates colimits, and so will be strictly monadic whenever it has a
left adjoint: which will be so whenever $\C$ is cocomplete
enough to admit left Kan extensions along~$j$. Using the Kan
extension formula one finds that only binary coproducts are required;
the free split epi $\r f$ on $f \colon A \to B$ is $\spn{f,1} \colon
A+B \to B$ with section $\iota_B$, while the unit $f \to \r f$ is
given by:
\begin{equation}
  \cd{A \ar[d]_{f} \ar[r]^-{\iota_{A}} & A+B \ar@<-3pt>[d]_{\spn{f,1}} \\
B \ar[r]_-{1} & B\rlap{ .} \ar@<-3pt>[u]_{\iota_B}}\label{eq:11}
\end{equation}

The remaining structure of the \awfs for split epis can  be
calculated using Proposition~\ref{prop:reconstruct}. We find that each
coproduct inclusion admits a $\mathsf L$-coalgebra structure; while if
$\C$ is lextensive, then this coalgebra structure is \emph{unique},
and the category of $\mathsf L$-coalgebras is precisely the category
of coproduct injections and pullback squares. This was proved in
\cite[Proposition~4.2]{Grandis2006Natural}, but see also
Example~\ref{ex:8}(\ref{item:coprod}) below.

Finally, let us observe an important property of the split epi \awfs.
Although $\mathsf R = (R,\eta,\mu)$ is a monad, its algebras---the
split epis---are simultaneously the algebras for its underlying
pointed endofunctor $(R,\eta)$; which is to say that the monad
$\mathsf R$ is \emph{algebraically-free}~\cite[\S
22]{Kelly1980A-unified} on its underlying pointed endofunctor. We will
make use of this property in Proposition~\ref{prop:6} below.

\subsection{Lalis}
\label{sec:lalis}
In the terminology of Gray \cite{Gray1966Fibred}, a \emph{lali} 
(left adjoint left inverse) in a $2$-category $\C$ is a
split epi $(g,p) \colon A \to B$ with the extra property that $g
\dashv p$ with identity counit. This property may be expressed either
by requiring that for each $x \colon A \to X$ and $y \colon B \to X$
in $\C$, the function
\begin{equation}
(\thg) \cdot p \colon \C(A,X)(x,yg) \to \C(B,X)(xp,y)\label{eq:25}
\end{equation}
be invertible; or by requiring the provision of a---necessarily
unique---unit $2$-cell $\eta \colon 1 \Rightarrow pg$ satisfying
$g\eta = 1$ and $\eta p = 1$. Lalis in $\C$ form a category
$\cat{Lali}(\C)$, wherein a morphism is a commuting diagram as on the
left of~\eqref{eq:22}; it is automatic by invertibility
of~\eqref{eq:25} that such a morphism also commutes with the unit
$2$-cells.

Since split epis and adjoints compose, so too do lalis; thus---writing
$\C_0$ for the underlying category of $\C$---lalis in $\C$ form a
concrete sub-double category $\dcat{Lali}(\C)$ of
$\dcat{SplEpi}(\C_0)$ which, since it is full on cells, inherits
right-connectedness. So $\dcat{Lali}(\C)$ will be the double category
of algebras for an \awfs on $\C_0$ whenever $U \colon \cat{Lali}(\C)
\to (\C_0)^{\atwo}$ is strictly monadic. Now, we may identify $U$ with
restriction
\[
\cat{2}\text-\CAT(j, \C) \colon \cat{2}\text-\CAT(\L, \C) \to \cat{2}\text-\CAT(\atwo, \C)
\]
along the inclusion $j \colon \atwo \to \L$ of $\atwo$ into the
\emph{free lali} $\L$---which has the same underlying category as the
free split epimorphism $\Ss$ and a single non-trivial 2-cell $1
\Rightarrow me$; so as before, monadicity obtains whenever $\C$ is
cocomplete enough (as a $2$-category) to admit left Kan extensions
along $j$. In this case, the colimits needed are \emph{oplax colimits
  of arrows}: given $f \colon C \to B$, its oplax colimit is the
universal diagram as on the left of:
\begin{equation*}
  \cd[@C-0.2em]{
    C \ar[rr]^f \ar[dr]_\ell & \rtwocell{d}{\alpha} & B \ar[dl]^p & & &
    C \ar[rr]^f \ar[dr]_f & \rtwocell{d}{1} & B \ar[dl]^1 \\
        & A & & & & & B\rlap{ .}
  }
\end{equation*}
Applying the one-dimensional aspect of this universality to the cocone
on the right yields a retraction $g \colon A \to B$ for $p$; now
applying the two-dimensional aspect shows that the pair $(g,p)$
verifies invertibility in~\eqref{eq:25}, and so is a lali, the free
lali on $f$, with as unit $f \to Rf$ the map $(\ell, 1)$.

When $\C = \Cat$, we may characterise the corresponding category of
$\mathsf L$-coalgebras as the subcategory of $\Cat^\atwo$ whose
objects are the \emph{categorical cofibrations}---those functors
arising as pullbacks of the domain inclusion $\cat{1} \to \atwo$---and
whose morphisms are the pullback squares; see
example~(\ref{item:catcof}) in Section~\ref{sec:stable-class-monom}
below.\footnote{This characterisation would extend to any suitably
  exact $2$-category, except that there is no account of the
  appropriate two-dimensional exactness notion; it appears to be
  closely related to the one mentioned in the introduction
  to~\cite{Street1980Fibrations}.}

On reversing or inverting the non-trivial $2$-cell of $\L$, it becomes
the free \emph{rali} (right adjoint left inverse) $\L^\co$ or the free
\emph{retract equivalence} $\L^g$; from which we obtain \awfs whose
algebras are ralis or retract equivalences on any $2$-category that
admits lax colimits or pseudocolimits of arrows. When $\C = \Cat$, the
respective $\mathsf L$-coalgebra structures on a morphism $f$ are
unique, and exist just when $f$ is a pullback of the codomain
inclusion $\aone \to \atwo$, respectively, when $f$ is injective on
objects.

\subsection{Via cocategories}\label{sec:via-cocategories}
The preceding examples fit into a common framework. Let $\V$ be a
suitable base for enriched category theory, and suppose that we are
given a cocategory object $\dcat A$ in $\VCat$:
\begin{equation*}
\cd{
\mathbb A_0 = \cat I \ar@<5pt>[rr]^{d} \ar@<-5pt>[rr]_{c} & &
\mathbb A_1 \ar[ll]|{i} \ar@<5pt>[rr]^{p} \ar@<-5pt>[rr]_{q}
\ar[rr]|{m}& & \mathbb A_{2} }
\end{equation*}
with object of co-objects the unit $\V$-category, with $d$ and $c$
jointly bijective on objects, and with $i \dashv c$ with identity
unit. For any $\V$-category $\C$, the image of $\mathbb A$ under the
limit-preserving $\VCAT(\thg, \C) \colon \VCat^\op \to \CAT$ is an
internal category $\mathbb A(\C)$ in $\CAT$---a double
category---which is right-connected and concrete over $\C_0$, the
underlying ordinary category of $\C$. The induced comparison functor
$\mathbb A(\C)_1 \to (\C_0)^\atwo$ is the functor
\begin{equation}
\VCAT(j, \C) \colon \VCAT(\mathbb A_1, \C) \to \VCAT(\atwo, \C)\label{eq:26}
\end{equation}
where $j \colon \atwo \to \mathbb A_1$ is the functor from the free
$\V$-category on an arrow which classifies the $d$-component of the
unit
$\eta \colon 1 \Rightarrow ci \colon \mathbb A_1 \to \mathbb A_1$.
This $j$ is bijective on objects, and so~\eqref{eq:26} will be monadic
whenever $\C$ is sufficiently cocomplete that each $\V$-functor
$\atwo \to \C$ admits a left Kan extension along $j$; whereupon
Theorem~\ref{thm:recognition} induces an \awfs on $\C$ with
$\mathbb A(\C)$ as its double category of algebras. Note that the
assignation $\C \mapsto \mathbb A(\C)$ is clearly $2$-functorial in
$\C$, so that by Proposition~\ref{prop:monadicrc}, it underlies a
$2$-functor $\VCAT' \to \Lax$ defined on all $\V$-categories with
sufficient colimits and all $\V$-functors (not necessarily
cocontinuous) between them.

Of course, the preceding examples are instances of this framework on
taking $\V = \Set$ with $\mathbb A_{1}= \Ss$ (for split epis); or
taking $\V = \cat{Cat}$ with $\mathbb A_1 = \L$, $\L^\co$ or $\L^g$
(for lalis, ralis or retract equivalences).

\subsection{Stable classes of monics}
\label{sec:stable-class-monom}

Let $\C$ be a category with pullbacks, and consider a class $\M$ of
monics which contains the isomorphisms, is closed under composition,
and is stable under pullback along maps of $\C$; we call this a
\emph{stable class of monics}. The class $\M$ is said to be
\emph{classified} if the category $\M_{\mathrm{pb}}$ of pullback
squares between $\M$-maps has a terminal object, which we call a
\emph{generic $\M$-map}. A standard
argument~\cite[p.~24]{Johnstone1977Topos} shows that the domain of a
generic $\M$-map must be terminal in $\C$.

Since $\M$-maps and pullback squares compose, the category $\M_{\mathrm{pb}}$ underlies a double category $\dcat M_{\mathrm{pb}}$
concrete over $\C$. The monicity of the $\M$'s ensures that for each
$m \in \M$, the square on the left below is a pullback and so in $\M_\mathrm{pb}$. Thus $\dcat M_{\mathrm{pb}}$ is left-connected, and so
will comprise the $\mathsf L$-coalgebras of an \awfs on $\C$ as soon
as the forgetful $U \colon \M_{\mathrm{pb}} \to \C^\atwo$ is
comonadic.
\[
\cd[@-0.5em]{
A \ar[d]_1 \ar[r]^1 & A \ar[d]^m \\ A \ar[r]_m & B}
\qquad \quad
\cd[@-0.5em]{
A_1 \ar[r]^{f_1} \ar[d]_{a} & B_1 \ar@<2pt>[r]^{g_1}
\ar@<-2pt>[r]_{h_1} \ar[d]_{b} & C_1 \ar[d]_c\\
A_2 \ar[r]_{f_2} & B_2 \ar@<2pt>[r]^{g_2}
\ar@<-2pt>[r]_{h_2} & C_2 
}
\]

First we note that $U$ always strictly creates equalisers. Indeed, if
$b \rightrightarrows c$ is a parallel pair in $\M_{\mathrm{pb}}$
as on the far right above, then on forming the equaliser $a \to b$ in
$\C^\mathbf 2$, as to the left, the equalising square will be a
pullback---since $c$ is monic and both rows are equalisers---and so
will lift uniquely to $\M_\mathrm{pb}$. Thus $U$ will be strictly
comonadic as soon as it has a right adjoint.

For this, we must assume that the generic $\M$-map $p \colon 1 \to
\Sigma$ is exponentiable; it follows that every $\M$-map is
exponentiable, since exponentiable maps are pullback-stable in any
category with pullbacks
(see~\cite[Corollary~2.6]{Street2010The-comprehensive}, for example).
In particular, each map $B \times p \colon B \to B \times \Sigma $ is
exponentiable, and so for any $f \colon A \to B$ of $\C$, we may form
the object $\Pi_{B \times p}(f)$ of $\C / (B \times \Sigma)$; which
has the universal property that, in the category $\P_f$ whose objects
are pullback diagrams as on the left in
\[
\cd[@C-0.5em]{ X \ar[d]_\ell \ar[r]^-h & A \ar[r]^{f} & B
  \ar[d]^{B \times p}\\
  Y \ar[rr]_-{k} & & B \times \Sigma } \qquad \qquad \quad
\cd[@C-0.5em]{ C \ar[d]_\eta \ar[r]^-\varepsilon & A \ar[r]^{f} & B
  \ar[d]^{B \times p}\\
  D \ar[rr]_-{\Pi_{B \times p}(f)} & & B \times \Sigma }
\]
and whose maps $(X,Y,h,k,\ell) \to (X',Y',h', k', \ell')$ are maps $X
\to X'$ and $Y \to Y'$ satisfying the obvious equalities, there is a
terminal object as on the right. Now, an object of $\P_f$ determines
and is determined by a pair of squares
\[
\cd[@-0.5em]{ X \ar[d]_\ell \ar[r]^-{h_1} & A \ar[r]^{f} & B
  \ar[d]^{1}\\
  Y \ar[rr]_{k_1} & & B} \qquad \qquad 
\cd[@-0.5em]{ X \ar[d]_\ell \ar[rr]^-{!} & & 1
  \ar[d]^{p}\\
  Y \ar[rr]_{k_2} & & \Sigma }
\]
with the right-hand one a pullback. This ensures that $\ell$---as a
pullback of $p$---is an $\M$-map, whereupon by genericity of $p$, the
map $k_2$ is uniquely determined. Thus we may identify $\P_f$ with the
comma category $U \downarrow f$, and the terminal object therein
provides the value $\eta \colon C \to D$ of the desired right adjoint
at $f$.

For the $(\mathsf L, \mathsf R)$ induced in this way, the restriction
of the monad $\mathsf R$ to the slice over $B$ is the \emph{partial
  $\M$-map classifier monad} on $\C / \B$;
see~\cite{Rosolini1986Continuity} and the references therein. Note
also that, in the terminology of~\cite{Dyckhoff1987Exponentiable}, the
$\Pi_{B \times p}(f)$ constructed above is a \emph{pullback
  complement} of $f$ and $B \times p$.

\begin{Exs}
  \begin{enumerate}[(i)]
  \item Let $\E$ be an elementary topos, and $\M$ the stable class of
    all monics. $\M$~is classified by $\top \colon 1 \to \Omega$,
    which---like any map in a topos---is exponentiable; and so we have
    an \awfs on $\E$ for which $\Coalg{L}$ is the category of
    monomorphisms and pullback squares. The corresponding $\mathsf
    R$-algebras are discussed in~\cite{Kock1991Algebras}. More
    generally, we can let $\E$ be any quasitopos, and $\M$ the class
    of strong monomorphisms.\vskip0.5\baselineskip
  \item Let $\E$ be an elementary topos, $j$ a Lawvere-Tierney
    topology on $\E$, and $\M$ the class of $j$-dense monomorphisms.
    This is a stable class of monics classified by $\top \colon 1 \to
    J$, where $J$ is the equaliser of $j, \top \colon \Omega \to
    \Omega$. So $j$-dense monomorphisms and their pullbacks form the
    coalgebras of an \awfs on $\E$. An algebraically fibrant object is
    a ``weak sheaf'': an object equipped with coherent, but not
    unique, choices of patchings for $j$-covers. An object $X$ is a
    sheaf if and only if both $X \to 1$ and $X \to X \times X$ admit
    $\mathsf R$-algebra structure. \vskip0.5\baselineskip
  \item\label{item:coprod} Let $\C$ be a lextensive category and $\M$
    the stable class of coproduct injections. This is classified by
    $\iota_1 \colon 1 \to 1 + 1$; which by extensivity is always
    exponentiable, with $\Pi_{\iota_1} = (\thg) + 1 \colon \C \to \C /
    (1+1)$. More generally, the right adjoint to pullback along $B
    \times \iota_1$ is $(\thg) + B \colon \C / B \to \C / (B+B)$;
    whence the \awfs generated on $\C$ is that for split epis. This
    re-proves the fact that, in any lextensive category, the
    coalgebras of the \awfs for split epis are the coproduct
    injections and pullback squares.

    \vskip0.5\baselineskip
  \item\label{item:catcof} Consider, as in Section~\ref{sec:lalis},
    the stable class of categorical cofibrations in $\cat{Cat}$. This
    is classified by the domain functor $d \colon \aone \to \atwo$,
    which is exponentiable in $\cat{Cat}$, so that we have an \awfs
    wherein $\mathsf L$-coalgebras are the categorical cofibrations
    and pullback squares. For any category $B$, the functor $\Pi_{B
      \times d} \colon \cat{Cat} / B \to \cat{Cat} / (B \times \atwo)$
    sends $f \in \cat{Cat} / B$ to the induced map between the oplax
    colimits of $f$ and of $1_B$---the latter being simply $B \times
    \atwo$---and it follows that this \awfs is the one for lalis. This
    proves the claim about its coalgebras made in
    Section~\ref{sec:lalis} above.\vskip0.5\baselineskip

  \item\label{item:1} By a \emph{left cofibration} of simplicial sets,
    we mean a pullback of the face inclusion $\delta_1 \colon \Delta[0] \to
    \Delta[1]$. The left cofibrations constitute a stable class of
    monics, classified by $\delta_1$; and since $\cat{SSet}$ is a topos,
    this $\delta_1$ is exponentiable. We thus have an \awfs on simplicial
    sets whose $\mathsf L$-coalgebras are left cofibrations and
    pullbacks. The corresponding $\mathsf R$-coalgebras we call the
    \emph{simplicial lalis}; we will have more to say about them in
    Section~\ref{sec:enrich-cofibr-gener} below.
  \end{enumerate}\label{ex:8}
\end{Exs}

\subsection{Projective and injective liftings}
\label{sec:liftings-awfs}
Our next example is a construction which, under suitable
circumstances, allows us to lift an \awfs $(\D, \mathsf L, \mathsf R)$
along a functor $F \colon \C \to \D$ to yield an \awfs on $\C$. There
are two ways of doing this, well known from the model category
literature: we either lift the algebras (as
in~\cite{Crans1995Quillen}) or the coalgebras (as
in~\cite{Bayeh2014Left-induced}). In the former case, we form the
pullback on the left in:
  \begin{equation}
    \cd[@C+0.5em]{
      \dcat A \pushoutcorner \ar[r] \ar[d]_{V} & \DAlg{R} \ar[d]^{U^{\mathsf
          R}}\\
      \Sq{\C} \ar[r]^{\Sq{F}} & \Sq{\D}} \qquad \qquad \qquad
    \cd{
      \A_1 \pushoutcorner\ar[r] \ar[d]_{V_1} & \Alg{R} \ar[d]^{U^{\mathsf
          R}}\\
      \C^\atwo \ar[r]^{F^\atwo} & \D^\atwo\rlap{ .}
    }\label{eq:39}
\end{equation}
The double category $\mathbb A$ therein, as a pullback of
right-connected double categories, is itself right-connected; while
the component functor $V_1$ of $V$ as on the right above is a pullback
of the strictly monadic $U^{\mathsf R}$, whence easily seen to
strictly create coequalisers for $V_1$-absolute coequaliser pairs. So
whenever $V_1$ admits a left adjoint, $\mathbb A$ will by
Theorem~\ref{thm:recognition} constitute the algebra double category
of an \awfs $(\mathsf L', \mathsf R')$ on $\C$, the \emph{projective
  lifting} of $(\mathsf L, \mathsf R)$ along $F$. The left square
of~\eqref{eq:39} then amounts to a lax morphism of \awfs $(\C, \mathsf
L', \mathsf R') \to (\D, \mathsf L, \mathsf R)$, which is easily seen
to be a cartesian lifting of $F$ along the forgetful functor $\Lax \to
\CAT$.

Dually, we may form the pullback of double categories on the left in
\begin{equation}
  \cd{
    \dcat B \pushoutcorner \ar[r] \ar[d]_{W} & \DCoalg{L} \ar[d]^{U^{\mathsf
        L}}\\
    \Sq{\C} \ar[r]^{\Sq{F}} & \Sq{\D}
  } \qquad \qquad \qquad
  \cd{
    \B_1 \pushoutcorner \ar[r] \ar[d]_{W_1} & \Coalg{L} \ar[d]^{U^{\mathsf
        L}}\\
    \C^\atwo \ar[r]^{F^\atwo} & \D^\atwo\rlap{ ;}   
  }\label{eq:40}
\end{equation}
if on doing so, the component $W_1$ of $W$ as to the right has a right
adjoint, then $\dcat B$ will be the coalgebra double category of an
\awfs on $\C$, the \emph{injective lifting} of $(\mathsf L, \mathsf
R)$ along $F$. The left square of~\eqref{eq:40} now corresponds to an
oplax map of \awfs providing a cartesian lifting of $F$ with respect
to the forgetful $\Oplax \to \CAT$.

Sometimes, it may be that the required adjoints to $V_1$ or $W_1$
simply exist. For example, the projective and injective liftings of an
\awfs $(\C, \mathsf L, \mathsf R)$ along a forgetful functor $\C / X \to
\C$ always exist, and coincide; the factorisations of the resultant
\emph{slice \awfs} are given by
\begin{equation*}
  \cd[@-0.2em]{A \ar[rr]^f \ar[dr]_a & & B \ar[dl]^b \\
    & X  } 
  \qquad \mapsto \qquad
  \cd[@-0.2em]{ A \ar[r]^{\l f} \ar[dr]_a & 
    Ef \ar[r]^{\r f} \ar[d]|{b \cdot \r f} & B \ar[dl]^b \rlap{ .} \\
    & X
  }
\end{equation*}

More typically, we will appeal to a general result like the following,
which ensures that the desired adjoint exists without necessarily
giving a closed formula for it. As in the introduction, we call an
\awfs on a locally presentable category \emph{accessible} if its
comonad $\mathsf L$ and monad $\mathsf R$ are so, in the sense of
preserving $\kappa$-filtered colimits for some regular cardinal
$\kappa$; in fact, it is easy to see that accessibility of the comonad
implies that of the monad, and vice versa.

\begin{Prop}
\label{prop:2}  Let $\C$ and $\D$ be locally presentable categories, and
  $(\mathsf L, \mathsf R)$ an accessible \awfs on $\D$.
  \begin{enumerate}[(a)]
  \item The projective lifting of $(\mathsf L, \mathsf R)$ along any
    right adjoint $F \colon \C \to \D$ exists and is
    accessible;\vskip0.2\baselineskip
  \item The injective lifting of $(\mathsf L, \mathsf R)$ along any
    left adjoint $F \colon \C \to \D$ exists and is accessible.
  \end{enumerate}
\end{Prop}
\begin{proof}
  For (a), let $F \colon \C \to \D$ be a right adjoint. To show the
  projective lifting exists, we must prove that the functor $V_1$ to
  the right of~\eqref{eq:39} has a left adjoint. Now, $\C^\atwo$ and
  $\D^\atwo$ are locally presentable since $\C$ and $\D$ are,
  by~\cite[\S7.2(h)]{Gabriel1971Lokal}, while $\Alg{R}$ is locally
  presentable since $\mathsf R$ is accessible,
  by~\cite[Satz~10.3]{Gabriel1971Lokal}; moreover, both $F^\atwo$ and
  $U^\mathsf R$ are right adjoints. Since $U^\mathsf R$ has the
  isomorphism-lifting property, its pullback against $F^\atwo$ is also
  a bipullback~\cite{Joyal1993Pullbacks}; but
  by~\cite[Theorem~2.18]{Bird1984Limits}, the $2$-category of locally
  presentable categories and right adjoint functors is closed under
  bilimits in $\CAT$, whence $V_1$, like $F^\atwo$ and $U^\mathsf R$,
  lies in this $2$-category; thus it has a left adjoint $K_1$ and is
  accessible by~\cite[Satz~14.6]{Gabriel1971Lokal}. So the projective
  lifting of $(\mathsf L, \mathsf R)$ along $F$ exists, and is
  accessible as its monad $V_1 K_1$ is so.

  For (b), let $F \colon \C \to \D$ be a left adjoint; we must show
  that $W_1$ to the right of~\eqref{eq:40} has a right adjoint. We now
  argue using~\cite[Theorem~3.15]{Bird1984Limits}, which states that the
  $2$-category of locally presentable categories and left adjoint
  functors is closed under bilimits in $\CAT$. As before, $F^\atwo
  \colon \C^\atwo \to \D^\atwo$ lies in this $2$-category, and
  $U^\mathsf L \colon \Coalg{L} \to \D^\atwo$ will so long as
  $\Coalg{L}$ is in fact locally presentable. But $\mathsf L$ is an
  accessible comonad on the accessible category $\D^\atwo$;
  by~\cite[Theorem~5.1.6]{Makkai1989Accessible}, the $2$-category of
  accessible categories and accessible functors is closed in $\CAT$
  under all bilimits, in particular under Eilenberg-Moore objects of
  comonads, and so $\Coalg{L}$ is accessible; it is also cocomplete
  (since $U^\mathsf L$ creates colimits) and so is locally
  presentable, as required. We thus conclude that the functor $W_1$,
  like $U^\mathsf L$ and $F^\atwo$, is a left adjoint between locally
  presentable categories; while by~\cite[Satz~14.6]{Gabriel1971Lokal},
  its right adjoint $G_1$ is accessible. So the projective lifting of
  $(\mathsf L, \mathsf R)$ along $F$ exists, and is accessible as its
  comonad $W_1 G_1$ is so.
\end{proof}

In fact, we may drop the requirement of local presentability from the
first part of the preceding proposition if we strengthen the
hypotheses on $U$.
\begin{Prop}
  \label{prop:7}
  Let $\D$ be a cocomplete category, let $\mathsf T$ be an accessible
  monad on $\D$ and let $(\mathsf L, \mathsf R)$ be an accessible
  \awfs on $\D$. Then the projective lifting of $(\mathsf L, \mathsf
  R)$ along the forgetful functor $U^\mathsf T \colon \Alg{T}
  \to \D$ exists and is accessible.
\end{Prop}
\begin{proof}
As in~\eqref{eq:39}, we form the pullback of double categories as to
the left in
\begin{equation*}
  \cd[@C+0.5em]{
    \dcat A \ar[r] \ar[d]_V \pushoutcorner & \DAlg{R}
    \ar[d]^{U^\mathsf R}\\
    \Sq{\Alg{T}} \ar[r]^-{\Sq{U^\mathsf T}} & \Sq{\D}
  } \qquad
  \cd{
    \A_1 \ar[r] \ar[d]_{V_1} & \Alg{R} \ar[d]\\
    (\Alg{T})^\atwo \ar[r]^-{(U^\mathsf T)^\atwo} & \D^\atwo
  } \qquad
  \cd{
    \mathsf S \ar@{<-}[r]^\sigma \ar@{<-}[d]_\nu & \mathsf R \ar@{<-}[d]^{!} \\
    \mathsf T^\atwo \ar@{<-}[r]^{!} & \id
  }
\end{equation*}
and must show that the underlying $1$-component $V_1$, as in the
centre, has a left adjoint. Now, the monad $\mathsf R$ on $\D^\atwo$
is accessible by assumption, while $\mathsf T^\atwo$ is accessible
since $\mathsf T$ is so; whence
by~\cite[Theorem~27.1]{Kelly1980A-unified} the coproduct $\mathsf S =
\mathsf T^\atwo + \mathsf R$ exists in $\cat{MND}(\D^\atwo)$ and is
accessible. This coproduct is equally the pushout on the right above,
and~\cite{Kelly1980A-unified} guarantees that the functor
$\Alg{(\thg)}$ sends this square to a pullback of categories, so that we
may identify $V_1$ with $\nu^\ast \colon \Alg{S} \to \Alg{T^\atwo} =
(\Alg{T})^\atwo$. Now as $\D^\atwo$ is cocomplete and $\mathsf S$ is
accessible, this $\nu^\ast$ has a left adjoint
by~\cite[Theorem~25.4]{Kelly1980A-unified}.
\end{proof}

A standard application of the preceding results is to the construction
of \awfs on diagram categories. If the cocomplete $\C$ bears an
accessible \awfs $(\mathsf L, \mathsf R)$, and $\I$ is a small
category, then the functor $U \colon \C^\I \to \C^{\ob \I}$ given by
precomposition with $\ob \I \to \I$ has a left adjoint (given by left
Kan extension) and is strictly monadic; so that by
Proposition~\ref{prop:7} we may projectively lift the pointwise \awfs
on $\C^{\ob \I}$ along $U$ to obtain the \emph{projective} \awfs on
$\C^\I$. If $\C$ is moreover complete, then $U$ also has a right
adjoint given by right Kan extension; if it is moreover locally
presentable then we may apply Proposition~\ref{prop:2}(b) to
injectively lift the pointwise \awfs along $U$, so obtaining the
\emph{injective} \awfs on $\C^\I$.

\section{Cofibrant generation by a category}
\label{sec:cofibrant-generation-1}
For our next application of Theorem~\ref{thm:recognition}, we use it
to give a simplified treatment of algebraic weak factorisation systems
\emph{cofibrantly generated by a small category of morphisms}. An
\awfs is cofibrantly generated in this sense when its
$\mathsf R$-algebras are precisely morphisms equipped with a choice of
liftings against some small category $\J \to \C^\atwo$ of ``generating
cofibrations''. This extends and tightens the familiar notion of
cofibrant generation of a weak factorisation system by a set of
generating cofibrations; and the main result
of~\cite{Garner2009Understanding} extends and tightens Quillen's small
object argument to show that, under commonly-satisfied conditions on
$\C$, the \awfs cofibrantly generated by any small $\J \to \C^\atwo$
exists. The proof given there is quite involved; we will give an
simpler one exploiting Theorem~\ref{thm:recognition}. First we recall
the necessary background.

\subsection{Lifting operations}
\label{sec:lifting-operations}
Given categories $U \colon \J \to \C^\atwo$ and $V \colon \K \to
\C^\atwo$ over $\C^\atwo$, a \emph{$(\J,\K)$-lifting operation} is a
natural family of functions $\phi_{j,k}$ assigning to each $j \in \J$
and $k \in \K$ and each commuting square
\begin{equation*}
\cd{\mathrm{dom}\ Uj \ar[d]_{Uj} \ar[r]^-{u} &  \mathrm{dom}\ Vk \ar[d]^{Vk}  \\
\mathrm{cod}\ Uj \ar@{.>}[ur]|{\phi_{j,k}(u,v)} \ar[r]^-{v}  & \mathrm{cod}\ Vk}
\end{equation*}
a diagonal filler as indicated making both triangles commute.
Naturality here expresses that the $\phi_{j,k}$'s are components of a
natural transformation
$ \phi \colon \C^{\atwo}(U\thg, V?) \Rightarrow \C(\mathrm{cod}\
U\thg, \mathrm{dom}\ V?) \colon \J^\op \times \K \to \Set$.
Thus Section~\ref{sec:lift} describes the canonical
$(\Coalg{L}, \Alg{R})$-lifting operation associated to any \awfs.

The assignation taking $(\J,\K)$ to the collection of $(\J,
\K)$-lifting operations is easily made the object part of a
functor $\cat{Lift} \colon (\CAT / \C^\atwo)^\op \times (\CAT /
\C^\atwo)^\op \to \SET$; we now have, as
in~\cite[Proposition~3.8]{Garner2009Understanding}:

\begin{Prop}
  \label{prop:5}
  Each functor $\cat{Lift}(\J, \thg)$ and $\cat{Lift}(\thg, \K)$ is
  representable, so that we induce an adjunction
 \begin{equation} \cd[@C+1em]{(\CAT/\C^\atwo)^\op
     \ar@<-4.5pt>[r]_-{{\Rl[(\thg)]}} \ar@{}[r]|-{\bot} &
     \CAT/\C^{\atwo}\rlap{ .} \ar@<-4.5pt>[l]_-{{\Ll[(\thg)]}} }\label{eq:33}
 \end{equation}
\end{Prop}
\begin{proof}
  The category $\Rl \to \C^{\atwo}$ representing $\cat{Lift}(\J,
  \thg)$, has as objects, pairs of $g \in \C^\atwo$ and $\phi_{\thg
    g}$ a $(\J, g)$-lifting operation; and as maps $(g, \phi_{\thg g})
  \to (h,\phi_{\thg h})$, those $g \to h$ of $\C^\atwo$ which commute
  with the lifting functions. Dually, the representing category $\Ll
  \to \C^\atwo$ for $\cat{Lift}(\thg, \K)$ comprises objects $f \in
  \C^\atwo$ equipped with $(f, \K)$-lifting operations, together with
  the maps of $\C^\atwo$ between them that commute with the lifting
  functions.
\end{proof}

We have seen a particular instance of this result in
Section~\ref{sec:algebr-coalg-from} above; the category
$\Rl[\Coalg{\mathsf L}]$ described there is the representing object
for $\cat{Lift}(\Coalg{L},\thg)$, and the functor $\bar \Phi \colon
\Alg{R} \to \Rl[\Coalg{\mathsf L}]$ is the one induced by this
representation from the canonical lifting operation of
Section~\ref{sec:lift}.
 
\subsection{Cofibrant generation}
\label{sec:cofibr-fibr-gener}

If the maps $g \colon C \to D$ and $h \colon D \to E$ of $\C$ are equipped
with lifting operations $\phi_{\thg g}$ and $\phi_{\thg h}$ against a
category $U \colon \J \to \C^\atwo$, then their composite $hg$ also
bears a lifting operation $\phi_{\thg hg}$, defined as in~\eqref{eq:7}
by $\phi_{j, hg}(u, v) = \phi_{j, g}(u, \phi_{j, h}(gu, v))$. This
composition, together with the equipment of an identity map with its
\emph{unique} lifting operation, provides the necessary vertical
structure to make $\Rl \to \C^\atwo$ into a concrete double category
$\Drl \to \Sq{\C}$; dually, we can make $\Ll \to \C^\atwo$ into a
concrete double category $\Dll \to \Sq{\C}$.

  We now define an \awfs $(\mathsf L, \mathsf R)$ on $\C$ to be \emph{cofibrantly
    generated} by a small $\J \to \C^\atwo$ if $\DAlg{R} \cong \Drl$
  over $\Sq{\C}$. If this isomorphism is verified for a $\J$ which is
  large, we say instead that $\C$ is \emph{class-cofibrantly
    generated} by $\J \to \C^\atwo$.
There are dual notions of fibrant or class-fibrant generation,
involving an isomorphism $\Dll \cong \DCoalg{L}$ over $\Sq{\C}$;
however, these are markedly less prevalent than their duals in
categories of mathematical interest.

\subsection{Existence of cofibrantly generated awfs}
In~\cite[Definition~3.9]{Garner2009Understanding} is given the notion
of an \awfs being ``algebraically-free'' on $U \colon \J \to
\C^\atwo$: to which the notion of cofibrant generation given above,
though apparently different in form, is in fact
equivalent.\footnote{This follows from Proposition~\ref{prop:15} below.}
Theorem~4.4 of \emph{ibid}.\ thus guarantees, among other things, that
in a locally presentable category $\C$, the \awfs cofibrantly
generated by \emph{any} small $U \colon \J \to \C^\atwo$ exists. We
now use Theorem~\ref{thm:recognition} to give a shorter proof of this.

\begin{Prop}
  \label{prop:6}
  If $\C$ is locally presentable then the \awfs $(\mathsf L, \mathsf
  R)$ cofibrantly generated by any small $U \colon \J \to \C^\atwo$
  exists; its underlying monad $\mathsf R$ is algebraically-free on a pointed
  endofunctor and accessible.
\end{Prop}

This result as stated is less general
than~\cite[Theorem~4.4]{Garner2009Understanding}, which also deals
with certain kinds of non-locally presentable $\C$. Though we have no
need for this extra generality here, let us note that reincorporating
it would simply be a matter of adapting the final paragraph of the
following proof.

\begin{proof}
  It is easy to see that $V \colon \Drl \to \Sq{\C}$ is
  right-connected; so by Theorem~\ref{thm:recognition}, it will
  comprise the double category of algebras for the desired \awfs as
  soon as $V_1 \colon \Rl \to \C^\atwo$ is strictly monadic. Now an
  object of $\Rl$ is equally a map $g$ of $\C$ together with a section
  $\phi_{\thg g}$ of the natural transformation
  \begin{equation}\label{eq:35}
    \psi_{\thg g} \colon \C(\mathrm{cod}\ U\thg, \mathrm{dom}\ g)
    \Rightarrow \C^\atwo(U\thg, g) \colon \J^\op \to \Set
  \end{equation}
  whose $j$-component sends $m \colon \mathrm{cod}\ Uj \to
  \mathrm{dom}\ g$ to
  $(m\cdot Uj, gm) \colon Uj \to g$; while a map
  $(g, \phi_{\thg g}) \to (h, \phi_{\thg h})$ of $\Rl$ is a map
  $g \to h$ of $\C^\atwo$ for which the induced
  $\psi_{\thg g} \to \psi_{\thg h}$ in $[\J^\op, \Set]^\atwo$ commutes
  with the sections. We thus have a pullback
  \begin{equation}
    \cd[@-0.4em]{
      \Rl \ar[d]_{V_1} \ar[r]^{\phantom{W}} \pushoutcorner & \cat{SplEpi}([\J^\op,
      \Set]) \ar[d] \\
      \C^\atwo \ar[r]^-{\psi} & [\J^\op, \Set]^\atwo\rlap{ .}}
    \label{eq:32}
  \end{equation}

  We first show that $\psi$ has a left adjoint. The composite
  $\mathrm{cod} \cdot \psi \colon \C^\atwo \to [\J^\op, \Set]$ is the
  singular functor $\C^\atwo(U, 1)$ which has left adjoint $F_1$ given
  by the left Kan extension of $U \colon \J \to \C^\atwo$ along the
  Yoneda embedding; while $\mathrm{dom} \cdot \psi$ is the functor
  $\C(\mathrm{cod}\cdot U, \mathrm{dom}) \cong \C^\atwo(\mathrm{id}
  \cdot \mathrm{cod} \cdot U, 1)$
  (as $\mathrm{id} \dashv \mathrm{dom} \colon \C^\atwo \to \C$) and so
  also has a left adjoint $F_2$. The natural transformation
  $\mathrm{dom} \cdot \psi \to \mathrm{cod} \cdot \psi$ between right
  adjoints transposes to one $\alpha \colon F_1 \to F_2$ between left
  adjoints, using which $\psi$ too has a left adjoint $K$ sending
  $f \colon X \to Y$ to the pushout of $F_1 f \colon F_1X \to F_1Y$
  along $\alpha_X \colon F_1X \to F_2X$.

  Now as we noted in Section~\ref{sec:split-epimorphisms} above,
  $\cat{SplEpi}([\J^\op, \Set])$ may be identified with the
  category of algebras for a pointed endofunctor $(T,
  \eta)$ on $[\J^\op, \Set]^\atwo$ with unit $\eta_f \colon f \to
  Tf$ given by~\eqref{eq:11}. It follows that $\J^\pitchfork$ is
  isomorphic over $\C^\atwo$ to the category of algebras for the
  pointed endofunctor $(P,\rho)$ in the pushout square 
  \begin{equation*}
    \cd[@C+0.5em]{
      K \psi \ar[d]_{\varepsilon} \ar[r]^{K\eta\psi} & K T \psi \ar[d] \\
      1 \ar[r]^-{\rho} & \pullbackcorner P}
  \end{equation*}
  (c.f.~\cite[Theorem~2.1]{Wolff1978Free}). An easy consequence of
  this identification is that $V_1$ strictly creates coequalisers of
  $V_1$-absolute coequaliser pairs; so it will be strictly monadic
  whenever it has a left adjoint---that is, whenever free $(P,
  \rho)$-algebras exist. We show this existence using results
  of~\cite{Kelly1980A-unified}.

  Since $\C$ is locally presentable, there is a regular cardinal
  $\kappa$ such that the domain and codomain of each $Uj$ is
  $\kappa$-presentable; thus the domain and codomain of each $\psi_{j
    \thg} \colon \C(\mathrm{cod}\ Uj,\, \mathrm{dom}\ \thg)
  \Rightarrow \C^\atwo(Uj, \thg)$ preserves $\kappa$-filtered
  colimits, so that $\psi$ does so too. Now $K$, being a left adjoint,
  is cocontinuous, while $T$ is so by inspection of~\eqref{eq:11};
  whence the $P$ of~\eqref{eq:32} is a pushout of functors preserving
  $\kappa$-filtered colimits, and so also preserves them. Thus
  by~\cite[Theorem~22.3]{Kelly1980A-unified}, free $(P,
  \rho)$-algebras exist, which is to say that $V_1$ has a left adjoint
  $F_1$ as required. Since $P$ preserves $\kappa$-filtered
  colimits, $V_1$ creates them, and so the induced monad $\mathsf R =
  V_1F_1$ preserves them; it is thus accessible. Finally, it follows
  from~\cite[Proposition~22.2]{Kelly1980A-unified} that $\mathsf R$ is
  the algebraically-free monad on the pointed endofunctor $(P, \rho)$.
\end{proof}
This result provides a rich source of algebraic weak factorisation
systems; in particular, we may make any cofibrantly generated weak
factorisation system on a locally presentable category into an \awfs
by applying the result to its set $J$ of generating cofibrations, seen
as a discrete category over $\C^\atwo$. For applications that make
serious use of the algebraicity so obtained,
see~\cite{Garner2009A-homotopy-theoretic,Garner2010Homomorphisms}; for
applications utilising the extra generality of cofibrant generation by
a small \emph{category}, rather than a small set of morphisms, see
\cite{Barthel2013Six-model,Batanin2013Multitensor}.

\subsection{Inadequacy of cofibrant generation by a category}
\label{sec:inad-cofibr-gener}
Proposition~\ref{prop:6} tells us that the monad of an \awfs
cofibrantly generated by a small category has two specific properties:
it is accessible, and it is algebraically-free on a pointed
endofunctor. Our experience of monads suggests that the former
property should be more common than the latter, and in fact this is
the case: many \awfs of practical interest verify the accessibility,
but not the freeness. For example:
\begin{Prop}\label{prop:10}
  The monad of the \awfs for lalis on $\Cat$ is accessible, but not
  algebraically-free on a pointed endofunctor; in particular, this
  \awfs is
  not cofibrantly generated by a small category.
\end{Prop}
\begin{proof}
  The monad $\mathsf R$ at issue is given by left Kan extension and
  restriction along the $2$-functor $j \colon \atwo \to \L$ of
  Section~\ref{sec:lalis};  it is thus cocontinuous and in particular, accessible.
  On the other hand, if it were algebraically-free on a pointed
  endofunctor, then its category of
  algebras---$\cat{Lali}(\Cat)$---would be isomorphic to the category
  of algebras for a pointed endofunctor, and as such would be
  retract-closed: meaning that, for any lali $f \dashv u$ and any
  retract $g$ of $f$ in $\Cat^\atwo$, there would exist lali structure
  on $g$. Now, to equip the unique functor $\A \to 1$ with the
  structure of a lali is to specify a terminal object of $\A$; so it
  is enough to describe a category $\A$ with a terminal object, and a
  retract $\B$ of $\A$ without one. To this end, take $\B$ to be the
  free idempotent $t^2 = t$ as on the left in:
\begin{equation*}
\cd{
0 \ar@(ul,dl)_{t}
} \quad \quad
\cd{{} \ar@{|->}[r]^{} & {}
} \quad \quad
\cd{
0 \ar@(ul,dl)_{t} \ar[r]^{!} & 1
} \quad \quad
\cd{{} \ar@{|->}[r]^{} & {}
} \quad \quad
\cd{
0 \ar@(ul,dl)_{t}
}
\end{equation*}
and take $\A$ to be $\B$ with a terminal object $1$ freely adjoined.
The inclusion of $\B$ into $\A$ admits a retraction, which fixes $t$
and sends $! \colon 0 \to 1$ to the identity on $0$. Yet $\B$ does not
admit a terminal object.
\end{proof}
In particular, this result tells us that, since lalis are not closed
under retracts, they cannot be characterised as the right class of
maps for a mere weak factorisation system; thus the algebraicity is,
in this case, essential.

\section{Cofibrant generation by a double category}
\label{sec:cofibr-gener-double}
In light of Proposition~\ref{prop:10}, it is natural to ask whether
there is a more refined notion of cofibrant generation which
encompasses such examples of \awfs as the one for lalis on $\Cat$. In
this section, we describe such a notion; it involves lifting
properties against a small \emph{double} category, rather than a small
category, of generating cofibrations. 

\subsection{Double-categorical lifting operations}
\label{sec:double-categ-lift}
Let $U \colon \dcat J \to \Sq{\C}$ and $V \colon \dcat K \to
\Sq{\C}$ be double categories over $\Sq{\C}$. We define a
\emph{$(\dcat J,\dcat K)$-lifting operation} to be a $(\J_1,
\K_1)$-lifting operation in the sense of
Section~\ref{sec:lifting-operations} which is also compatible
  with vertical composition in $\mathbb J$ and $\mathbb K$, in the
  sense that
\begin{equation}
  \begin{aligned}
    \phi_{j, \ell \cdot k}(u, v) &= \phi_{j, k}(u, \phi_{j, \ell}(Vk
    \cdot u, v)) \\ \text{and} \quad \phi_{j \cdot i, k}(u, v) &=
    \phi_{j, k}(\phi_{i, k}(u, v \cdot Uj), v)
  \end{aligned}
\label{eq:31}\end{equation}
for all vertically composable maps $j \cdot i \colon A \to B \to C$
and $\ell \cdot k \colon D \to E \to F$ in $\mathbb J$ and in $\mathbb
K$. For example, by virtue of~\eqref{eq:7} and its dual, the lifting
operation~\eqref{eq:4} associated to an \awfs is an $(\DCoalg{L},
\DAlg{R})$-lifting operation. As before, the assignation sending
$\dcat J$ and $\dcat K$ to the collection of $(\dcat J, \dcat
K)$-lifting operations underlies a functor $\dcat{Lift} \colon (\DBL /
\Sq{\C})^\op \times (\DBL / \Sq{\C})^\op \to \SET$; and also as
before, we have:
\begin{Prop}
  \label{prop:1}
  Each functor $\dcat{Lift}(\dcat J, \thg)$ and $\dcat{Lift}(\thg,
  \dcat K)$ is representable, so that the adjunction~\eqref{eq:33}
  extends to one
\begin{equation} \cd[@C+1em]{(\DBL/\Sq{\C})^\op
  \ar@{}[r]|-{\bot} \ar@<-4.5pt>[r]_-{{\Drl[(\thg)]}} & \DBL/\Sq{\C}\rlap{ .}
  \ar@<-4.5pt>[l]_-{{\Dll[(\thg)]}}}\label{eq:34}
\end{equation}
\end{Prop}
\begin{proof}
  Given $U \colon \dcat J \to \Sq{\C}$, we have as in
  Section~\ref{sec:cofibr-fibr-gener} the concrete double category
  $\Drl[\J_1] \to \Sq{\C}$. The representing object $\Drl[\dcat J] \to
  \Sq{\C}$ for $\dcat{Lift}(\dcat J, \thg)$ is the sub-double category
  of $\Drl[\J_1]$ with the same objects and horizontal arrows, and
  just those vertical arrows $(g, \phi_{\thg g})$ whose lifting
  operations respect vertical composition in $\dcat J$, together with
  all cells between them. The representing object $\Dll[\dcat K] \to
  \Sq{\C}$ for $\dcat{Lift}(\thg, \dcat K)$ is defined dually.
\end{proof}

Note that the $2$-functor $(\thg)_1 \colon \DBL \to \CAT$ sending a double
category to its vertical category is represented by the free vertical
arrow $\atwo_v$, and consequently has a left adjoint, sending $\C$ to
the product of $\atwo_v$ with the free \emph{horizontal} double
category on $\C$. This lifts to a left adjoint $\dcat F$ for the
induced functor $(\thg)_1 \colon \DBL / \Sq{\C} \to \CAT / \C^\atwo$
on slice categories; and as there are no non-trivial vertical
composites in $\dcat F\J$, we have that in fact $\Drl = \Drl[(\dcat F
\J)]$; thus the double category structure on $\Drl$, for which we
offered no abstract justification in
Section~\ref{sec:cofibr-fibr-gener} above, is now explained in terms
of the adjunction~\eqref{eq:34}.

\subsection{Cofibrant generation by a double category}
\label{sec:cofibr-fibr-gener-1}
An algebraic weak factorisation system $(\mathsf L, \mathsf R)$ is 
said to be \emph{cofibrantly generated} by a small double category
$\dcat J \to \Sq{\C}$ if $\DAlg{R} \cong \Drl[\dcat J]$ over
$\Sq{\C}$; if we have this isomorphism for a large $\dcat J$, then we
call $(\mathsf L, \mathsf R)$ \emph{class-cofibrantly} generated by
$\dcat J \to \Sq{\C}$. Once again, we have the dual notions of fibrant
and class-fibrant generation by a double category.

If, as in the preceding section, we identify a category over
$\C^\atwo$ with the free double category thereon, then this definition
is a conservative extension of that given in
Section~\ref{sec:cofibr-fibr-gener} above. However, by contrast with
Proposition~\ref{prop:10}, we have:

\begin{Prop}
  \label{prop:11}
  The \awfs for lalis on $\Cat$ is cofibrantly generated by a small
  double category.
\end{Prop}
\begin{proof}
  We first claim that to equip a functor $f \colon A \to B$ with lali
  structure is equally to give:
   \begin{itemize}
   \item A section $u \colon \ob B \to \ob A$
   of the action of $f$ on objects; and
   \item For each $a \in A$ and $b \in B$, a section $\gamma_{a} \colon
     B(fa,b) \to A(a,ub)$ of the action of $f$ on morphisms; such that
   \item $\gamma_{a}(\beta \cdot
   f\alpha) = \gamma_{a'}(\beta) \cdot \alpha$ for all $\alpha \colon a \to
   a'$ in $A$ and $\beta \colon fa' \to b$ in $B$; and
   \item $\gamma_{ub}(1_b) = 1_{ub}$ for all $b \in B$.
 \end{itemize}
 Indeed, if $f$ is part of a lali $f \dashv p$ with unit $\eta$, then
 we obtain these data by taking $u$ to be the action of $p$ on
 objects, and taking $\gamma_a(\beta) = p(\beta) \cdot \eta_a$.
 Conversely, given these data, we define a section $p$ of $f$ on
 objects by $p(b) = u(b)$ and on morphisms by $p(\beta \colon b \to
 b') = \gamma_{ub}(\beta) \colon ub \to ub'$, and define a unit $\eta
 \colon 1 \Rightarrow pf$ with components $\eta_a = \gamma_a(1_{fa})$.
 This proves the claim.

 We now define a small $V \colon \dcat J \to \Sq{\Cat}$ for which
 $\Drl[\dcat J] \cong \dcat{Lali}(\Cat)$. The objects of $\dcat J$ are the
 ordinals $\cat 0$, $\cat 1$, $\cat 2$ and $\cat 3$, and its
 horizontal arrows are order-preserving maps; on these, $V$ acts as
 the identity. The vertical arrows of $\dcat J$ are freely generated
 by morphisms $k \colon \cat 0 \to \cat 1$, $\ell \colon \cat 1 \to
 \cat 2$ and $m \colon \cat 2 \to \cat 3$, which are sent by $V$ to
 the appropriate initial segment inclusions between ordinals. Its
 squares are freely generated by the following four:
\[
\cd{
\cat 0 \ar[d]_k \ar[r]^{!} \ar@{}[dr]|{\text{(a)}} & \cat 1 \ar[d]^\ell \\
\cat 1 \ar[r]_{\delta_0} & \cat 2} \qquad
\cd{
\cat 1 \ar[d]_\ell \ar[r]^{\delta_0} \ar@{}[dr]|{\text{(b)}} & \cat 2 \ar[d]^m \\
\cat 2 \ar[r]_{\delta_0} & \cat 3} \qquad
\cd{
\cat 1 \ar[d]_\ell \ar[r]^{\delta_1} \ar@{}[dr]|{\text{(c)}} & \cat 2 \ar[d]^m \\
\cat 2 \ar[r]_{\delta_1} & \cat 3}\qquad
\cd{
{\cat{0}} \ar[d]_k \ar[r]^{\id} \ar@{}[ddr]|{\text{(d)}} & {\cat 0} \ar[dd]^{k} \\
{\cat{1}} \ar[d]_\ell \\
{\cat{2}} \ar[r]_{!} & {\cat 1}\rlap{ ,}
}
\]
where (following standard simplicial notation) $\delta_0$ and
$\delta_1$ denote the order-preserving injections omitting $0$ and $1$
respectively. We claim that $\dcat J \to \Sq{\Cat}$ cofibrantly
generates the \awfs for lalis; in other words, that $\Drl[\dcat J] \cong
\dcat{Lali}(\Cat)$ over $\Sq{\Cat}$. Indeed, given a functor $f
\colon A \to B$, we see that:
\begin{itemize}
\item To equip $f$ with a $Vk$-lifting operation is to give a section
  $u$ of its action on objects;\vskip0.25\baselineskip
\item To equip $f$ with a $V\ell$-lifting operation is to give, for
  every $a \in A$ and $\beta \colon fa \to b$ in $B$, a map
  $\gamma_a(\beta) \colon a \to \bar b$ in $A$ over $\beta$; to ask for
  the compatibility (a) with the $Vk$-lifting operation is to ask
  that, in fact, $\bar b = ub$.\vskip0.25\baselineskip
\item To equip $f$ with a $Vm$-lifting operation is to give,
  for every $\alpha \colon a \to a'$ in $A$ and $\beta \colon fa' \to
  b$, a map $\psi_{\alpha}(\beta) \colon a' \to \bar b$ in $A$ over $\beta$. The
  compatibility (b) now forces $\psi_{\alpha}(\beta) =
  \gamma_{a'}(\beta)$, while that in (c) forces $\psi_{\alpha}(\beta)
  \cdot \alpha = \gamma_a(\beta \cdot f\alpha)$; so $\psi$ is
  determined by $\gamma$, and we have $\gamma_a(\beta \cdot f\alpha) =
  \gamma_{a'}(\beta) \cdot \alpha$.
\item Finally, the compatibility (d) forces $\gamma_{ub}(1_{b}) =
  1_{ub}$.
\end{itemize}
This verifies that vertical arrows of $\Drl[\dcat J]$ are lalis in
$\Cat$; similar arguments show that squares of $\Drl[\dcat J]$ are
ones preserving the lali structure, and that composition is
composition of lalis, as required.
  \end{proof}
 \subsection{Canonical class-cofibrant generation}
 \label{sec:canon-class-cofibr}
As further evidence for the adequacy of the notion of
double-categorical cofibrant generation, we have the following result,
which tells us that any \awfs is class-cofibrantly generated by its
(typically large) double category of coalgebras.
\begin{Prop}
  \label{prop:3}
  Any \awfs is class-cofibrantly generated by $\DCoalg{L} \to \Sq{\C}$
  and class-fibrantly generated by $\DAlg{R} \to \Sq{\C}$. 
\end{Prop}
\begin{proof}
  By duality, we need prove only the first statement. As above,
  \eqref{eq:4} is an $(\DCoalg{L}, \DAlg{R})$-lifting operation, and
  so by Proposition~\ref{prop:1} corresponds to a concrete double
  functor $\Lambda \colon \DAlg{R} \to \Drl[\DCoalg{L}]$ over $\C$. Of
  course, $\Lambda$ is the identity on objects and horizontal arrows;
  while its component $\Lambda_1 \colon \Alg{R} \to
  (\Drl[\DCoalg{L}])_1$ on vertical arrows and squares postcomposes
  with the inclusion $(\Drl[\DCoalg{L}])_1 \hookrightarrow
  \Rl[\Coalg{L}]$ to yield the $\bar \Phi$ of Lemma~\ref{lem:2}. As
  $\bar \Phi$ is injective on objects and fully faithful, so too is
  $\Lambda_1$; while since the left square of~\eqref{eq:13} is one of
  $\DCoalg{L}$, the condition~\eqref{eq:31} defining the objects
  in~$(\Drl[\DCoalg{L}])_1$ implies that in~\eqref{eq:37}, so that
  $\Lambda_1$ is also surjective on objects, and thus an isomorphism.
\end{proof}

The fact of an isomorphism $\DCoalg{L} \cong \Dll[\DAlg{R}]$ is often
a useful tool for calculation. Many \awfs that arise in practice are
cofibrantly generated by double categories in a natural way; as such,
we have a concrete understanding of the $\mathsf R$-algebras. The
above isomorphism offers one technique for obtaining the corresponding
$\mathsf L$-coalgebras without explicit calculation of the
values of $L$ or $R$.

\subsection{Freeness of cofibrantly generated awfs}

Above, we have defined cofibrant generation of an \awfs in terms of a
universal property of its double category of algebras. We conclude
this section by showing that this implies another universal property:
that its double category of coalgebras is \emph{freely generated} by
the given $\dcat J$ with respect to left adjoint functors between
concrete double categories. In the case of cofibrant generation by a
mere category, this was shown
in~\cite[Theorem~6.22]{Riehl2011Algebraic}; our result generalises
this, and simplifies the proof.

The key step will be to extend the adjunction~\eqref{eq:34} to account
for change of base. To this end, we consider the $2$-category $\DBL /
\Sq{\thg\,_{\mathrm{ladj}}}$, whose objects are double functors $\dcat
J \to \Sq{\C}$, whose $1$-cells are squares as on the left of
\begin{equation}
  \cd[@R+0.4em@C+1em@-0.5em]{
    \dcat J \ar[d]_{U} \ar[r]^{\bar F} & \dcat J'
    \ar[d]^{U'} \\
    \Sq{\C} \ar[r]_-{\Sq{F}} & \Sq{\D}
  } \qquad \qquad \qquad
  \cd[@R+0.4em@C+2em]{
    \dcat J \ar[d]_{U} \ar@/^0.8em/[r]^(0.475){\bar F} \dtwocell{r}{\bar \alpha}
    \ar@/_0.8em/[r]_(0.525){\bar{F}'} & \dcat J'
    \ar[d]^{U' }\\
    \Sq{\C} \ar@/^0.8em/[r]^(0.475){\Sq{F}}
    \dtwocell[0.4]{r}{\Sq{\alpha}} \ar@/_0.8em/[r]_(0.525){\Sq{F'}}&
    \Sq{\D}
  }
\label{eq:44}
\end{equation}
together with a chosen right adjoint $G$ for $F$, and whose $2$-cells
are diagrams as on the right above. We define the $2$-category $\DBL /
\Sq{\thg\,_{\mathrm{radj}}}$ dually.

\begin{Prop}
  \label{prop:16}
  The adjunction~\eqref{eq:34} extends to a $2$-adjunction
  \begin{equation} \cd[@C+1em]{(\DBL /
      \Sq{\thg\,_{\mathrm{ladj}}})^{\co\op} \ar@{}[r]|-{\bot}
      \ar@<-4.5pt>[r]_-{{\Drl[(\thg)]}} & \DBL /
      \Sq{\thg\,_{\mathrm{radj}}}\rlap{ ,}
      \ar@<-4.5pt>[l]_-{{\Dll[(\thg)]}}}\label{eq:43}
  \end{equation}
  with respect to which the canonical isomorphisms $\Lambda \colon
  \DAlg{R} \to \Drl[\DCoalg{L}]$ and $\bar \Lambda \colon \DCoalg{L}
  \to \Dll[\DAlg{R}]$ of Proposition~\ref{prop:3} are $2$-natural.
\end{Prop}

\begin{proof}
  The $2$-functor $\Drl[(\thg)]$ is defined as before on objects, and
  as before on $1$-cells~\eqref{eq:44} whose bottom edge is an
  identity. To complete the definition on $1$-cells, it thus suffices
  to consider squares whose \emph{top} edge is an identity, as on the
  left in
  \begin{equation}\label{eq:48}
    \cd[@R+0.4em@C+1em@-0.5em]{
      \dcat J \ar[d]_{U} \ar[r]^{1} & \dcat J
      \ar[d]^{\Sq{F} \cdot U} \\
      \Sq{\C} \ar[r]^{\Sq{F}} & \Sq{\D}
    } \qquad \qquad \qquad
    \cd[@R+0.4em@C+1em@-0.5em]{
      \Drl[(F\dcat J)] \ar[d] \ar[r]^{} \pushoutcorner & \Drl[\dcat J]
      \ar[d] \\
      \Sq{\D} \ar[r]^{\Sq{G}} & \Sq{\C}\rlap{ .}
    }
  \end{equation}
  Let $F\dcat J$ denote the double category $\dcat J \to \Sq{\D}$ over
  $\Sq{\D}$ down the right edge of this square, and let $G$ be the
  chosen right adjoint for $F$. We claim that there is in fact a
  \emph{pullback} square as displayed right above: which we will take
  to be the value of $\Drl[(\thg)]$ at the left-hand square. For this,
  we must show that for each $g \colon C \to D$ of $\D$, the $(F \dcat
  J, g)$-lifting operations $\phi_{\thg, g}$ correspond functorially
  with $(\dcat J, Gg)$-lifting operations $\bar \phi_{\thg, Gg}$. But
  such lifting operations are given by natural transformations
  \begin{gather*}
    \phi_{\thg, g} \colon \C(\mathrm{cod}\ FU\thg, \mathrm{dom}\ g)
    \Rightarrow \C^\atwo(FU\thg, g) \colon \J^\op \to \Set\\
    \bar \phi_{\thg, {Gg}} \colon \C(\mathrm{cod}\ U\thg,
    \mathrm{dom}\ Gg) \Rightarrow \C^\atwo(U\thg, Gg) \colon \J^\op
    \to \Set
  \end{gather*}
  satisfying axioms; so as $F \dashv G$, these are in a clear
  bijective correspondence, under which the maps of lifting operations
  also correspond as required.

  Finally, we must define $\Drl[(\thg)]$ on $2$-cells. Given a diagram
  as on the right of~\eqref{eq:44}, the $2$-cell $\alpha \colon F
  \Rightarrow F'$ therein transposes under adjunction to one $\beta
  \colon G' \Rightarrow G$. As the forgetful double functors
  $\Drl[(\dcat J')] \to \Sq{\D}$ and $\Drl[\dcat J] \to \Sq{\C}$ are
  concrete, there is a unique possible $2$-cell $\bar \beta \colon
  \Drl[(\bar F')] \Rightarrow \Drl[\bar F]$ lifting $\beta$; a short
  calculation shows that this lifting indeed exists, so that we may
  take $\Drl[(\alpha, \bar \alpha)] = (\beta, \bar \beta)$. This
  completes the definition of $\Drl[(\thg)]$; that of $\Dll[(\thg)]$
  is dual.

  To verify that the extended $\Drl[(\thg)]$ and $\Dll[(\thg)]$ are
  $2$-adjoint, we observe that for $\dcat J$ and $\dcat K$ over
  $\Sq{\C}$ and $\Sq{\D}$, a map $\dcat J \to \Dll[\dcat K]$ in $\DBL
  / \Sq{\thg\,_{\mathrm{ladj}}}$ is an adjunction $F \dashv G \colon
  \D \to \C$ together with a $(F\dcat J, \dcat K)$-lifting operation,
  while a map $\dcat K \to \Drl[\dcat J]$ in $\DBL /
  \Sq{\thg\,_{\mathrm{radj}}}$ is an adjunction $F \dashv G$ together
  with a $(\dcat J, G\dcat K)$-lifting operation; by the above
  argument, these data are in bijective correspondence. We argue
  similarly for the bijection on $2$-cells. Finally, the naturality of
  the isomorphisms $\Lambda$ and $\bar \Lambda$ is
  exactly~\eqref{eq:46}; their $2$-naturality now follows by
  concreteness.
\end{proof}

Using the result, we now give our promised generalisation
of~\cite[Theorem~6.22]{Riehl2011Algebraic}.
\begin{Prop}
  \label{prop:15}
  If $(\mathsf L, \mathsf R)$ is (class-)cofibrantly generated by $U
  \colon \dcat J \to \Sq{\C}$, then it provides a left $2$-adjoint at
  $\dcat J$ for $\DCoalg{(\thg)} \colon \Ladj \to \DBL/\Sq{\thg
    \,_\mathrm{ladj}}$.
\end{Prop}
\begin{proof}
  We have natural isomorphisms
  \begin{align*}
    \Ladj((\mathsf L, \mathsf R),\, (\mathsf L', \mathsf R')) &
\cong \Radj((\mathsf L', \mathsf R'),\, (\mathsf L, \mathsf R))^\op \\ & 
    \cong \DBL/\Sq{\thg
    \,_\mathrm{radj}}(\DAlg{R'}, \DAlg{R})^\op \\ & \cong \DBL/\Sq{\thg
    \,_\mathrm{radj}}(\DAlg{R'},
    \Drl[\dcat J])^\op\\
    & \cong \DBL/\Sq{\thg
    \,_\mathrm{ladj}}(\dcat J, \Dll[\DAlg{R'}]) \\ & \cong \DBL/\Sq{\thg
    \,_\mathrm{ladj}}(\dcat J, \DCoalg{L'})\rlap{ ,}
  \end{align*}
  by, respectively, the isomorphism $\Ladj \cong \Radj^{\co\op}$; full
  fidelity of $\DAlg{(\thg)}$; the assumed isomorphism $\Drl[\dcat J]
  \cong \DAlg{R}$ over $\Sq{\C}$; adjointness~\eqref{eq:43}; and the
  natural isomorphisms of Proposition~\ref{prop:3}. 
\end{proof}

\section{A characterisation of accessible \awfs}\label{sec:char-access-awfs}
We are now ready to give our second main result,
Theorem~\ref{thm:accessible}, which gives a characterisation of the
accessible \awfs on a locally presentable category. We will show that,
for a locally presentable $\C$, any small double category over
$\Sq{\C}$ cofibrantly generates an accessible \awfs, and, moreover,
that every accessible \awfs on $\C$ arises in this way.

\subsection{Existence of cofibrantly generated AWFS}
\label{sec:exist-cofibr-gener}
We begin by extending Proposition~\ref{prop:6} from cofibrant
generation by categories to that by double categories.

\begin{Prop}\label{prop:4}
  If $\C$ is locally presentable, then the \awfs cofibrantly generated
  by any small double category $U \colon \dcat J \to \Sq{\C}$ exists
  and is accessible.
\end{Prop}
\begin{proof}
  Write $\J_2 = \J_1 \times_{\J_0} \J_1$ for the category of vertically
  composable pairs in $\mathbb J$, and $m \colon \J_2 \to \J_1$ for
  the vertical composition. Now the triangle on the left in
  \begin{equation*}
    \cd[@-0.5em]{\J_{2} \ar[rr]^{m}  \ar[dr]_{U_1 m} & & \J_{1} \ar[dl]^{U_1} \\
      & \C^{\atwo} }
    \qquad\qquad\quad
    \cd[@-0.6em@C-0.5em]{\Drl[\J_1] \ar[rr]^{\Drl[m]} \ar[dr]_{} & & \Drl[\J_2] \ar[dl]^{} \\
      & \Sq{\C}}
  \end{equation*}
  induces a morphism $\Drl[m]$ of concrete double categories over $\C$
  as on the right, which on vertical arrows sends $(g, \phi_{\thg g})
  \colon D \to E$ to $(g, \psi_{\thg g}) \colon D \to E$, with $(\J_2,
  g)$-lifting operation $\psi_{\thg g}$ defined as on the left in
  \[
  \psi_{(j, i), g}(u,v) = \phi_{j \cdot i, g}(u,v) \qquad \qquad
  \theta_{(j, i), g}(u,v) = \phi_{j, g}(\phi_{i, g}(u, v \cdot Uj),
  v)\rlap{ .}
  \]
  We also have the $(\J_2, g)$ lifting operation $\theta_{\thg g}$ as
  on the right above, and the assignation $(g, \phi_{\thg g}) \mapsto
  (g, \theta_{\thg g})$ in fact gives the action on vertical arrows of
  a further double functor $\delta_{\mathbb J} \colon \Drl[\J_1] \to
  \Drl[\J_2]$ concrete over $\C$. It is easy to see that we now have
  an equaliser of concrete double categories over $\C$ as on the left
  in
  \begin{equation}\label{eq:equaliser}
    \cd[@-0.5em@C-0.5em]{ \Drl[\dcat J] \ar[rr] \ar[drr] & & \Drl[\J_1] \ar[d] \ar@<3pt>[rr]^{\Drl[m]} \ar@<-3pt>[rr]_{\delta_{\dcat J}} & & \Drl[\J_2] \ar@<3pt>[dll] \\
      & & \Sq{\C} } \qquad \qquad \cd{
      \Rl[\J_1] \ar@<2.5pt>[r]^{\Rl[m]} \ar@<-2.5pt>[r]_{(\delta_{\dcat J})_1} \ar[d]_{\cong} & \Rl[\J_2] \ar[d]^{\cong} \\
      \Alg{R_1} \ar@<2.5pt>[r]^{s^\ast} \ar@<-2.5pt>[r]_{t^\ast} & \Alg{R_2}\rlap{ .}}
  \end{equation}
  Since $\Drl[\J_1]$ and $\Drl[\J_2]$ are right-connected, so too will
  $\Drl[\dcat J]$ be; whence by Theorem~\ref{thm:recognition}, the
  \awfs cofibrantly generated by $\dcat J$ will exist if and only if
  $(\Drl[\dcat J])_{1} \to \C^{\atwo}$ is strictly monadic. As $\J_1$
  and $\J_2$ are small, Proposition~\ref{prop:6} ensures that there
  are \awfs $(\mathsf{L_1},\mathsf{R_1})$ and
  $(\mathsf{L_2},\mathsf{R_2})$ on $\C$ with $\Rl[\J_1] \cong
  \Alg{R_1}$ and $\Rl[\J_2] \cong \Alg{R_2}$ over $\C^\atwo$; now by
  full fidelity of the assignment $\cat{Mnd}(\C^{\atwo})^{\op} \to
  \CAT/\C^{\atwo}$, there are unique monad maps $s,t \colon \mathsf
  R_{2} \rightrightarrows \mathsf R_{1}$ rendering commutative the
  diagram above right. It follows that $(\Drl[\dcat J])_1$ is
  isomorphic over $\C^{\atwo}$ to the equaliser of the lower row.
  Since $\mathsf R_1$ and $\mathsf R_2$ are accessible, the parallel
  pair of monad morphisms $s,t$ admits
  by~\cite[Theorem~27.1]{Kelly1980A-unified} a coequaliser $q \colon
  \mathsf R_1 \to \mathsf R$ which is again accessible, and for which
  $q^\ast$ is the equaliser of $s^\ast$ and $t^\ast$. Thus
  $(\Drl[\dcat J])_1 \cong \Alg{R}$ over $\C^{\atwo}$ so that
  $(\Drl[\dcat J])_1 \to \C^{\atwo}$ is strictly monadic for an
  accessible monad, as required.
\end{proof}
We now wish to show that every accessible \awfs on a locally
presentable category arises in the manner of the preceding
proposition. The key idea is to show that any accessible \awfs has a
small dense subcategory of $\mathsf L$-coalgebras, and then deduce the
result using Proposition~\ref{prop:3} and the following lemma:

\begin{Lemma}\label{lem:3}
  Given a morphism of categories over $\C^{\atwo}$
  \begin{equation*}
    \cd[@-0.85em]{\J \ar[rr]^{F} \ar[dr]_{U} & & \K\rlap{ ,} \ar[dl]^{V} \\
      & \C^{\atwo}
    }
  \end{equation*}
  the induced $\Rl[F] \colon \Rl[\K] \to \Rl[\J]$ is an isomorphism of
  categories whenever the identity $2$-cell exhibits $V$ as the
  pointwise Kan extension $\Lan_F U$. In particular, this is the case
  whenever $F$ is dense and $V$ preserves the colimits exhibiting this
  density.
\end{Lemma}
\begin{proof}
  As in the proof of Proposition~\ref{prop:6}, an element of
  $\K^{\pitchfork}$ is a pair $(g,\phi_{\thg g})$ where $g \colon C
  \to D$ in $\C$ and $\phi_{\thg g}$ is a section of the $\psi_{\thg
    g} \colon \C(\mathrm{cod}\ V\thg, C) \to \C^\atwo(V\thg, g)$
  of~\eqref{eq:35}; or equally (as $\mathrm{cod} \dashv \id \colon
  \C^\atwo \to \C$) a section of $(1, g) \cdot (\thg) \colon
  \C^\atwo(V\thg, 1_C) \to \C^\atwo(V\thg, g)$. In these terms, a map
  $(g, \phi_{\thg g}) \to (h, \phi_{\thg h})$ of $\Rl[\K]$ is a map $g
  \to h$ of $\C^\atwo$ for which the induced $\C^{\atwo}(V\thg,g) \to
  \C^{\atwo}(V\thg,h)$ commutes with the given sections. The map
  $\Rl[\K] \to \Rl$ is induced by the restriction functor $[F^{\op},1]
  \colon [\K^{\op},\Set] \to [\J^{\op},\Set]$, and sends
  $(g, \phi_{\thg g})$ to the pair $(g,\phi')$ with second component
  \[
  \phi' = \phi_{\thg g} F \colon \C^{\atwo}(U\thg,g) \to \C^{\atwo}(U
  \thg,1_C)\rlap{ .}
  \]
  So to prove $\Rl[F] \colon \Rl[\K] \to \Rl$ is an isomorphism, it
  will suffice to show that $[F^{\op},1]$ becomes fully faithful when
  restricted to the full subcategory of $[\K^\op, \Set]$ on the
  presheaves of the form $\C^{\atwo}(V\thg,g)$. Now, to say that the
  identity $2$-cell exhibits $V$ as $\mathrm{Lan}_F U$ is equally to
  say that it exhibits $V^{\op}$ as $\mathrm{Ran}_{F^\op} U^\op$; the
  continuous $\C^\atwo(\thg, g) \colon (\C^\atwo)^\op \to \Set$
  preserves this pointwise Kan extension, and so the identity $2$-cell
  exhibits $\C^{\atwo}(V\thg,g)$ as the right Kan extension of
  $\C^{\atwo}(U\thg ,g)$ along $F^{\op}$. The universal property of
  this right Kan extension now immediately implies that the action of
  $[F^\op, 1]$ on morphisms
  \[
  [\K^{op},\Set](\C^{\atwo}(V\thg,f),\C^{\atwo}(V\thg,g)) \to
  [\J^{op},\Set](\C^{\atwo}(U\thg,f),\C^{\atwo}(U\thg,g))\] is
  invertible for any $f$ and $g$ in $\C^\atwo$, as required. Finally,
  note that if $F$ is dense, then the identity $2$-cell exhibits
  $1_\K$ as $\mathrm{Lan}_F F$, and that if $V$ preserves the colimits
  exhibiting this density, then it preserves this left Kan extension,
  so that the identity $2$-cell exhibits $V$ as $\Lan_F U$.
\end{proof}

\begin{Thm}\label{thm:accessible}
  An \awfs on a locally presentable category is accessible if and only
  if it is cofibrantly generated by some small double category.
\end{Thm}
\begin{proof}
  One direction is Proposition~\ref{prop:4}; for the other, let $(\mathsf
  L, \mathsf R)$ be an accessible \awfs on the locally presentable
  $\C$. By Proposition~\ref{prop:3} we have an isomorphism
  $\Drl[\DCoalg{L}] \cong \DAlg{R}$ over $\Sq{\C}$; it will thus
  suffice to find a small sub-double category $\DCoalg{L_{\lambda}}
  \hookrightarrow \DCoalg{L}$ with the induced 
  $\Drl[\DCoalg{L}] \to \Drl[\DCoalg{L_{\lambda}}]$ invertible.
  To this end, we consider the following diagram, whose left column is
  yet to be defined:
\[
\cd[@C+2em@-0.5em]{
\Coalg{L_\lambda}_2 \ar[r]^-{j_2} \ar[d] \ar@<-8pt>[d] \ar@<8pt>[d] &
\Coalg{L}_2 \ar[r]^-{V_2} \ar[d]|{m_\mathsf L}
\ar@<-8pt>[d]_{p_\mathsf L} \ar@<8pt>[d]^{q_\mathsf L} &
\C^\athree \ar[d]|{m} \ar@<-8pt>[d]_{p} \ar@<8pt>[d]^{q} \\
\Coalg{L_\lambda} \ar[r]^-{j_1} \ar@{<-}[d] \ar@<-8pt>[d] \ar@<8pt>[d] &
\Coalg{L} \ar[r]^-{V_1} \ar@{<-}[d]|{i_\mathsf L}
\ar@<-8pt>[d]_{d_\mathsf L} \ar@<8pt>[d]^{c_\mathsf L} &
\C^\atwo \ar@{<-}[d]|{i} \ar@<-8pt>[d]_{d} \ar@<8pt>[d]^{c} \\
\C_\lambda \ar[r]^-{j_0} & \C \ar[r]^-1 & \C\rlap{\ \ \text.}
}
\]
As colimits are pointwise in functor categories the right column is
composed of cocontinuous functors between cocomplete categories; since
$V_1$ creates colimits, the same is true for $\Coalg{L}$ and each of
$d_\mathsf L$, $i_\mathsf L$ and $c_\mathsf L$. Moreover $d_\mathsf L$
has the isomorphism-lifting property since $V_1$ and $d$ do, whence
by~\cite{Joyal1993Pullbacks} the pullback $\Coalg{L}_{2}$ of
$d_\mathsf L$ along $c_\mathsf L$ is also a bipullback. Consequently,
$\Coalg{L}_{2}$ is cocomplete and the projections $q_{\mathsf L}$ and
$p_{\mathsf L}$ cocontinuous; so too is $m_\mathsf L$, since
$d_\mathsf L$ and $c_\mathsf L$ jointly create colimits; and similarly
for $V_{2}$. Thus the right two columns are composed of
cocomplete categories and cocontinuous functors.

Now $\C$ is locally presentable, whence also $\C^\atwo$ and
$\C^\mathsf 3$; and since $\mathsf L$ is accessible, $\Coalg{L}$ is
locally presentable as in the proof of Proposition~\ref{prop:2}.
Finally, $\Coalg{L}_2$ is a bipullback of cocontinuous functors
between locally presentable categories, and so itself locally
presentable by~\cite[Theorem~3.15]{Bird1984Limits}. So the right two
columns are composed of locally presentable categories and
cocontinuous functors.

Any cocontinuous functor between locally presentable categories has an
accessible right adjoint~\cite[Satz~14.6]{Gabriel1971Lokal}; thus
there is a regular $\lambda$ so that the right two columns are
composed of locally $\lambda$-presentable categories and functors with
$\lambda$-accessible right adjoints; it follows that each left adjoint
preserves $\lambda$-presentable objects. We now define
$\DCoalg{L_{\lambda}}$\begin{footnote}{Our notation
    $\DCoalg{L_{\lambda}}$ acknowledges the fact that this is equally
    the double category of coalgebras associated to the restricted
    \awfs on $\C_{\lambda}$.}\end{footnote} as the full sub-double
category of $\DCoalg{L}$ with as objects, isomorphism-class
representatives of the $\lambda$-presentable objects of $\C$. This
category is small, and the inclusion $j_0 \colon \C_\lambda \to \C$ is
dense; we claim that $j_{1}$ and $j_{2}$ are too. For $j_{1}$, we
consider the full dense subcategory $(\Coalg{L})_{\lambda}$ of
$\lambda$-presentable objects in $\Coalg{L}$. Because $d_\mathsf L$
and $c_\mathsf L$ preserve $\lambda$-presentability it follows that
each $\lambda$-presentable of $\Coalg{L}$ has $\lambda$-presentable
domain and codomain; so the fully faithful dense functor
$(\Coalg{L})_{\lambda} \to \Coalg{L}$ factors through the full
inclusion $j_{1} \colon \Coalg{L_{\lambda}} \to \Coalg{L}$, whence, by
\cite[Theorem~5.13]{Kelly1982Basic} $j_{1}$ is dense. The density of
$j_{2}$ follows in the same manner, by considering the full dense
subcategory $(\Coalg{L}_{2})_{\lambda}$ of $\Coalg{L}_{2}$ and arguing
that it is contained in $(\Coalg{L_{\lambda}})_{2}$.

Since both $j_{1}$ and $j_{2}$ are dense, and since all of the
functors to their right are cocontinuous it follows from
Lemma~\ref{lem:3} that both $\Rl[j_{1}]$ and $\Rl[j_2]$ are
isomorphisms of categories; whence $\Drl[j_1]$ and $\Drl[j_2]$ as
displayed in the right two columns of
\[
\cd[@-0.6em@C+0.5em]{ \Drl[\DCoalg{L}] \ar[r] \ar[d]_{\Drl[j]} &
  \Drl[\Coalg{L}] \ar[d]^{\Drl[j_1]} \ar@<3pt>[r]^{\Drl[m]}
  \ar@<-3pt>[r]_{\delta}
  & \Drl[(\Coalg{L}_2)] \ar[d]_{\Drl[j_2]}\\
  \Drl[\DCoalg{L_\lambda}] \ar[r] & \Drl[\Coalg{L_\lambda}]
  \ar@<3pt>[r]^{\Drl[m]} \ar@<-3pt>[r]_{\delta} &
  \Drl[(\Coalg{L_\lambda}_2)] }
\]
are isomorphisms of double categories. But both rows are, as
in~\eqref{eq:equaliser}, equalisers, and it follows that the induced
double functor $\Drl[j]$ on the left is invertible as claimed.
\end{proof}

\section{Fibre squares and enriched cofibrant generation}
\label{sec:enrich-small-object}
We conclude this paper with a further batch of examples based on a
surprisingly powerful modification of the small object argument of
Proposition~\ref{prop:6}. Given a well-behaved monoidal category $\V$
equipped with an \awfs $(\mathsf L, \mathsf R)$, a well-behaved
$\V$-category $\C$, and a small category $U \colon \J \to
(\C_0)^\atwo$ over $(\C_0)^\atwo$, it constructs the ``free $(\mathsf
L, \mathsf R)$-enriched \awfs cofibrantly generated by $\J$''. This is
an \awfs on $\C_0$ whose algebras $\alg g \colon C \to D$ are maps
equipped with, for every $Uj \colon A \to B$ in the image of $U$, a
choice of $\mathsf R$-map structure on the canonical comparison map
$\C(B,C) \to \C(B,D) \times_{\C(A,D)} \C(A,C)$ in $\V$, naturally in
$j$.

The ``enriched small object argument'' which builds \awfs of this kind
is most properly understood in the context of \emph{monoidal algebraic
  weak factorisation systems}~\cite{Riehl2013Monoidal}; and as indicated
in the introduction, a comprehensive treatment along those lines must
await a further paper. Here, we content ourselves with giving the
construction and a range of applications.

\subsection{Fibre squares}
\label{sec:fibre-squares-1}
The cleanest approach to enriched cofibrant generation makes use of
the following \emph{fibre square} construction. It associates to any
\awfs $(\mathsf L, \mathsf R)$ on a category $\C$ with pullbacks an
\awfs on the arrow category $\C^\atwo$ whose algebras are the
\emph{$\mathsf R$-fibre squares}: maps $(u,v) \colon f \to g$ of
$\C^\atwo$ as to the left of
\[
\cd[@-1.4em@C-0.25em]{
A \ar@/^1em/[drrr]^{u} \ar@/_1em/[dddr]_{f}
\ar@{.>}[dr]|-{k} \\ 
& B \times_D C \ar[rr] \ar[dd] \pushoutcorner & & C
\ar[dd]^{g} \\ \\
& B \ar[rr]_-{v} & & D
}\qquad \qquad
\cd[@+0.2em]{
  \cat{Fibre}_\mathsf R(\C) \ar[r] \ar[d]_V \pushoutcorner & \Alg{R}
  \ar[d]^{U^\mathsf R} \\
  (\C^{\atwo})^{\atwo} \ar[r]_-{(u,v) \mapsto k} & \C^\atwo}
\]
equipped with an $\mathsf R$-algebra structure on the comparison map
$k \colon A \to B \times_D C$. The $\mathsf R$-fibre squares are the
objects of a category $\cat{Fibre}_\mathsf R(\C)$, whose morphisms are
determined by way of the pullback square on the right above.

We may compose $\mathsf R$-fibre squares: given $(u,v) \colon f \to g$
with $\mathsf R$-algebra structure $\alg k \colon A \to B \times_D C$
and $(s,t) \colon g \to h$ with structure $\alg \ell \colon C \to D
\times_F E$, we can pull back $\alg \ell$ along $v \times_F E$ and use
Proposition~\ref{prop:overcodomain2} to obtain an algebra structure
$v^\ast \alg \ell \colon B \times_D C \to B \times_F E$; the composite
algebra $v^\ast \alg \ell \cdot \alg k \colon A \to B \times_D C \to B
\times_F E$ now equips $(s,t) \cdot (u,v) \colon f \to h$ with
$\mathsf R$-fibre square structure. In this way we obtain a double
category $\dcat{Fibre}_\mathsf R(\C) \to \Sq{\C^\atwo}$ which is
concrete over $\C^\atwo$ and easily seen to be right-connected; so it
will be the double category of algebras for an \awfs on $\C^\atwo$ as
soon as $V \colon \cat{Fibre}_\mathsf R(\C) \to (\C^\atwo)^\atwo$ is
strictly monadic. As a pullback of the strictly monadic $U^\mathsf R
\colon \Alg{R} \to \C^\atwo$, this $V$ will always create $V$-absolute
coequaliser pairs; so it suffices to show that it has a left adjoint.
Given a square in $\C$ as on the left below, we factor the induced $k
\colon A \to B \times_D C$ as $\r k \cdot \l k \colon A \to Ek \to B
\times_D C$, yielding the factorisation
\begin{equation}
\cd{
  A \ar[r]^u \ar[d]_f & C \ar[d]^g \ar@{}[drr]|{=} & &
A \ar[d]_{f} \ar[r]^-{\l k} & Ek \ar[r]^-{\pi_2 \cdot \r k}
\ar[d]_{\pi_1 \cdot \r k} & C \ar[d]^g \\
B \ar[r]_v & D & & 
  B \ar[r]_1 & B \ar[r]_v & D}\label{eq:8}
\end{equation}
on the right. The rightmost square, when equipped with the algebra
structure $\fr k \colon Ek \to B \times_D C$ on the comparison map,
now comprises the free $\mathsf R$-fibre square on $(u,v)$. This shows
that $V$ has a left adjoint, whence by Theorem~\ref{thm:recognition},
$\dcat{Fibre}_\mathsf R(\C) \to \Sq{\C^\atwo}$ is the double category
of algebras for an \awfs on $\C^\atwo$; it is not hard to calculate
that the corresponding coalgebras are maps $(u,v) \colon f \to g$ for
which $u$ is equipped with $\mathsf L$-coalgebra structure and $v$ is
invertible. Note that if
$(\mathsf L, \mathsf R)$ is an accessible \awfs on a locally
presentable category, then the fibre square \awfs on $\C^\atwo$ is
also accessible, since its factorisations~\eqref{eq:8} are constructed
using the accessible $L$ and $R$ and finite limits in $\C$.

\begin{Ex}\label{ex:3}
  When $\C = \cat{Set}$ and $(\mathsf L, \mathsf R)$ is the split epi
  \awfs of Section~\ref{sec:split-epimorphisms}, the $\mathsf R$-fibre
  squares are the \emph{algebraic weak pullbacks}: squares
  $(u,v) \colon f \rightarrow g$ as to the left of~\eqref{eq:8} such
  that, for every $b \in B$ and $c \in C$ with $vb = gc$, there is
  provided a choice of $a \in A$ with $fa = b$ and $ua = c$. Now, the
  Yoneda embedding $y \colon \atwo^\op \rightarrow \cat{Set}^\atwo$
  sends the unique arrow $0 \rightarrow 1$ of $\atwo$ to the
  unique arrow $y(1) \rightarrow y(0)$ in $\cat{Set}^\atwo$:
  \begin{equation*}
    \cd{
      {0} \ar[r]^-{!} \ar[d]_{y(1) = !} &
      {1} \ar[d]^{1 = y(0)} \\
      {1} \ar[r]_-{1} &
      {1}\rlap{ ,}
    }
  \end{equation*}
  and to equip $(u,v) \colon f \rightarrow g$ with a lifting structure
  against $y(1) \rightarrow y(0)$ is precisely to make it into an
  algebraic weak pullback. In fact, it is not hard to show that the
  algebraic weak pullback \awfs on $\cat{Set}^\atwo$ is cofibrantly
  generated in the sense of Section~\ref{sec:cofibr-fibr-gener} by the
  single arrow $y(1) \rightarrow y(0)$. Moreover, as a cocomplete
  category, $\cat{Set}^\atwo$ is freely generated by this arrow;
  combining the preceding two facts with Proposition~\ref{prop:15}, we
  conclude that the algebraic weak pullback \awfs on $\cat{Set}$ is in
  fact the \emph{free cocomplete \awfs generated by an
    $\mathsf L$-map}. By this we mean that, for any \awfs
  $(\mathsf L, \mathsf R)$ on a cocomplete category $\C$, there is an
  equivalence of categories
  \begin{equation*}
    \Ladj(\, (\cat{Set}^\atwo, \mathsf L_{\mathrm{wkpb}}, \mathsf R_{\mathsf{wkpb}}),\, (\C, \mathsf L, \mathsf R) ) \simeq \Coalg{L}\rlap{ ,}
  \end{equation*}
  pseudonatural in $(\C, \mathsf L, \mathsf R)$. By a similar argument
  we see that the split epi \awfs on $\cat{Set}$ is the \emph{free
    cocomplete \awfs generated by a cofibrant object}; for a different
  proof of this fact, see~\cite[Corollary~15]{Bourke2014AWFS2}.
\end{Ex}

\subsection{Enriched cofibrant generation}
\label{sec:enrich-cofibr-gener}
Suppose now that $\V$ is a locally presentable symmetric monoidal
closed category: thus a suitable base for enriched category theory.
Let $(\mathsf L, \mathsf R)$ be an accessible \awfs on $\V$, and let
$\C$ be a $\V$-category whose underlying category $\C_0$ is locally
presentable; finally, let $U \colon \J \to (\C_0)^\atwo$ be a small
category over $(\C_0)^\atwo$. By applying its monad and comonad
pointwise, we may lift the \awfs $(\mathsf L, \mathsf R)$ on $\V$ to
an accessible \awfs on $[\J^\op, \V]$; now applying the fibre square
construction of the preceding section to this yields an accessible
\awfs on $[\J^\op, \V]^\atwo$. There is a right adjoint functor
$\tilde U \colon \C_0 \to [\J^\op, \V]^\atwo$ sending $x \in \C_0$ to
the arrow $\C(\mathrm{cod}\ U\thg, x) \to \C(\mathrm{dom}\ U\thg, x)$
of $[\J^\op, \V]$; and we now define the \emph{enriched \awfs
  cofibrantly generated by $\J$} to be the projective lifting of the
fibre square \awfs on $[\J^\op, \V]^\atwo$ along $\tilde U$; note that
the existence of this lifting is guaranteed
by Proposition~\ref{prop:2}(a). The double category of algebras of the
\awfs $(\mathsf L_\J, \mathsf R_\J)$ so obtained is thus defined by a
pullback square
\[
\cd{
\DAlg{\mathsf R_\J} \pushoutcorner \ar[r] \ar[d] &
\dcat{Fibre}_{[\J^\op, \mathsf R]}([\J^\op, \V]) \ar[d] \\
\Sq{\C_0} \ar[r]_-{\Sq{\tilde U}} & \Sq{[\J^\op, \V]^\atwo}\rlap{ ;}
}
\]
so that in particular, an algebra structure on a map $g \colon C \to
D$ is given by the choice, for each $j \in \J$ with image $Uj \colon A
\to B$ in $\C^\atwo$, of an $\mathsf R$-algebra structure on the
comparison map $\C(B,C) \to \C(A,C) \times_{\C(A,D)} \C(B,D)$,
naturally in $j$; we may say that $g$ is equipped with an
\emph{enriched lifting operation} against $\J$.

\subsection{Examples}
\label{sec:examples}
The remainder of the paper will be given over to a range of examples
of enriched cofibrant generation. All these examples will in fact be
generated by a mere \emph{set} of maps $J$, seen as a discrete
subcategory $J \hookrightarrow (\C_0)^\atwo$ of a locally presentable $\C_0$.

\begin{Exs} 
  \begin{enumerate}[(i)]
  \item Let $\V = \Set$ and $(\mathsf L, \mathsf R)$ be the \awfs for
    split epis thereon; then the notion of enriched cofibrant
    generation reduces to the unenriched one of
    Section~\ref{sec:cofibr-fibr-gener}. A full treatment of enriched
    cofibrant generation would in fact generalise each aspect of the
    theory developed in Sections~\ref{sec:cofibrant-generation-1}
    and~\ref{sec:cofibr-gener-double}; but as we have said above, this
    must await another paper. 
\vskip0.25\baselineskip

  \item More generally, let $\V$ be any locally presentable symmetric
    monoidal closed category, and let $(\mathsf L, \mathsf R)$ be the
    split epi \awfs thereon. The enriched \awfs cofibrantly generated
    by $J \hookrightarrow \C^\atwo$ is precisely that obtained by the
    ``enriched small object argument''
    of~\cite[Proposition~13.4.2]{Riehl2014Categorical}.\vskip0.25\baselineskip

  \item Let $\V = \Set$ and let $(\mathsf L, \mathsf R)$ be the
    initial \awfs thereon; its algebra category is the full
    subcategory of $\Set^\atwo$ on the isomorphisms. In this
    case, the algebra category of the enriched \awfs $(\mathsf L_J,
    \mathsf R_J)$ generated by $J \hookrightarrow \C^\atwo$ may be
    identified with the full subcategory of $\C^\atwo$ on those maps
    with the \emph{unique} right lifting property against each map in
    $J$. In particular, since $\Alg{R_\mathnormal{J}} \to \C^\atwo$ is
    fully faithful, Proposition~\ref{prop:8} applies, so that
    $(\mathsf L_J, \mathsf R_J)$ in fact describes the
    \emph{orthogonal} factorisation system whose left class is
    generated by $J$~\cite[\S2.2]{Freyd1972Categories}.
  \end{enumerate}\label{ex:2}
\end{Exs}

\begin{Exs}
  For our next examples of enriched cofibrant generation, we take $\V
  = \cat{Cat}$ equipped with the \awfs $(\mathsf L, \mathsf R)$ for
  retract equivalences of Section~\ref{sec:lalis}.
  \begin{enumerate}[(i)]
  \item Take $\C = \Cat$ and let $J$ comprise the single functor $!
    \colon \cat 0 \to \cat 1$. It's easy to see that the enriched
    \awfs generated by $J$ is again that for retract
    equivalences.\vskip0.25\baselineskip
  \item Take $\C = \Cat$ and let $J$ comprise the single functor $\top
    \colon \cat 1 \to \atwo$ picking out the terminal object of
    $\atwo$. An algebra for the enriched \awfs generated by $J$ is a
    functor $G \colon C \to D$ such that the induced $C^\atwo \to D
    \downarrow G$ bears retract equivalence structure, or
    equivalently, is fully faithful and equipped with a section on
    objects. Full fidelity corresponds to the requirement that every
    arrow of $C$ be cartesian over $D$, whereupon a section on objects
    amounts to a choice of cartesian liftings: so an algebra is a
    cloven fibration whose fibres are groupoids. We find further that
    maps of algebras are squares strictly preserving the cleavage, and
    that composition of algebras is the usual composition of
    fibrations. \vskip0.25\baselineskip
  \item Let $(\C, j)$ be a small site, let $\E = \cat{Hom}(\C^\op,
    \Cat)$ be the $2$-category of pseudofunctors, strong
    transformations and modifications $\C^\op \to \Cat$, and let $y
    \colon \C \to \E$ be the Yoneda embedding. If $(U_i \to U : i \in
    I)$ is a covering family in $\C$, then the associated
    \emph{covering $2$-sieve} $f \colon \phi \to yU$ in $\E$ is the
    second half of the pointwise (bijective on objects, fully
    faithful) factorisation $\sum_i yU_i \to \phi \to yU$; recall that
    an object $X \in \E$ is called a \emph{stack} if the restriction
    functor $\E(f, X) \colon \E(yU, X) \to \E(\phi, X)$ is an
    equivalence of categories for each covering $2$-sieve.

    \newcommand{\TAlgs}{\mathsf{T}\text-\cat{Alg}_{s}}
    \newcommand{\TAlg}{\mathsf{T}\text-\cat{Alg}} Consider now the
    $2$-category $\E_s = \cat{Hom}(\C^\op, \cat{Cat})_s$ of
    pseudofunctors, $2$-natural transformations and modifications. By
    the general theory of~\cite{Blackwell1989Two-dimensional}, there
    is an accessible $2$-monad $\mathsf T$ on $\cat{Cat}^\mathrm{\ob
      \C}$ whose $2$-category $\TAlgs$ of strict algebras and strict
    algebra morphisms is $\E_s$---so that in particular, $\E_s$ is
    locally presentable---and whose $2$-category $\TAlg$ of strict
    algebras and pseudomorphisms is $\E$. It follows
    by~\cite[Theorem~3.13]{Blackwell1989Two-dimensional} that the
    inclusion $2$-functor $\E_s \to \E$ has a left adjoint $Q \colon
    \E_s \to \E$. Thus we may identify stacks with pseudofunctors $X$
    such that the restriction functor $\E_s(Qf, X) \colon \E_s(QyU, X)
    \to \E_s(Q\phi, X)$ is an equivalence for each covering $2$-sieve $f$.

    As $\E_s$ is locally presentable, we may consider the enriched
    \awfs thereon generated by the set $\{Qf : \text{$f$ a covering
      $2$-sieve in $\E$}\}$, and by the above, we see that its
    algebraically fibrant objects are a slight strengthening of the
    notion of stack; namely, they are pseudofunctors $X \colon \C^\op
    \to \cat{Cat}$ such that each $\E(f, X) \colon \E(yU, X) \to
    \E(\phi, X)$ is given the structure of a \emph{retract}
    equivalence. \vskip0.25\baselineskip
  \item We may make arbitrary stacks into the algebraically fibrant
    objects of an \awfs by way of the following observations. Given $F
    \colon C \to D$ in $\Cat$, we may form its pseudolimit $D
    \downarrow_{\cong} F$---the category whose objects are triples $(d
    \in D, c \in C, \theta \colon d \cong Fc)$---and choices of
    equivalence section for $D \downarrow_{\cong} F \to D$ now
    correspond to choices of equivalence pseudoinverse for $F$. It
    follows that if $f \colon X \to Y$ is a map in a finitely
    cocomplete $2$-category $\E$, and $\bar f \colon X \to \bar Y$ is
    the injection of $X$ into the pseudocolimit of $f$, then for each
    $F \in \E$, retract equivalence structures on $\E(\bar f, F)$
    correspond with equivalence structures on $\E(f, F)$. 

    In particular, if $(\C, j)$ is a small site and $\E$ and $\E_s$
    are as above, then we may consider the enriched \awfs on $\E_s$
    cofibrantly generated by the set $\{\overline{Qf} : f \text{ a
      covering $2$-sieve in $\E$}\}$. Its algebraically fibrant
    objects are now stacks in the usual sense; to be
    precise, they are pseudofunctors $X$ such that each $\E(yU, X) \to
    \E(\phi, X)$ is provided with a chosen equivalence pseudoinverse.
  \end{enumerate}\label{ex:4}
\end{Exs}
\begin{Exs}
  Our next collection of examples again take $\V = \cat{Cat}$, but now with
  $(\mathsf L, \mathsf R)$ the \awfs for lalis of
  Section~\ref{sec:lalis}. 
  \begin{enumerate}[(i)]
  \item Let $\C = \Cat$ and let $J$ comprise the single functor $!
    \colon \cat 0 \to \cat 1$; then the enriched \awfs generated by
    $J$ is simply that for lalis.\vskip0.25\baselineskip
  \item Let $\C = \Cat$ and let $J$ comprise the single functor $\top
    \colon \cat 1 \to \atwo$ picking out the terminal object of
    $\atwo$. To equip $G \colon C \to D$ with algebra structure for
    the enriched \awfs this generates is to equip the induced $C^\atwo
    \to D \downarrow G$ with a right adjoint right inverse; which
    by~\cite{Gray1966Fibred}, is equally to equip $G$ with the
    structure of a cloven fibration. As before, maps of algebras are
    squares strictly preserving the cleavage; and composition of
    algebras is the usual composition of fibrations.
    \vskip0.25\baselineskip
    \vskip0.25\baselineskip
  \item Generalising the previous two examples, let $\Phi$ be some set
    of small categories, and let $J$ comprise the class of functors
    $\{A \to A_\bot : A \in \Phi\}$ obtained by freely adjoining an
    initial object to each $A$. To equip $G \colon C \to D$ with
    algebra structure for the induced \awfs is to give, for each $A
    \in \Phi$, a right adjoint section for the comparison functor
    $[A_\bot, C] \to [A_\bot, D] \times_{[A,D]} [A,C]$. Since a
    functor out of $A_\bot$ is the same thing as a functor out of $A$
    together with a cone over it, a short calculation shows that this
    amounts to giving, for each $A \in \Phi$, each diagram $F \colon A
    \to C$, and each cone $p \colon \Delta V \to GF$ in $D$, a cone $q
    \colon \Delta W \to F$ in $C$ with $Gq = p$, such that for any
    cone $q' \colon \Delta W' \to F$, each factorisation of $Gq'$
    through $p$ via a map $k \colon GW' \to V$ lifts to a unique
    factorisation of $q'$ through $q$.\footnote{Note that when $A$ is a
    discrete category, this is precisely the notion of
    \emph{$G$-initial
      lifting}~\cite[Definition~10.57]{Adamek1990Abstract} of a
    discrete cone.} In particular, an algebraically fibrant object for
    this \awfs is a category equipped with a choice of limits for all
    diagrams indexed by categories in $\Phi$. \vskip0.25\baselineskip
  \item Generalising the preceding example, let $\C$ be any locally
    presentable $2$-category. For any arrow $f \colon A \to B$ of
    $\C$, let $\bar f \colon A \to \bar B$ denote the injection into
    the lax colimit of the arrow $f$, and for any set $J$ of arrows in
    $\C$, let $\bar J = \{\bar f : f \in J\}$. By using the universal
    property of the colimit, we may calculate that an algebra
    structure on a morphism $g \colon C \to D$ for the \awfs induced
    by $\bar J$ is given by the choice, for every $f \colon A \to B$
    in $J$ and every $2$-cell $\theta$ as on the left below, of a map
    $k$ and $2$-cell $\gamma$ as in the centre with $gk = v$ and $g
    \gamma = \theta$, and such that $\gamma$ is \emph{initial} over
    $g$; thus, for every diagram as on the right with $g \alpha =
    \theta \cdot \beta f$, there exists a unique $2$-cell $\delta
    \colon \ell \Rightarrow k$ with $g \delta = \beta$ and $\alpha =
    \gamma \cdot \delta f$.
    \begin{equation*}
      \cd{
        A \ar[d]_f \ar[r]^u \utwocell{dr}{\theta} & C \ar[d]^g \\
        B \ar[r]_{v} & D
      } \qquad \qquad
      \cd{
        A \ar[d]_f \ar[r]^u \utwocell[0.3]{dr}{\gamma} & C \ar[d]^g \\
        B \ar[r]_{v} \ar[ur]_k & D
      } \qquad \qquad
      \cd[@+0.5em]{
        A \ar[d]_f \ar[r]^u \utwocell[0.3]{dr}{\alpha} & C \ar[d]^g \\
        B \ar[r]_v \ar[ur]|{\ell} & \dtwocell[0.3]{ul}{\beta} D\rlap{ .}
      }
    \end{equation*}
    In particular, the algebraically fibrant objects for this \awfs
    are those $C \in \C$ such that, for each $f \colon A \to B$ in
    $J$, the functor $\C(f, C) \colon \C(B,C) \to \C(A,C)$ is equipped
    with a right adjoint; that is, such that each morphism $A \to C$
    in $\C$ is equipped with a chosen right Kan extension along $f$.
  \end{enumerate}\label{ex:5}
\end{Exs}

  Of course, we may construct examples dual to the preceding
  ones---dealing with ralis, opfibrations, final liftings of cocones,
  categories with colimits, and left Kan extensions---by replacing the
  \awfs for lalis in the above by that for ralis.
  
  \begin{Ex}
    Let $\V$ be the category $\cat{SSet} = [\Delta^\op, \Set]$ of
    simplicial sets, equipped with the \awfs for \emph{trivial Kan
      fibrations}, cofibrantly generated by the set of boundary
    inclusions $\{\partial \Delta[n] \to \Delta[n] : n \in \mathbb
    N\}$. Given a set $J$ of maps in the locally presentable
    $\cat{SSet}$-category $\C$, the algebraically fibrant objects of
    the enriched \awfs cofibrantly generated by $J$ are those $X \in
    \C$ such that, for each $f \colon A \to B$ in $J$, the induced map
    $\C(B, X) \to \C(A,X)$ bears algebraic trivial fibration
    structure. Most typically, one would consider this in the
    situation where $\C$ is a simplicial model category and the $J$'s
    are a class of cofibrations; then the algebraically fibrant
    objects above which are also fibrant for the underlying model
    structure are precisely the \emph{$J$-local
      objects}~\cite[Definition~3.1.4]{Hirschhorn2003Model}. The
    enriched small object argument in this particular case was
    described in~\cite[\S7]{Bousfield1977Constructions}.\label{ex:6}
\end{Ex}

Our final examples bear on the theory of
\emph{quasicategories}~\cite{Joyal2002Quasi-categories}.
Quasicategories are simplicial sets with fillers for all inner horns;
they are a particular model for $(\infty, 1)$-categories---weak higher
categories whose cells of dimension $>1$ are weakly invertible---and
have a comprehensive theory~\cite{Joyal2008The-theory,Lurie2009Higher}
paralleling the classical theory of categories. In particular, there
are good notions of \emph{Grothendieck fibration}, \emph{limit}, and
\emph{Kan extension} for quasicategories; and by analogy with
Examples~\ref{ex:4}, we may hope to obtain these notions as fibrations
or fibrant objects for \awfs constructed by enriched cofibrant
generation, so long as we can find a good quasicategorical analogue of
the \awfs for lalis.

For this we propose the \awfs of Examples~\ref{ex:8}\eqref{item:1}
above, whose $\mathsf R$-algebras we called the \emph{simplicial
  lalis}. By using the explicit formulae of
Section~\ref{sec:stable-class-monom}, we may calculate the monad
$\mathsf R$ at issue, and so obtain the following concrete description
of its algebras. First some notation: given a simplicial set $X$, we
write $x \colon a \rightsquigarrow b$ to denote a simplex $x \in
X_{n+1}$ whose $(n+1)$st face is $a$ and whose final vertex is $b$.
Now by a \emph{simplicial lali}, we mean a simplicial map $f \colon A
\to B$ together with:
\begin{itemize}
\item A section $u \colon B_0 \to A_0$ of the action of $f$ on
  $0$-simplices;
\item For each $a \in A_n$ and $x \colon fa \rightsquigarrow b$ in
  $B$, a simplex $\gamma_a(x) \colon a \rightsquigarrow ub$ over $x$;
\end{itemize}
subject to the coherence conditions that:
\begin{itemize}
\item $\gamma_a(x) \cdot \delta_i = \gamma_{a \cdot \delta_i}(x \cdot
  \delta_i)$ and $\gamma_a(x) \cdot \sigma_i = \gamma_{a \cdot
    \sigma_i}(x \cdot \sigma_i)$;
\item $\gamma_{ub}(b \cdot \sigma_0) = ub \cdot \sigma_0$; and
\item $\gamma_{\gamma_a(x)}(x \cdot \sigma_{n+1}) = \gamma_a(x) \cdot \sigma_{n+1}$.
\end{itemize}

(For the reader who is familiar with Joyal's \emph{slice construction}
on simplicial sets~\cite[\S 3]{Joyal2002Quasi-categories}, we observe
that giving the data of $\gamma$ together with the first coherence
condition amounts to giving a family of sections
$f/b \rightarrow A/ub$ for the maps
$f/ub \colon A/ub \rightarrow f/b$.)

By comparing with Proposition~\ref{prop:11}, we see that if $A = N\C$
and $B = N\D$ are the nerves of small categories, then a simplicial
lali $A \to B$ is the same thing as a lali $\C \to \D$. This, of
course, does not in itself justify the notion of simplicial lali;
in order to do so, we will relate it to a notion introduced by Riehl
and Verity in~\cite[Example~4.4.7]{Verity2013The-2-category}, which
for the present purposes we shall refer to as a \emph{quasicategorical
  lali}. A simplicial map $f \colon A \to B$ is a quasicategorical
lali if equipped with a strict section $u \colon B \to A$, a
simplicial homotopy $\eta \colon 1 \Rightarrow uf$ satisfying $f \eta
= 1_f$ (on the nose), and a $2$-homotopy
\[
\cd{
& u \ar[dr]^{1_u} \ar@{}[d]|(0.6){\theta} \\ u \ar[rr]_{1_u}
\ar[ur]^{\eta u} & & u
}
\]
satisfying $f \theta = 1_{1_B}$. We have established the following result;
its proof is beyond the scope of this paper, but note that it
is a quasicategorical analogue of the correspondence between the two
views of ordinary lalis described in Proposition~\ref{prop:11}.
\begin{Prop}
  \label{prop:13}
  Let $f \colon A \to B$ be a map of simplicial sets. If $f$ is a
  simplicial lali, then it is a quasicategorical lali; if it is a
  quasicategorical lali and an inner Kan fibration, then it is a
  simplicial lali.
\end{Prop}
Since Riehl and Verity are able to use quasicategorical lalis to
describe various aspects of the theory of quasicategories, the above
proposition allows us to describe these same aspects using simplicial
lalis, and, hence, using the theory of enriched \awfs.


\begin{Exs}
  \label{ex:7}
  We consider enriched cofibrant generation in the case where $\V =
  \cat{SSet}$ and $(\mathsf L, \mathsf R)$ is the \awfs for simplicial
  lalis.
  \begin{enumerate}[(i)]
  \item \emph{Limits}. Let $A$ be a quasicategory, and let $A_\bot$
    denote the quasicategory $\Delta[0] \oplus A$ obtained by
    adjoining an initial vertex to $A$. Riehl and Verity show
    in~\cite[Corollary~5.2.19]{Verity2013The-2-category} that a
    quasicategory $X$ admits all limits of shape $A$---in the sense
    of~\cite[Definition~4.5]{Joyal2002Quasi-categories}---just when
    the simplicial map $X^{A_\bot} \to X^A$ bears quasicategorical
    lali structure. It follows that, among the quasicategories, those
    equipped with a choice of all limits of shape $A$ may be realised
    as the algebraically fibrant objects of an \awfs on $\cat{SSet}$;
    namely, the enriched \awfs cofibrantly generated by the single map
    $A \to A_\bot$. By enriched cofibrant generation with respect to
    the set of maps $J = \{A \to A_\bot : A \text{ a finite
      quasicategory}\}$ we may capture quasicategories with all finite
    limits as algebraically fibrant objects.\vskip0.25\baselineskip
  \item \emph{Grothendieck fibrations}. Riehl and Verity are in the
    process of analysing the quasicategorical Grothendieck
    fibration---the ``Cartesian fibrations''
    of~\cite[\S2.4]{Lurie2009Higher}---and have
    shown~\cite{Verity2014Talk} that an isofibration $g \colon C \to
    D$ of quasicategories bears such a structure just when the
    comparison functor $C^\atwo \to D \downarrow g$ bears
    quasicategorical lali structure. Here, $C^\atwo$ denotes the
    exponential $C^{\Delta[1]}$, while $D \downarrow g$ is the
    pullback of $g \colon C \to D$ along $D^{\delta_0} \colon D^{\Delta[1]}
    \to D$. Thus, among the isofibrations of quasicategories, the
    Grothendieck fibrations can be realised as the algebras of an
    \awfs on $\cat{SSet}$---namely, that obtained by enriched
    cofibrant generation with respect to the single map $\delta_0 \colon
    \Delta[0] \to \Delta[1]$.\vskip0.25\baselineskip
\item \emph{Right Kan extensions}.
  By~\cite[Example~5.0.4~and~Proposition~5.1.19]{Verity2013The-2-category},
  a morphism of quasicategories $f \colon C \to D$ admits a right
  adjoint if and only if the projection $f \downarrow D \to D$ admits
  simplicial lali structure. So suppose that $J$ is a set of morphisms between
  quasicategories; for each $f \colon A \to B$ in $J$, define $\bar B$ to be the
  the pushout of $f$ along $\delta_1 \times A \colon A \to
  \Delta[1] \times A$, and define $\bar f \colon A \to \bar B$ to be
  the composite of $\delta_0 \times A \colon A \to \Delta[1] \times A$ with
  the pushout injection $\Delta[1] \times A \to \bar B$. Then a
  quasicategory $X$ is an algebraically fibrant object for the
  enriched \awfs cofibrantly generated by $\{\bar f :
  f \in J\}$ just when each functor $X^f \colon X^B \to X^A$ has a
  right adjoint; that is, just when each functor $A \to X$ admits a
  right Kan extension along $f$.
\end{enumerate}
Of course, by considering the \awfs for simplicial \emph{ralis} in
place of simplicial lalis, we may capture notions such as colimits,
opfibrations and left Kan extensions.
\end{Exs}


\end{document}